\documentclass[11pt]{amsart}
\usepackage{amssymb}

\newtheorem{theo}{Theorem} 

\newtheorem{lemma}{Lemma}[section]
\newtheorem{prop}[lemma]{Proposition}
\newtheorem{corol}[lemma]{Corollary}
\newtheorem{claim}[lemma]{Claim}

\theoremstyle{remark}
\newtheorem{remark}[lemma]{Remark}

\theoremstyle{definition}

\newcommand{\dd}{\mathsf{d}}
\newcommand{\DD}{\mathsf{d}}

\newcommand{\CC}{\mathbb{C}}
\newcommand{\NN}{\mathbb{N}}
\newcommand{\RR}{\mathbb{R}}

\newcommand{\eps}{\varepsilon}

\newcommand{\DDD}{\mathcal{D}}

\newcommand{\MMM}{\mathcal{M}}

\newcommand{\SSS}{\mathcal{S}}

\newcommand{\xx}{X}
\newcommand{\YYY}{\mathcal{Y}}
\newcommand{\ZZZ}{\mathcal{Z}}

\newcommand{\tR}{\widetilde{R}}
\newcommand{\tf}{\widetilde{f}}
\newcommand{\ttau}{\tilde{\tau}}
\newcommand{\tpsi}{\widetilde{\psi}}

\newcommand{\tsigma}{\tilde{\sigma}}

\newcommand{\tgam}{\tilde{\gamma}}
\newcommand{\tw}{\tilde{w}}

\newcommand{\tW}{\widetilde{W}}

\newcommand{\aexp}{a}

\newcommand{\tlbda}{\widetilde{\lambda}}

\newcommand{\tx}{\widetilde{x}}

\newcommand{\hdot}{\dot{H}^1}
\newcommand{\Hdot}{\dot{H}^1(\RR^N)}

\DeclareMathOperator{\supp}{supp}

\DeclareMathOperator{\vect}{span}

\newcommand{\ds}{\displaystyle}

\numberwithin{equation}{section} 

\title[Dynamic for energy critical wave]{Dynamic of threshold solutions for energy-critical wave equation}
\author[T.~Duyckaerts]{Thomas Duyckaerts$^1$}
\email{thomas.duyckaerts@u-cergy.fr}
\address{Thomas Duyckaerts\\
Universit{\'e} de Cergy-Pontoise\\
D\'epartement de Math\'ematiques\\ 
Site de Saint Martin, 2 avenue Adolphe-Chauvin\\ 
95302 Cergy-Pontoise cedex, France. }
\author[F.~Merle]{Frank Merle$^2$}
\thanks{$^1$Cergy-Pontoise (UMR 8088)}
\thanks{$^2$Cergy-Pontoise, IHES, CNRS}
\thanks{This work was partially supported by the French ANR Grant ONDNONLIN}

\date{\today}

\begin{document}

\begin{abstract}
We consider the energy-critical non-linear focusing wave equation in dimension $N=3,4,5$. An explicit stationnary solution, $W$, of this equation is known.
 In \cite{KeMe06Pb}, the energy $E(W,0)$ has been shown to be a threshold for the dynamical behavior of solutions of the equation. In the present article we study the dynamics at the critical level $E(u_0,u_1)=E(W,0)$ and classify the corresponding solutions. We show in particular the existence of two special solutions, connecting different behaviors for negative and positive times. Our results are analoguous to \cite{DuMe07P}, which treats the energy-critical non-linear focusing radial Schr\"odinger equation, but without any radial assumption on the data. We also refine the understanding of the dynamical behavior of the special solutions.
\end{abstract}

\maketitle


\section{Introduction and main results}
We consider the focusing energy-critical wave equation on an interval $I$ ($0\in I$)
\begin{equation}
\label{CP}
\left\{ 
\begin{gathered}
\partial_t^2 u -\Delta u-|u|^{\frac{4}{N-2}}u=0,\quad (t,x)\in I\times \RR^N\\
u_{\restriction t=0}=u_0\in \hdot,\quad \partial_t u_{\restriction t=0}=u_1\in L^2.
\end{gathered}\right.
\end{equation}
where $u$ is real-valued, $N\in\{3,4,5\}$, and $\hdot:=\Hdot$.
The theory of the Cauchy problem for \eqref{CP} was developped in many papers (see \cite{Pecher84,GiSoVe92,LiSo95,ShSt94,ShSt98,Sogge95,Kapitanski94}). Namely, if $(u_0,u_1)\in \hdot\times L^2$, there exists an unique solution defined on a maximal interval $I=(-T_-(u),T_+(u))$ and the energy
$$ E(u(t),\partial_tu(t))=\frac{1}{2} \int |\partial_t u(t,x)|^2dx+\frac{1}{2} \int |\nabla u(t,x)|^2dx-\frac{1}{2^*}|u(t,x)|^{2^*}dx$$
is constant ($2^*:=\frac{2N}{N-2}$ is the critical exponent for the $H^1$-Sobolev embedding in $\RR^N$).

An explicit solution of \eqref{CP} is the stationnary solution in $\hdot$ (but in $L^2$ only if $N\geq 5$)
\begin{equation}
\label{defW}
W:=\frac{1}{\left(1+\frac{|x|^2}{N(N-2)}\right)^{\frac{N-2}{2}}}.
\end{equation}
The works of Aubin and Talenti \cite{Au76,Ta76}, give the following elliptic characterization of $W$
(throughout the paper we denote by $\|\cdot\|_p$ the $L^p$ norm on $\RR^N$)
\begin{gather}
\label{SobolevIn}
\forall u\in\hdot,\quad \|u\|_{2^*}\leq C_N\|\nabla u\|_{2}\\
\label{CarW}
\|u\|_{2*}=C_N\|\nabla u\|_{2}\Longrightarrow \exists \;\lambda_0>0,x_0\in\RR^N,\delta_0\in\{-1,+1\}\quad u(x)=\frac{\delta_0}{\lambda_0^{(N-2)/2}} W\Big(\frac{x+x_0}{\lambda_0}\Big),
\end{gather}
where $C_N$ is the best Sobolev constant in dimension $N$.\par
The dynamical behavior of some solutions of \eqref{CP} was recently described in \cite{KrSc05}, \cite{KrScTa07P} (in the radial three-dimensional case) and \cite{KeMe06Pb}. In \cite{KeMe06Pb}, Kenig and Merle has shown the important role of $W$, whose energy $E(W,0)=\frac{1}{NC_N^N}$ 
is an energy threshold for the dynamics in the following sense. Let $u$ be a solution of \eqref{CP}, not necessarily radial, such that
\begin{equation}
\label{hypsubcrit}
E(u_0,u_1)<E(W,0).
\end{equation}
Then 
\begin{itemize}
\item if $\|\nabla u_0\|_{2}<\|\nabla W\|_{2}$, we have
$I=\RR \text{ and } \|u\|_{L^{\frac{2(N+1)}{N-2}}_{t,x}}<\infty,$
which implies from the linear theory of \eqref{CP} that the solution scatters;
\item if $\|\nabla u_0\|^2_2>\|\nabla W\|_2^2$ then
$T_+<\infty \text{ and } T_-<\infty.$
\end{itemize}

Our goal (as is \cite{DuMe07P} for the nonlinear Schr\"odinger equation in the radial case)  is to give a classification of solutions of \eqref{CP}, not necessarily radial, with \textit{critical} energy, that is with initial condition $(u_0,u_1)\in \hdot\times L^2$ such that
\begin{equation*}
E(u_0,u_1)=E(W,0).
\end{equation*}
The stationnary solution $W$ belongs to this energy level, is globally defined and does not scatter.
Another example of special solutions is given by the following. 
\begin{theo}[Connecting orbits]
\label{th.exist}
Let $N\in\{3,4,5\}$. There exist radial solutions $W^-$ and $W^+$ of \eqref{CP}, with initial conditions $\left(W^{\pm}_0,W^{\pm}_1\right)\in \hdot\times L^2$ such that
\begin{gather}
\label{ex.energy}
E(W,0)=E(W^+_0,W^-_1)=E(W^-_0,W^{-}_1),\\
\label{ex.lim}
T_+(W^-)=T_+(W^+)=+\infty \text{ and }\lim_{t\rightarrow +\infty} W^{\pm}(t)=W \text{ in } \hdot,\\
\label{ex.sub}
\big\|\nabla W^{-}\big\|_{2}<\|\nabla W\|_{2},\quad  T_-(W^-)=+\infty,\quad \|W^-\|_{L^{\frac{2(N+1)}{N-2}}\left((-\infty,0)\times\RR^N\right)}<\infty,\\
\label{ex.super}
\big\|\nabla W^{+}\big\|_{2}>\|\nabla W\|_{2},\quad T_-(W^+)<+\infty.
\end{gather}
\end{theo}
\begin{remark}
Our construction gives a precise asymptotic development of $W^{\pm}$ near $t=+\infty$. Indeed there exists an eigenvalue $e_0>0$ of the linearized operator near $W$, such that, if $\YYY\in \SSS(\RR^N)$ is the corresponding eigenfunction with the appropriate normalization,
\begin{equation}
\label{as.dev}
\left\|\nabla\left(W^{\pm}(t)-W\pm e^{-e_0 t}\YYY\right)\right\|_{L^2}+\left\|\partial_t\left(W^{\pm}(t)-W\pm e^{-e_0 t}\YYY\right)\right\|_{L^2}\leq Ce^{-2 e_0t}.
\end{equation}
We refer to \eqref{defWk} and \eqref{CondWa2} for the development at all orders in $e^{-e_0t}$.
\end{remark}
\begin{remark}
Similar solutions were constructed for NLS in \cite{DuMe07P}. However, in the NLS case, we were not able to prove that $T_-(W^+)<\infty$ except in the case $N=5$. We see this fact, in particular in the case $N=3$, as a nontrivial result. Note that $W^+$ is not in $L^2$ except for $N=5$, so that case \eqref{theo.super} of Theorem \ref{th.classif} below does not apply.
\end{remark}
Our next result is that $W$, $W^-$ and $W^+$ are, up to the symmetry of the equation, the only examples of new behavior at the critical level.
\begin{theo}[Dynamical classification at the critical level]
\label{th.classif}
Let $N\in\{3,4,5\}$. Let $(u_0,u_1)\in \hdot\times L^2$ such that 
\begin{equation}
\label{threshold}
E(u_0,u_1)=E(W,0)=\frac{1}{NC_N^N}.
\end{equation}
Let $u$ be the solution of \eqref{CP} with initial conditions $(u_0,u_1)$ 
and $I$ its maximal interval of definition. Then the following holds:
\begin{enumerate}
\item \label{theo.sub} If $\ds \int |\nabla u_0|^2<\int |\nabla W|^2=\frac{1}{C_N^N}$ then $I=\RR$. Furthermore, either $u=W^-$ up to the symmetry of the equation, or $\|u\|_{L^{\frac{2(N+1)}{N-2}}_{t,x}}<\infty$.
\item \label{theo.crit} If $\ds \int |\nabla u_0|^2=\int |\nabla W|^2$ then $u=W$ up to the symmetry of the equation.
\item \label{theo.super} If $\ds \int |\nabla u_0|^2>\int |\nabla W|^2$, and $u_0\in L^2$ then either $u=W^+$ up to the symmetry of the equation, or $I$ is finite.
\end{enumerate}
\end{theo}
The constant $C_N$ is defined in \eqref{SobolevIn}. In the theorem, by \emph{$u$ equals $v$ up to the sym\-me\-try of the equation}, we mean that there exists $t_0\in\RR$, $x_0\in \RR^N$, $\lambda_0>0$, $\delta_0,\delta_1\in \{-1,+1\}$ such that
$$ u(t,x)=\frac{\delta_0}{\lambda_0^{(N-2)/2}}v\Big(\frac{t_0+\delta_1 t}{\lambda_0},\frac{x+x_0}{\lambda_0}\Big).$$
\begin{remark}
\label{RemPersist}
Case \eqref{theo.crit} is a direct consequence of the variational characterization of $W$ given by \eqref{CarW}. Furthermore, using assumption \eqref{threshold}, it shows (by continuity of $u$ in $\hdot$ and the conservation of energy) that the assumptions $ \ds \int |\nabla u(t_0)|^2<\int |\nabla W|^2, \; \ds \int |\nabla u(t_0)|^2>\int |\nabla W|^2$
do not depend on the choice of the initial time $t_0$. Of course, this dichotomy does not persist when $E(u_0,u_1)>E(W,0)$.
\end{remark}
\begin{remark}
Theorem \ref{th.classif} is also the analoguous, for the wave equation, of Theorem 2 of \cite{DuMe07P} for NLS, but without any radial assumption. The nonradial situation carries various problems partially solved in \cite{KeMe06Pb}, the major difficulty being a sharp control in time of the space localization of the energy. We conjecture that the NLS result also holds in the nonradial situation.
Note that case \eqref{theo.sub} implies ($W$ being radial) that any solution of \eqref{CP} satisfying \eqref{threshold} and whose initial condition is not radial up to a space-translation must scatter if $\int |\nabla u_0|^2<\int |\nabla W|^2$.
\end{remark}

\begin{remark}
In dimension $N=3$ or $N=4$, $W^+$ is not in $L^2$, and case \eqref{theo.super} means that any critical-energy solution such that $\int |\nabla u_0|^2>\int|\nabla W|^2$ and $u_0\in L^2$ blows-up for $t<0$ and $t>0$. It seems a delicate problem to get rid of the assumption $u_0\in L^2$.
\end{remark}

\begin{remark}
As a corollary, in dimension $N=5$, a dynamical characterization of $W$ is obtained. It is, up to the symmetry of the equation, the only $L^2$-solution such that $E(u_0,u_1)\leq E(W,0)$ that does not explode and does not scatter neither for positive nor negative time.
\end{remark}

The paper is organized as follows. In Section \ref{sec.compact} we recall previous results on the Cauchy Problem for \eqref{CP} and give preliminary properties of solutions of \eqref{CP} at the energy threshold such that $\int |\nabla u_0|^2<\int |\nabla W|^2$ and which do not scatter for positive times. These properties mainly follow from \cite{KeMe06Pb}. In Section \ref{sec.sub}, we show that these solutions converge exponentially to $W$ as $t\rightarrow +\infty$, which is the first step of the proof of Theorem \ref{th.classif} in case \eqref{theo.sub}. In Section \ref{sec.super}, we show the same result for energy-threshold solutions such that $\int |\nabla u_0|^2>\int |\nabla W|^2$, $u_0\in L^2$ and that are globally defined for positive time. In Section \ref{sec.lin}, we study the linearized equation around the solution $W$. Both theorems are proven in Section \ref{sec.proofs}. The main tool of the proofs is a fixed point giving the existence of the special solutions and, by the uniqueness property, the rigidity result in Theorem \ref{th.classif}.

\section{Preliminaries of subcritical threshold solutions}
\label{sec.compact}
\subsection{Quick review on the Cauchy problem}
We recall some results on the Cauchy Problem for \eqref{CP}. We refer to \cite[Section 2]{KeMe06Pb}, for a complete overview.
If $I$ is an interval, write
\begin{align}
\label{defS}
S(I)&:=L^{\frac{2(N+1)}{N-2}}(I\times \RR^N),\quad N(I):=L^{\frac{2(N+1)}{N+3}}(I\times\RR^N)\\
\label{defl}
\|u\|_{\ell(I)}&:=\|u\|_{S(I)}+\|D_x^{1/2} u\|_{L^{\frac{2(N+1)}{N-1}}(I\times \RR^N)}+\|\partial_t D_x^{-1/2}u\|_{L^{\frac{2(N+1)}{N-1}}(I\times \RR^N)}.
\end{align}
We first consider the free wave equation:
\begin{gather}
\label{LS1}
\partial_t^2 u-\Delta u=f,\quad t\in (0,T) \\
\label{LS2}
u_{\restriction t=0}=u_0,\; \partial_t u_{\restriction t=0}=u_1,
\end{gather}
where $D_x^{1/2} f \in N(0,T)$, $u_0\in \hdot$, $u_1\in L^2$. The solution of \eqref{LS1}, \eqref{LS2} is given by
$$ u(t,x)=\cos\big(t\sqrt{-\Delta}\big)u_0+\frac{\sin\big(t\sqrt{-\Delta}\big)}{\sqrt{-\Delta}} u_1+\int_0^t \frac{\sin\big((t-s)\sqrt{-\Delta}\big)}{\sqrt{-\Delta}}f(s) ds.$$
Then we have the following Strichartz estimates (see \cite{GiVe95,LiSo95}).
\begin{prop}
\label{prop.Strichartz}
Let $u$ and $f$ be as above. Then $u\in C^{0}(0,T;\hdot)$ and $\partial_t u\in C^0(0,T;L^2)$. Furthermore, for a constant $C>0$ independent of $T\in [0,\infty]$
\begin{equation}
\label{Strichartz}
\|u\|_{\ell(0,T)}+\sup_{t\in [0,T]} \|\nabla u(t)\|_2+\|\partial_t u(t)\|_2\leq C\left( \|\nabla u_0\|_{2}+\|u_1\|_2+\big\|D_x^{1/2}f\big\|_{N(0,T)}\right).
\end{equation}
Furthermore, if $D_x^{1/2} f\in N(T,+\infty)$, 
\begin{multline}
\label{dualStrichartz}
\forall t\geq 0,\quad \left\|\nabla\left(\int_T^{+\infty} \frac{\sin\big((t-s)\sqrt{-\Delta}\big)}{\sqrt{-\Delta}}f(s)ds\right)\right\|_2\\
+\left\|\partial_t\left(\int_T^{+\infty} \frac{\sin\big((t-s)\sqrt{-\Delta}\big)}{\sqrt{-\Delta}}f(s)ds\right)\right\|_2\leq C\big\|D_x^{1/2}f\big\|_{N(T,+\infty)}.
\end{multline}
\end{prop}

A solution of \eqref{CP} on an interval $I\ni 0$ is a function $u\in C^{0}(I,\hdot)$ such that $\partial_t u\in C^0(I,\hdot)$ and $u\in S(J)$ for all interval $J\Subset I$ and
$$ u(t,x)=\cos\big(t\sqrt{-\Delta}\big)u_0+\frac{\sin\big(t\sqrt{-\Delta}\big)}{\sqrt{-\Delta}} u_1+\int_0^t \frac{\sin\big((t-s)\sqrt{-\Delta}\big)}{\sqrt{-\Delta}}|u(s)|^{\frac{4}{N-2}} u(s) ds.$$

\begin{prop}(see \cite{Pecher84,GiSoVe92,ShSt94})
\label{exi.uni}
\begin{enumerate}
\item\emph{Existence.} If $u_0\in \hdot$, $u_1\in L^2$, there exists an interval $I\ni 0$ and a solution $u$ of \eqref{CP} on $I$ with initial conditions $(u_0,u_1)$.
\item\emph{Uniqueness.} If $u$ and $\tilde{u}$ are solutions of \eqref{CP} on an interval $I \ni 0$ such that $u(0)=\tilde{u}(0)$ and $\partial_t u(0)=\partial_t \tilde{u}(0)$, then $u=\tilde{u}$ on $I$.
\end{enumerate}
\end{prop}
According to Proposition \ref{exi.uni}, if $(u_0,u_1)\in \hdot\times L^2$, there exists a maximal open interval of definition for the solution $u$ of \eqref{CP}, that we will denote by $\left(-T_-(u),T_+(u)\right)$. The following holds 

\begin{prop}
\label{prop.criterion}
\begin{enumerate}
\item[]
\item\emph{Finite blow-up criterion.} If $T_+:=T_+(u)<\infty$ then 
$$ \|u\|_{S(0,T_+)}=\infty.$$
A similar result holds for negative times.
\item\label{continuity}\emph{Continuity.} Let $u$ be a solution of \eqref{CP} on an interval $I$ with initial condition $(u_0,u_1)\in \hdot\times L^2$. If $(u^k)$ is a sequence of solution of \eqref{CP} with initial conditions
$$(u_0^k,u_1^k)\underset{k \rightarrow +\infty}{\longrightarrow}(u_0,u_1) \text{ in }\hdot\times L^2$$ 
and $J\Subset (-T_-,T_+)$, then for large $k$, $J\subset \left(-T_-\big(u^k\big),T_+\big(u^k\big)\right)$, and
$$ (u^k,\partial_t u^k)\underset{k \rightarrow +\infty}{\longrightarrow}(u,\partial_t u) \text{ in }C^0(J,\hdot)\times C^0(J,L^2),\quad u^k  \underset{k \rightarrow +\infty}{\longrightarrow} u \text{ in }S(J).$$
\item\emph{Scattering.}
If $u$ is a solution of \eqref{CP} such that $T_+(u)=\infty$ and $\|u\|_{S(0,\infty)}<\infty$, then $u$ scatters.
\item\label{fsop}\emph{Finite speed of propagation.} (see \cite[Lemma 2.17]{KeMe06Pb})
There exist $\eps_0,\,C_0>0$, depending only on $\|\nabla u_0\|_{2}$ and $\|u_1\|_{2}$, such that if there exist $M,\,\eps>0$ satisfying $\eps<\eps_0$ and $\int_{|x|\geq M} |\nabla_x u_0|+\frac{1}{|x|^2}|u_0|^2+|u_0|^{2^*}+|u_1|^2\leq \eps$, then, 
$$ \forall t\in [0,T_+(u)), \quad \int_{|x|\geq \frac{3}{2}M+t} |\nabla u(t,x)|^2+\frac{1}{|x|^2}|u(t,x)|^2+|u|^{2^*}+|\partial_t u(t,x)|^2dx\leq C_0\eps.$$
\end{enumerate}
\end{prop}
\subsection{Properties of subcritical threshold solutions}
We are now interested in solutions of \eqref{CP} with maximal interval of definition $(T_-,T_+)$ and such that
\begin{gather}
\label{hyp.sub}
E(u_0,u_1)=E(W,0),\quad \|\nabla u_0\|_{2}<\|\nabla W\|_{2}\\
\label{hyp.noscatter}
\|u\|_{S(0,T_+)}=\infty.
\end{gather}
We start with the following claim (see \cite[Theorem 3.5]{KeMe06Pb}).
\begin{claim}[Energy Trapping]
\label{usefulclaim}
Let $u$ be a solution of \eqref{CP} satisfying \eqref{hyp.sub}. Then for all $t$ in the interval of existence $(-T_-,T_+)$ of $u$. 
\begin{equation}
\label{ineg.var}
\|\nabla u(t)\|_2^2+\frac{N}{2}\|\partial_t u\|_2^2\leq \|\nabla W\|_2^2.
\end{equation}
\end{claim}
\begin{proof}
Recall the following property which follows from a convexity argument
\begin{equation}
 \label{variationnal}
\forall v\in \hdot,\quad \|\nabla v\|_2^2\leq \|\nabla W\|_2^2\text{ and }E(v,0)\leq E(W,0)\Longrightarrow \frac{\|\nabla v\|_2^2}{\|\nabla W\|_2^2}\leq \frac{E(v,0)}{E(W,0)}.
\end{equation}
(See \cite[Claim 2.6]{DuMe07P}). Let $u$ be as in the claim. Note that by remark \ref{RemPersist} $\|\nabla u(t)\|_2< \|\nabla W\|_2$ for all $t$ in the domain of existence of $u$.
Now, according to \eqref{variationnal} and the fact that $E(u(t),\partial_t u(t))=E(W,0)$, 
\begin{equation*}
\frac{\|\nabla u(t)\|_2^2}{\|\nabla W\|_2^2}\leq \frac{E(u,\partial_t u)-\frac{1}{2}\|\partial_t u(t)\|_2^2}{E(W,0)}= \frac{E(W,0)-\frac{1}{2}\|\partial_t u(t)\|^2_2}{E(W,0)},
\end{equation*}
which yields \eqref{ineg.var}, recalling that  $N E(W,0)=\|\nabla W\|_2^2$.
\end{proof}

We recall now some key results from \cite{KeMe06Pb}. In their work, these results are shown for a critical element $u$, where assumptions \eqref{hyp.sub} are replaced by $E(u_0,u_1)<E(W,0)$, and $\|\nabla u_0\|_{2}<\|\nabla W\|_{2}$. Rather than recalling the proofs which are long and far from being trivial, we will briefly explain how they adapt in our case.
If $(f,g)$ is in $\hdot\times L^2$, we write
\begin{equation*}
(f,g)_{\lambda_0,x_0}(y)= \left(\frac{1}{\lambda_0^{\frac{N}{2}-1}}f\Big(\frac{y}{\lambda_0}+x_0\Big),\frac{1}{\lambda_0^\frac{N}{2}}g\Big(\frac{y}{\lambda_0}+x_0\Big)\right),\;
f_{\lambda_0,x_0}(y)=\frac{1}{\lambda_0^{\frac{N}{2}-1}}f\Big(\frac{y}{\lambda_0}+x_0\Big).
\end{equation*}

\begin{lemma}
\label{lem.compact}
Let $u$ be a solution of \eqref{CP} satistisfying \eqref{hyp.sub} and \eqref{hyp.noscatter}.
Then there exist continuous functions of $t$, $(\lambda(t),x(t))$ such that
$$K:=\left\{ \big(u(t),\partial_t u(t)\big)_{\lambda(t),x(t)},\; t\in [0,T_+)\right\}$$
has compact closure in $\hdot\times L^2$.
\end{lemma}
The proof of Lemma \ref{lem.compact}, which corresponds to Proposition 4.2 in \cite{KeMe06Pb}, is very close to the proof of Proposition 4.1 in \cite{KeMe06} and of Proposition 2.1 in \cite{DuMe07P}. The two main ingredients are the fact, proven in \cite{KeMe06Pb} that a solution of \eqref{CP} such that $E(u_0,u_1)<E(W,0)$ and $\|\nabla u_0\|_{2}<\|\nabla W\|_{2}$ is globally defined and scatters, a Lemma of concentration-compactness for solution to the linear wave equation due to Bahouri and G\'erard \cite{BaGe99}, and variational estimates as in Claim \ref{usefulclaim}. 

\begin{prop}[\cite{KeMe06Pb}]
\label{propKeMe.a}
Let $u$ be a solution of \eqref{CP} satisfying \eqref{hyp.sub} and \eqref{hyp.noscatter}.
Assume that there exist functions $(\lambda(t),x(t))$ such that
$$K:=\left\{ \big(u(t),\partial_t u(t)\big)_{\lambda(t),x(t)},\; t\in [0,T_+)\right\}$$
has compact closure in $\hdots\times L^2$ and that one of the following holds
\begin{enumerate}
\item \label{propKeMe.finite} $T_+<\infty$, \emph{or}
\item \label{propKeMe.infinite} $T_+=+\infty$ and there exists $\lambda_0>0$ such that $\forall t \in [0,+\infty)$, $\lambda(t)\geq \lambda_0$.
\end{enumerate}
Then $\ds \int u_1\nabla u_0=0$.
\end{prop}
\begin{proof}
If $\inf_{t\in (-T_-,T_+)} \|\partial_t u(t)\|_2^2=0$, then, using that $\|\nabla u(t)\|_2$ is bounded, and that $\int \partial_t u \nabla u(t)$ is conserved, we get immediately that $\int u_1\nabla u_0=0$.
Thus we may assume
\begin{equation}
\label{major.nabla}
\exists \delta_0>0, \; \forall t\in (-T_-,T_+),\quad \|\nabla u(t)\|^2_2\leq \|\nabla W\|^2_2-\delta_0.
\end{equation}
In this case, the proof is the same as in 
\cite[Propositions 4.10 and 4.11]{KeMe06Pb} which is shown under assumption \eqref{hyp.noscatter} and
\begin{equation}
\label{hyp.subsub}
\|\nabla u_0\|_2<\|\nabla W\|_2,\quad E(u_0,u_1)<E(W,0).
\end{equation}
This implies by variational estimates \eqref{major.nabla}, which is what is really needed in the proof of the proposition.
\end{proof}
\begin{prop}[\cite{KeMe06Pb}]
\label{propKeMe.b}
Let $u$ be a solution of \eqref{CP} satisfying \eqref{hyp.sub} and \eqref{hyp.noscatter}. Assume
\begin{equation}
\label{hyp.momentnul}
\int u_1\nabla u_0=0.
\end{equation}
Then $T_+=\infty$.
\end{prop}
This result is proven in \cite[Section 6]{KeMe06Pb} under the assumptions \eqref{hyp.noscatter} and \eqref{hyp.subsub}, but assumption \eqref{hyp.subsub} is only used to show that $\|\nabla u(t)\|_2$ is bounded, which is, in our case, a consequence of \eqref{hyp.sub} (see \cite[Remark 6.14]{KeMe06Pb}).

As a consequence of Proposition \ref{propKeMe.a} and \ref{propKeMe.b} we have, following again \cite{KeMe06Pb}:
\begin{prop}
\label{propsub}
Let $u$ be a solution of \eqref{CP} satisfying \eqref{hyp.sub} and \eqref{hyp.noscatter}. Let $\lambda(t)$, $x(t)$ given by Lemma \ref{lem.compact}. Then
\begin{enumerate}
\item \label{T+infty} $\ds T_+=\infty$.
\item \label{lambda.sub}
$\ds \lim_{t\rightarrow +\infty} t\lambda(t)=+\infty.$
\item \label{moment.sub}
$\ds \int_{\RR^N} u_1 \nabla u_0=0.$
\item \label{x.sub}
$\ds \lim_{t\rightarrow +\infty} \frac{x(t)}{t}=0.$
\end{enumerate}
\end{prop}
\begin{corol}
\label{corolGlobal}
Let $u$ be a solution of \eqref{CP} with maximal interval of definition $(-T_-,T_+)$ and such that $E(u_0,u_1)\leq E(W,0)$ and $\|\nabla u_0\|_2\leq \|\nabla W\|_2$. Then
$$ T_+=T_-=+\infty.$$
\end{corol}
\begin{proof}[Proof of Corollary \ref{corolGlobal}]
It is a consequence of \eqref{T+infty}. Indeed, if $\|\nabla u_0\|_2=\|\nabla W\|_2$, then by Claim \ref{usefulclaim}, $u_1=0$. Furthemore by \eqref{variationnal}, $E(W,0)=E(u_0,0)=E(u_0,u_1)$. Thus $\|u_0\|_{2^*}=\|W\|_{2^*}$, and, by the characterization \eqref{CarW} of $W$, $u_0=\pm W_{\lambda_0,x_0}$ for some parameters $\lambda_0,x_0$. By uniqueness in \eqref{CP}, $u$ is one of the stationnary solutions $\pm W_{\lambda_0,x_0}$, which are globally defined.

Let us assume now $\|\nabla u_0\|_2<\|\nabla W\|_2$. Then if $E(u_0,u_1)<E(W,0)$, we are in the setting of \cite[Theorem 1.1]{KeMe06Pb}, which asserts than $T_+=T_-=+\infty$. On the other hand, if $E(u_0,u_1)=E(W,0)$, then if $\|u\|_{S(0,T_+)}<\infty$, we know by the finite blow-up criterion of Proposition \ref{prop.criterion}, that $T_+=\infty$, and if $\|u\|_{S(0,T_+)}=\infty$, then by \eqref{T+infty}, $T_+=\infty$. The same argument for negative times shows that $T_-=\infty$.
 \end{proof}

\begin{proof}[Proof of Proposition \ref{propsub}]

\noindent\emph{Proof of \eqref{T+infty}.}
Let $u$ be as in Lemma \ref{lem.compact}. By Proposition \ref{propKeMe.a}, if $T_+<\infty$, then $\int u_1\nabla u_0=0$, but then by Proposition \ref{propKeMe.b}, $T_+=\infty$ which is a contradiction. Thus $T_+=\infty$, which shows \eqref{T+infty}.

\medskip

\noindent\emph{Proof of \eqref{lambda.sub}.}
Assume that \eqref{lambda.sub} does not hold. Then there exists a sequence $t_n\rightarrow +\infty$ such that
\begin{equation}
\label{absurdtau0}
\lim_{n\rightarrow +\infty} t_n\lambda(t_n)=\tau_0\in[0,+\infty).
\end{equation}
Consider
\begin{gather}
\label{defwn}
w_n(s,y)=\frac{1}{\lambda(t_n)^{\frac{N-2}{2}}} u\left(t_n+\frac{s}{\lambda(t_n)},x(t_n)+\frac{y}{\lambda(t_n)}\right).\\
\label{defwn0}
w_{n0}(y)=w_n(0,y),\quad w_{n1}(y)=\partial_sw_n(0,y).
\end{gather}
By the compactnees of $K$, $(w_{n0},w_{n1})_n$ converges (up to the extraction of a subsequence) in $\hdot\times L^2$. Let $(w_0,w_1)$ be its limit, and $w$ be the solution of \eqref{CP} with initial condition $(w_0,w_1)$. Note that $E(w_0,w_1)=E(W,0)$ and $\|\nabla w_0\|_2\leq \|\nabla W\|_2$. Thus by Corollary \ref{corolGlobal}
$$ T_-(w_0,w_1)=\infty.$$
Furthermore, as $-t_n\lambda(t_n)\rightarrow -\tau_0$, and by the continuity of the Cauchy Problem for \eqref{CP} (\eqref{continuity} in Proposition \ref{prop.criterion}),
\begin{align*}
\frac{1}{\lambda(t_n)^{\frac{N-2}{2}}} u_0\left(x(t_n)+\frac{y}{\lambda(t_n)}\right)&=w_n(-t_n\lambda(t_n),y)\underset{n \rightarrow\infty}{\longrightarrow}  w(-\tau_0,y) \text{ in }\hdot\\
\frac{1}{\lambda(t_n)^{\frac{N}{2}}} u_1\left(x(t_n)+\frac{y}{\lambda(t_n)}\right)&=\partial_s w_n(-t_n\lambda(t_n),y)\underset{n \rightarrow\infty}{\longrightarrow}\partial_s w(-\tau_0)\text{ in }L^2.
\end{align*}
Since $\lambda(t_n)$ tends to $0$, we obtain that $w(-\tau_0)=0$ and $\partial_s w(-\tau_0)=0$,
which contradicts the equality $E(w_0,w_1)=E(W,0)$. The proof of \eqref{lambda.sub} is complete.

\medskip

\noindent\emph{Proof of \eqref{moment.sub}.}
According to \eqref{T+infty}, $T_+=\infty$. By Proposition \ref{propKeMe.b}, \eqref{moment.sub} holds unless
\begin{equation}
\label{liminflambda}
\liminf_{t\geq 0} \lambda(t)=0.
\end{equation}
Let us show \eqref{moment.sub} in this case. We will use the same argument as in the proof of Theorem 7.1 in \cite{KeMe06Pb}. Let us sketch it. Consider $(t_n)_n$ such that 
\begin{equation}
\label{hyp.lambda.0}
t_n\underset{n\rightarrow \infty}{\longrightarrow} +\infty, \quad \lambda(t_n) \underset{n\rightarrow \infty}{\longrightarrow}0,\quad \text{and}\quad \forall t\in [0,t_n),\;\lambda(t)>\lambda(t_n).
\end{equation}
Define $w_n$, $w_{n0}$ and $w_{n1}$ by \eqref{defwn} and \eqref{defwn0},
and consider $(w_0,w_1) \in \hdot\times L^2$ such that 
\begin{equation}
\label{limwn0}
\lim_{n\rightarrow +\infty} (w_{n0},w_{n1})_n=(w_0,w_1) \text{ in }\hdot\times L^2,
\end{equation}
and $w$ the solution of \eqref{CP} with initial conditions $(w_0,w_1)$. 

By Corollary \ref{corolGlobal}, $T_-(w)=\infty$. By \eqref{hyp.noscatter} and \eqref{continuity} in Proposition \ref{prop.criterion},
\begin{equation}
\label{wSinfty}
\|w\|_{S(-\infty,0)}=+\infty.
\end{equation}
Now, fix $s\leq 0$, and consider
\begin{equation*}
\tlbda_n(s)=\frac{\lambda\Big(t_n+\frac{s}{\lambda(t_n)}\Big)}{\lambda(t_n)},\quad \tx_n(s)=\lambda(t_n)\left[x\Big(t_n+\frac{s}{\lambda(t_n)}\Big)-x(t_n)\right].
\end{equation*}

By \eqref{lambda.sub}, $t_n\lambda(t_n)\rightarrow +\infty$, and thus for large $n$, $0<t_n+\frac{s}{\lambda(t_n)}\leq t_n$. Hence by \eqref{hyp.lambda.0},
\begin{equation}
\label{tlbda1}
\exists n_0(s), \; \forall n\geq n_0(s),\quad\tlbda_n(s)\geq 1.
\end{equation}
Now, for $t=t(n,s):=t_n+\frac{s}{\lambda_n(s)}$,
\begin{gather}
\label{vn0} 
v_{n0}(s,y):=\frac{1}{\tlbda_n(s)^{\frac{N-2}{2}}}w_n\left(s,\tx_n(s)+\frac{y}{\tlbda_n(s)}\right)=\frac{1}{\lambda(t)^{\frac{N-2}{2}}}u\left(t,x(t)+\frac{y}{\lambda(t)}\right),\\
\label{vn1}
v_{n1}(s,y):=\frac{1}{\tlbda_n(s)^{\frac{N}{2}}} (\partial_sw_n)\left(s,\tx_n(s)+\frac{y}{\tlbda_n(s)}\right)=
\frac{1}{\lambda(t)^{\frac{N}{2}}}(\partial_tu)\left(t,x(t)+\frac{y}{\lambda(t)}\right).
\end{gather}
which shows that $(v_{n0}(s),v_{n1}(s))\in K$. In view of \eqref{limwn0} and the continuity property \eqref{continuity} in Proposition \ref{prop.criterion}, $(w_n(s),\partial_sw_n(s))$ converges in $\hdot\times L^2$. By the compactness of $\overline{K}$ and \eqref{tlbda1} this shows that there exists $\tlbda(s)\in [1,+\infty)$, $\tx(s)\in \RR^N$ such that for some subsequences,
\begin{equation*}
\lim_{n\rightarrow +\infty} \tlbda_n(s))=\tlbda(s),\quad \lim_{n\rightarrow +\infty} \tx_n(s))=\tx(s)
\end{equation*}
and
$$ \left(\frac{1}{\tlbda(s)^{\frac{N-2}{2}}}w\left(s,\tx(s)+\frac{\cdot}{\tlbda(s)}\right),
\frac{1}{\tlbda(s)^{\frac{N}{2}}} (\partial_sw)\left(s,\tx(s)+\frac{\cdot}{\tlbda(s)}\right)\right)\in \overline{K}.$$ 

Thus $w$ fullfills all the assumptions of Proposition \ref{propKeMe.a}, case \eqref{propKeMe.infinite}, which shows that  $\int w_1\nabla w_0=0$. By \eqref{limwn0} and the conservation of $\int \partial_t u(t) \nabla u(t)$
$$\int u_1\nabla u_0=\int \partial_t u(t_n)\nabla u(t_n)=\int w_{n1}\nabla w_{n0}\underset{n\rightarrow\infty}{\longrightarrow} \int w_1\nabla w_0=0.$$
The proof of \eqref{moment.sub} is complete.
\medskip

\noindent\emph{Proof of \eqref{x.sub}.}

We follow the lines of the proof of Lemma 5.5 in \cite{KeMe06Pb}. Assume that \eqref{x.sub} does not hold, and consider $t_n\rightarrow +\infty$ such that
\begin{equation*}
\frac{|x(t_n)|}{t_n}\geq \eps_0>0.
\end{equation*}
In particular, $x(t)$ is not bounded. We may assume that $x(0)=0$, $\lambda(0)=1$ and that $x$ and $\lambda$ are continuous. For $R>0$, let
$$ t_0(R):=\inf\big\{t\geq 0,\;|x(t)|\geq R\big\}\in [0,+\infty).$$
Thus $t_0(R)$ is well defined, $t_0(R)>0$,  $|x(t)|<R$ for $0\leq t<t_0(R)$ and $|x(t_0(R))|=R$. As a consequence, if $R_n:=|x(t_n)|$, then $t_n\geq t_0(R_n)$, which implies
\begin{equation}
\label{absurd.x(t)}
\frac{R_n}{t_0(R_n)}\geq \eps_0.
\end{equation}
Let 
\begin{equation}
\label{defe(u)}
e(u):=\frac{1}{2}|\partial_t u|^2+\frac{1}{2}|\nabla u|^2-\frac{1}{2^*}|u|^{2^*},\quad
r(u):=|\partial_t u|^2+|\nabla u|^2+\frac{1}{|x|^2}|u|^2+|u|^{2^*}.
\end{equation}
By compactness of $K$, we know that for each $\eps>0$, there exists $R_0(\eps)$ such that
\begin{equation}
\label{defR_0}
\forall t\geq 0,\quad \int_{\lambda(t)|x-x(t)|\geq R_0(\eps)} r(u)dx \leq \eps. 
\end{equation}
Let $\varphi\in C_0^{\infty}(\RR^N)$ be radial, nonincreasing and such that $\varphi(x)=1$ for $|x|\leq 1$ and $\varphi(x)=0$ if $|x|\geq 2$. Let $\psi_R(x):=x\varphi(\frac{x}{R})$. Let $\eps>0$, to be chosen later independently of $n$ and
\begin{align*}
\tR_n&:=\frac{R_0(\eps)} {\inf_{t\in [0,t_0(R_n)]} \lambda(t)}+|x(t_n)|=\frac{R_0(\eps)}{\inf_{t\in [0,t_0(R_n)]} \lambda(t)}+R_n\\
z_{\tR_n}(t)&:=\int_{\RR^N} \psi_{\tR_n}(x) e(u)(t,x) dx.
\end{align*}
Note that 
\begin{equation}
\label{minor.important}
0\leq t\leq t_0(R_n),\; |x|\geq \tR_n \Longrightarrow \lambda(t)|x-x(t)|\geq \lambda(t)(\tR_n-R_n)\geq R_0(\eps).
\end{equation}

\medskip

\noindent\emph {Step 1. Bound on $z_{\tR_n}'(t)$.}
Let us show
\begin{equation}
\label{bound.z'R}
\exists C_1>0, \; \forall n,\quad 0\leq t\leq t_0(R_n)\Longrightarrow|z'_{\tR_n}(t)|\leq C_1\eps.
\end{equation}

Indeed by explicit calculation and equation \eqref{CP} (see \cite[Lemma 5.3]{KeMe06Pb}), we get  recalling that by \eqref{moment.sub} and the conservation of the moment $\int \partial_t u(t)\nabla u(t)=\int u_1\nabla u_0=0$ 
\begin{equation}
\label{z'R}
z'_{\tR_n}(t)=\int_{\RR^N}\partial_t u\nabla u+O\left(\int_{|x|\geq \tR_n} r(u)dx\right)=O\left(\int_{|x|\geq \tR_n} r(u)dx\right)
\end{equation}
Estimate \eqref{bound.z'R} then follows from \eqref{defR_0} and \eqref{minor.important} 

\medskip

\noindent\emph{Step 2. Bounds on $z_{\tR_n}(0)$ and $z_{\tR_n}(t_0(R_n))$.} We next show
\begin{align}
\label{boundzR(0)}
|z_{\tR_n}(0)|&\leq 2\eps \tR_n+MR_0(\eps)\\
\label{boundzR(t_0)}
R_n\big(E(W)-\eps\big)-2\tR_n\eps-M\frac{R_0(\eps)}{\lambda(t_0(R_n))} &\leq  |z_{\tR_n}(t_0(R_n))|,
\end{align}
where $M:=\sup_{t\geq 0} \int r(u)(t,x) dx\leq C\|\nabla W\|^2_2$ by Claim \ref{usefulclaim}.
We have.
\begin{equation*}
|z_{\tR_n}(0)|= \int_{|x|\geq R_0(\eps)}\psi_{\tR_n}e(u)dx+\int_{|x|\leq R_0(\eps)}\psi_{\tR_n}e(u) dx.
\end{equation*}
Recall that $|\psi_{\tR_n}(x)|\leq |x|\leq 2\tR_n$. According to \eqref{defR_0} (using that $x(0)=0$ and $\lambda(0)=1$), the first term is bounded by $2\tR_n\eps$. The second term is bounded by $R_0(\eps)\int r(u) dx$, which yields \eqref{boundzR(0)}.

Write, for $t\in[0,t_0(R_n)]$
\begin{equation}
\label{formule.zt}
z_{\tR_n}\big(t\big)=\int_{\lambda(t)|x-x(t)|\geq R_0(\eps)}\psi_{\tR_n}e(u)dx+\int_{\lambda(t)|x-x(t)|\leq R_0(\eps)}\psi_{\tR_n}e(u) dx
\end{equation}
Using again \eqref{defR_0}, we get $
\left|\int_{\lambda(t)|x-x(t)|\geq R_0(\eps)}\psi_{\tR_n}e(u)dx\right|\leq 2\tR_n\eps.$
According to \eqref{minor.important} and the definition of $\psi_{\tR_n}$, if $\lambda(t)|x-x(t)|< R_0(\eps)$, then $|x|<\tR_n$ which implies $\psi_{\tR_n}(x)=x$. Thus the second term of \eqref{formule.zt} is
\begin{multline}
\label{zt2}
\int_{\lambda(t)|x-x(t)|\leq R_0(\eps)}x e(u) dx\\=\int_{\lambda(t)|x-x(t)|\leq R_0(\eps)}x(t) e(u) dx+\int_{\lambda(t)|x-x(t)|\leq R_0(\eps)}(x-x(t)) e(u) dx.
\end{multline}
The second term in the right-hand side of \eqref{zt2} is bounded by $\frac{R_0(\eps)}{\lambda(t)}\int r(u)(t,x)dx$. On the other hand
\begin{gather*}
\int_{\lambda(t)|x-x(t)|\leq R_0(\eps)}x(t) e(u) dx=x(t)E(W)-\int_{\lambda(t)|x-x(t)|\geq R_0(\eps)}x(t) e(u) dx
\end{gather*}
and thus by \eqref{defR_0}
\begin{equation}
\label{zt3}
\left|\int_{\lambda(t)|x-x(t)|\leq R_0(\eps)}x(t) e(u) dx\right|\geq |x(t)|(E(W)-\eps).
\end{equation}
Combining \eqref{zt2} and \eqref{zt3} with $t=t_0(R_n)$ we get \eqref{boundzR(t_0)}.

\medskip
\noindent\emph{Step 3. Conclusion of the proof of \eqref{x.sub}.}
According to the two precedent steps 
\begin{align}
\notag
C_1\eps t_0(R_n)&\geq R_n(E(W)-\eps)-4\tR_n\eps-\frac{2MR_0(\eps)}{\lambda(t_0(R_n))}\\
\label{boundRnConclu}
C_1\eps&\geq \frac{R_n}{t_0(R_n)}(E(W)-\eps)-4\frac{\tR_n}{t_0(R_n)}\eps-M\frac{R_0(\eps)}{t_0(R_n)\lambda(t_0(R_n))}.
\end{align}
As a consequence of \eqref{lambda.sub}  
$$t_0(R_n)\lambda(t_0(R_n))\rightarrow 0 \text{ as }n\rightarrow \infty.$$ 
Furthermore
$\frac{\tR_n}{t_0(R_n)}=\frac{R_n}{t_0(R_n)}+\frac{R_0(\eps)}{t_0(R_n)\inf_{t\in [0,t_0(R_n)]} \lambda(t)}.$
Again by \eqref{lambda.sub},
$t_0(R_n)\inf_{t\in [0,t_0(R_n)]} \lambda(t)$ tends to $+\infty$. Thus
$$ \frac{\tR_n}{t_0(R_n)}=\frac{R_n}{t_0(R_n)}+o(1),\; n\rightarrow \infty.$$
Together with \eqref{absurd.x(t)} and \eqref{boundRnConclu}, we get
\begin{equation*}
C_1\eps\geq \frac{R_n}{t_0(R_n)}(E(W)-5\eps)+o(1)\geq \eps_0(E(W)-5\eps)+o(1), \quad n\rightarrow \infty.
\end{equation*}
Chosing $\eps$ small enough, so that $C_1\eps\leq \frac{\eps_0}{4}E(W)$ and $E(W)-5\eps>\frac 12 E(W)$ we get a contradiction.
\end{proof}

\section{Convergence to $W$ in the subcritical case}
\label{sec.sub}
The aim of this section is to prove the following result in the subcritical  situation ($\|\nabla u_0\|_2<\|\nabla W\|_2$), which is the nonradial version of the result of \cite[Proposition 3.1]{DuMe07P} in the NLS radial setting. The main difficulty here compared to \cite{DuMe07P} and \cite{KeMe06Pb} is to control the space localization of the energy (see \S \ref{sub.CVseq}).
\begin{prop}
\label{prop.CVexpu}
Let $u$ be a solution of \eqref{CP} such that \eqref{hyp.sub} and \eqref{hyp.noscatter} hold.
Then there exist $\lambda_0,x_0$ such that
\begin{equation}
\label{CVexpu}
\|\nabla(u(t)-W_{\lambda_0,x_0})\|_2+\|\partial_t u\|_2\leq Ce^{-ct}.
\end{equation}
Furthermore,
\begin{equation}
\label{scattering-}
\|u\|_{S(-\infty,0)}<\infty.
\end{equation}
\end{prop}
\begin{remark}
 From Corollary \ref{corolGlobal}, \eqref{hyp.sub} implies that $u$ is defined on $\RR$.
\end{remark}

\subsection{Convergence for a subsequence}
\label{sub.seq}
Let
\begin{equation}
\label{defdd}
\DD(t):=\left|\int |\nabla u(t,x)|^2 dx -\int |\nabla W(x)|^2 dx \right|+\int|\partial_t u(t,x)|^2dx.
\end{equation}
Then the equality $E(u(t),\partial_tu(t))=E(W,0)$ implies $\Big|\|u\|_{2^*}^{2^*}-\|W\|_{2^*}^{2^*}\Big|\leq C\dd(t)$. It is known (see \cite{Li85Reb}) that the variational characterization \eqref{CarW} of $W$ by Aubin and Talenti implies that there exists a function $\eps_0(\delta)$ such that $\eps_0(\delta)\rightarrow 0$ as $\delta\rightarrow 0$ and, for any fixed $t$
\begin{equation}
\label{CarW2}
\inf_{\mu,X,\pm}\left\|\nabla \big(u_{\mu,X}(t)\pm W\big)\right\|_2\leq \eps_0\big(\DD(t)\big).
\end{equation}
 The key point of the proof of Proposition \ref{prop.CVexpu} is to show that $\dd(t)$ tends to $0$, which by \eqref{CarW2} implies that there exists $(\lambda(t),x(t))_{t\geq 0}$ such that $u_{\lambda(t),x(t)}(t)-W$ tends to $0$ in $\hdot$ as $t$ tends to $+\infty$. We first show:
\begin{lemma}
\label{lem.dd.0}
Let $u$ be a solution of \eqref{CP} such that \eqref{hyp.sub} and \eqref{hyp.noscatter} hold.
Then 
\begin{equation}
\label{dd.CV.0}
\lim_{T\rightarrow +\infty} \frac 1T \int_0^T \DD(t) dt \underset{t\rightarrow +\infty}{\rightarrow} 0.
\end{equation}
\end{lemma}
\begin{corol}
\label{corol.subseq}
There exists an increasing sequence $\tau_n\rightarrow +\infty$ such that
$$ \lim_{n\rightarrow +\infty} \DD(\tau_n)=0.$$
\end{corol}
\begin{proof}
Let $\varphi$ be a $C^{\infty}$ function such that $\varphi(x)=1$ if $|x|\leq 1$ and $\varphi(x)=0$ if $|x|\geq 2$. For $R>0$, write 
\begin{equation}
\label{defphiR}
\varphi_R(x)=\varphi(x/R) \text{ and }\psi_R(x)=x\varphi(x/R).
\end{equation}
Let $\eps>0$. Consider as in \cite[\S 5]{KeMe06Pb}
\begin{equation}
\label{defgR}
g_R(t):=\int \psi_R.\nabla_x u \partial_t u dx+(\frac{N-1}{2}) \int \varphi_R u \partial_t u dx.
\end{equation}

\noindent\emph{Step 1. Bound on $g_R(t)$.} We first show
\begin{equation}
\label{boundgR}
\exists C_1>0,\;\forall t\geq 0,\quad |g_R(t)|\leq C_1 R.
\end{equation}
Indeed, note that $\supp \varphi_R \cup \supp \psi_R \subset \{|x|\leq 2R\}$, so that $|\psi_R(x)|\leq 2R$ and $|\varphi_R(x)|\leq \frac{2R}{|x|}$. Hence
$$
|g_R(t)|\leq 2R \int |\nabla_x u \partial_t u| dx+\frac{N}{2} \int \frac{2R}{|x|}|u \partial_t u |dx,$$
which yields \eqref{boundgR} by Hardy's inequality and Claim \ref{usefulclaim}.

\noindent\emph{Step 2. Bound on $g'_R(t)$.}
There exists $C_2>0$, $c>0$, such that for all $\eps>0$, there exists $t_1=t_1(\eps)>0$ such that 
\begin{equation}
\label{boundg'R}
\forall t\in[t_1,T],\quad g_{\eps T}'(t)\leq -c\dd(t)+C_2 \eps.
\end{equation}
By explicit computation and the equality $E(u_0,u_1)=E(W,0)$ (see Claim \ref{calculs} in the appendix)
\begin{equation*}
g_{\eps T}'(t)=\frac{1}{N-2}\int |\partial_t u|^2dx-\frac{1}{N-2}\left(\int |\nabla W|^2dx-\int |\nabla u|^2dx\right)+O\left(\int_{|x|\geq \eps T} r(u)dx\right),
\end{equation*}
where $r(u)$ is defined in \eqref{defe(u)}. Note that by Claim \ref{usefulclaim} an elementary calculation
\begin{multline}
\label{bound.important}
-\frac{1}{N-2}\int |\partial_t u|^2dx+\frac{1}{N-2}\left(\int |\nabla W|^2dx-\int |\nabla u|^2dx\right)
\\ \geq \frac{1}{N+2} \left(\int |\partial_t u|^2dx+\int |\nabla W|^2dx-\int |\nabla u|^2dx \right).
\end{multline}
Thus there exists $C_2>0$ such that
\begin{equation}
\label{minor.gR'}
g_R'(t)\leq -\frac{1}{N+2}\dd(t)+C_2\int_{|x|\geq \eps T} r(u).
\end{equation}
By Proposition \ref{propsub}, $t\lambda(t)\rightarrow +\infty$ and $|x(t)|/t\rightarrow 0$. Thus we may chose $t_1=t_1(\eps)$ such that $$t\geq t_1\Longrightarrow |x(t)|\leq \frac{\eps}{2} t\text{ and }|\lambda(t)|\geq \frac{2R_0(\eps)}{\eps t},$$
where $R_0(\eps)$ is defined in \eqref{defR_0} 
Then for $t_1\leq t\leq T$,
\begin{equation*}
|x|\geq \eps T\Longrightarrow \lambda(t)(|x|-|x(t)|)\geq \frac{2R_0(\eps)}{\eps T}\left(\eps T-\frac{\eps T}{2}\right)\geq R_0(\eps),
\end{equation*}
which yields, together with \eqref{defR_0} and \eqref{minor.gR'}, our expected estimate \eqref{boundg'R}.

\noindent\emph{Step 3. End of the proof.}

Integrating \eqref{boundg'R} between $t_1$ and $T$, one gets
$$ g_{\eps T}(T)-g_{\eps T}(t_1)=\int_{t_1}^T g_{\eps T}'(t) dt\leq -c \int_{t_1}^T\dd(t)dt+C_2 (T-t_1)\eps.$$
and thus, by \eqref{boundgR},
$$ \frac{c}{T} \int_{t_1}^T\dd(t)\leq 2C_1\eps +C_2\eps,$$
hence
$$ \limsup_{T\rightarrow +\infty} \frac{1}{T}\int_{0}^T\dd(t)dt\leq \frac{2C_1+C_2}{c}\eps$$
which yields the result.
\end{proof}

\subsection{Modulation of solutions.}
\label{sub.modul}
We will now precise, using modulation theory, the dynamics of solutions of \eqref{CP} near $W$. We will only suppose 
\begin{equation}
 \label{energy} E(u_0,u_1)=E(W,0),
\end{equation}
without any further assumption on the size of $\|\nabla u_0\|_2$.
We have the following development of the energy near $W$:
\begin{equation}
\label{W+f}
E(W+f,g)=E(W,0)+Q(f)+\frac 12\|g\|_{2}^2+O\big(\|\nabla f\|^3_{2}\big),\quad f\in \hdot, g\in L^2
\end{equation}
where $Q$ is the quadratic form on $\hdot$ defined by
\begin{equation}
\label{defQ}
Q(f):=\frac{1}{2}\int |\nabla f|^2-\frac{N+2}{2(N-2)} \int W^{\frac{4}{N-2}} f^2.
\end{equation}
Let us specify an important coercivity property of $Q$. Consider the orthogonal directions $W$, $\tW$, $W_{j}$, $j=1\ldots N$ in the real Hilbert space $\hdot=\hdot(\RR^N,\CC)$ where $\tW$ and $W_{j}$ are defined by
\begin{equation}
\label{tW.Wj}
\begin{split}
\tW=\tilde{c} \frac{\partial}{\partial\mu}\left( W_{\mu,X} \right)_{\restriction (\mu,X)=(1,0)}=-\tilde{c}\Big(\frac{N-2}{2}W+x\cdot \nabla W\Big)\\
W_{j}=c_j\frac{\partial}{\partial X_j}\left(W_{\mu,X}\right)_{\restriction (\mu,X)=(1,0)}=c_j\partial_{x_j} W.
\end{split}
\end{equation}
and the constants $\tilde{c}$, $c_1$, \ldots, $c_N$ are chosen so that 
\begin{equation}
\label{normalisation}
\|\nabla \tW\|_2=\|\nabla W_1\|_2=\ldots \|\nabla W_N\|_2=1.
\end{equation}
We have
\begin{equation}
\label{valeursQ}
Q(W)=-\frac{2}{(N-2)C_N^N}, \quad Q_{\restriction \vect\{\tW,W_1,\ldots,W_n\}}= 0,
\end{equation} 
where $C_N$ is the best Sobolev constant in dimension $N$.  The first assertion follows from direct computation and the fact that $\|W\|_{2^*}^{2^*}= \|\nabla W\|_2^2=\frac{1}{C_N^N}$. From \eqref{W+f}, \eqref{tW.Wj}, and the invariance of $E$ by all transformations $f\mapsto f_{\mu,X}$, we get that $Q(\tW)=Q(W_1)=\ldots=Q(W_N)=0$. Furthermore, it is easy to check that $\tW$, $W_1$, \ldots, $W_N$ are $Q$-orthogonal, which gives the second assertion.

 Let $H:=\vect\{W,\tW,W_1,\ldots,W_N\}$ and $H^{\bot}$ its orthogonal subspace in the real Hilbert space $\hdot$. The quadratic form $Q$ is nonpositive on $H$. By the following claim, $Q$ is positive definite on $H^{\bot}$ (see \cite[Appendix D]{Re90} for the proof).
\begin{claim}
\label{coercivity}
There is a constant $\tilde{c}>0$ such that for all radial function $\tf$ in $H^{\bot}$
$$ Q(\tf) \geq \tilde{c}\|\nabla \tf\|_{2}^2.$$
\end{claim}
Now, let $u$ be a solution of \eqref{CP} satisfying \eqref{energy} and define $\dd(t)$ as in \eqref{defdd}. We would like to specify \eqref{CarW2}. We start by chosing $\mu$ and $X$.
\begin{claim}
\label{claim.ortho}
There exists $\delta_0>0$ such that for all solution $u$ of \eqref{CP} satisfying \eqref{energy}, and for all $t$ in the interval of existence of $u$ such that
\begin{equation}
\label{hyp.D}
\DD(t)\leq \delta_0,
\end{equation}
there exists $(\mu(t),\xx(t))\in (0,+\infty)\times \RR^N$ such that
$$ u_{\mu,\xx} \in \{\tW,W_1,\ldots,W_N\}^{\bot}.$$
\end{claim}
We omit the proof of Claim \ref{claim.ortho}, which follows from \eqref{CarW2} and the implicit function Theorem. We refer for example to \cite[Claim 3.5]{DuMe07P} for a similar proof.

If $a$ and $b$ are positive, we write $a\approx b$ when $C^{-1}a \leq b\leq Ca$ with a positive constant $C$ independent of all parameters of the problem.\par

Now, consider $u$ satisfying \eqref{energy} and assume that on an open subset $J$ of its interval of definition, $u$ also satisfies \eqref{hyp.D}.

According to the preceding claim, there exist $(\mu(t),\xx(t))$ such that 
$$\forall t\in J, \quad u_{\mu,\xx}(t) \in \{\tW,W_1,\ldots,W_N\}^{\bot}.$$
We will prove the following lemma, which is a consequence of Claim \ref{coercivity}, in Appendix \ref{Appendix.Decompo}.
\begin{lemma}[Estimates on the modulation parameters]
\label{lem.est.modul}
Let $u$, $\mu$, $\xx$ be as above. Changing $u$ into $-u$ if necessary, write
$$ u_{\mu,\xx}(t)=(1+\alpha(t))W+\tf(t) ,\quad 1+\alpha:=\frac{1}{\|\nabla W\|_2^2}\int \nabla W\cdot \nabla u_{\mu,\xx}dx,\quad \tf\in H^{\bot}.$$
Then
\begin{gather}
\label{est.modul.1}
|\alpha|\approx \|\nabla (\alpha W+\tf)\|\approx \|\nabla \tf\|_2+\|\partial_t u\|_2\approx \dd(t)\\
\label{est.modul.3}
\left|\frac{\alpha'}{\mu}\right|+\left|\frac{\mu'}{\mu^2}\right|+\left|\xx'(t)\right|\leq C\dd(t).
\end{gather}
\end{lemma}


\subsection{Exponential convergence to $W$.}
\label{sub.CVseq}
Using Subsections \ref{sub.seq} and \ref{sub.modul}, we are now ready to prove Proposition \ref{prop.CVexpu}

Let $u$ be as in the proposition. By Lemma \ref{lem.compact}, we may assume that 

\begin{equation}
\label{hyp.compact}
\begin{split}
\text{There exist functions } \lambda(t),\,x(t)\text{ continuous on }[0,+\infty)\text{ such that }\\
K:=\left\{ \big(u(t),\partial_t u(t)\big)_{\lambda(t),x(t)},\; t\in [0,+\infty)\right\}
\text{ is relatively compact in }\hdot\times L^2.
\end{split}
\end{equation}
Let $\mu(t)$ and $\xx(t)$ be the modulation parameters of Subsection \ref{sub.modul}, defined for $\DD(t)\leq \delta_0$. It is easy to see that the compactness of $K$ 
implies that the set 
$$K_1:=\left\{ \big(u(t),\partial_t u(t)\big)_{[\mu(t),\xx(t)]},\; t\in [0,+\infty),\;\dd(t)<\delta_0\right\}$$
has compact closure in $\hdot\times L^2$. By an elementary construction, one can find continuous functions $\tlbda(t)$ and $\tx(t)$ of $t\in [0,+\infty)$ such that $\big(\tlbda(t),\tx(t)\big)=(\xx(t),\mu(t))$ if $\DD(t)\leq \delta_0$. The set $\widetilde{K}$ defined as in \eqref{hyp.compact} has compact closure in $\hdot\times L^2$. We will still denote by $x(t)$ and $\lambda(t)$ the new parameters that satisfy, in addition to \eqref{hyp.compact},
\begin{equation}
\label{hyp.equal.modul}
\dd(t)<\delta_0\Longrightarrow \lambda(t)=\mu(t),\quad x(t)=\xx(t).
\end{equation}

The proof of Proposition \ref{prop.CVexpu} relies on the two following Lemmas.
\begin{lemma}[Virial type estimates on $\DD(t)$]
\label{lem.virial}
Let $u$ be a solution of \eqref{CP} satisfying \eqref{hyp.sub}, \eqref{hyp.compact}, and \eqref{hyp.equal.modul}.
Then there is a constant $C>0$ such that
$$ 0\leq \sigma<\tau \Longrightarrow \int_{\sigma}^{\tau} \DD(t)dt\leq C\left(\sup_{\sigma\leq t\leq\tau} |x(t)|+\frac{1}{\lambda(t)}\right)\big(\DD(\sigma)+\DD(\tau)\big).$$
\end{lemma}
\begin{remark}
\label{remvirial}
We will also need the following variant of Lemma \ref{lem.virial}, whose proof is exactly the same: if $u$ satisfies the assumptions of Lemma \ref{lem.virial} for all $t\in \RR$, then there is a constant $C>0$
$$ -\infty<\sigma<\tau<+\infty\Longrightarrow\int_{\sigma}^{\tau} \dd(t)dt\leq C\left(\sup_{\sigma\leq t\leq\tau} |x(t)|+\frac{1}{\lambda(t)}\right)\big(\dd(\sigma)+\dd(\tau)\big).$$
\end{remark}
\begin{lemma}[Parameters control]
\label{lem.bound.x}
Let $u$ be a solution of \eqref{CP} fullfilling the assumptions of Lemma \ref{lem.virial}. Then there exists a constant $C>0$ such that
$$  0\leq \sigma\text{ and } \sigma+\frac{1}{\lambda(\sigma)}\leq \tau\Longrightarrow |x(\tau)-x(\sigma)|+\left|\frac{1}{\lambda(\tau)}-\frac{1}{\lambda(\sigma)}\right|\leq C\int_{\sigma}^{\tau} \dd(t)dt.$$
\end{lemma}
\begin{remark}
The technical assumption $\sigma+\frac{1}{\lambda(\sigma)}<\tau$ is needed because of the infinite choice of parameters $x(t)$ and $\lambda(t)$ when $\DD(t)>\delta_0$.
\end{remark}
Let us first assume Lemma \ref{lem.virial} and \ref{lem.bound.x} to show Proposition \ref{prop.CVexpu}. 
\begin{proof}[Proof of Proposition \ref{prop.CVexpu}]
\noindent\emph{Step 1.}
Let us show that, (replacing $u$ by $u(\cdot-x_{\infty})$ for some $x_{\infty}\in\RR^N$ if necessary), there exist $c,C>0$ and  $\lambda_{\infty}\in (0,\infty)$
\begin{equation}
\label{expparam}
\int_t^{\infty}\DD(s)ds+ |\lambda(t)-\lambda_{\infty}|+|x(t)| \leq Ce^{-ct}.
\end{equation}

We first show that $x(t)$ and $\frac{1}{\lambda(t)}$ are bounded.
According to Lemmas \ref{lem.virial} and \ref{lem.bound.x}, there exists a constant $C_0>0$ such that for all $0\leq \sigma\leq s<t\leq\tau$ with $s+\frac{1}{\lambda(s)}<t$, we have
\begin{equation}
\label{ineg.crucial}
|x(s)-x(t)|+\left|\frac{1}{\lambda(s)}-\frac{1}{\lambda(t)}\right|\leq C_0 \left\{\sup_{\sigma\leq r \leq \tau} \Big(|x(r)|+\frac{1}{\lambda(r)}\Big)\right\}(\dd(\sigma)+\dd(\tau)).
\end{equation}
Now consider the increasing sequence $\tau_n\rightarrow +\infty$, given by Corollary \ref{corol.subseq}, and chose $n_0$ such that 
\begin{equation}
\label{bounddd}
n\geq n_0\Rightarrow\dd(\tau_{n})\leq \frac{1}{4C_0}
\end{equation}
Using \eqref{ineg.crucial} with $\sigma=s=\tau_{n_0}$, and $\tau=\tau_n$ for some large $n$ we get, in view of \eqref{bounddd}
$$ \tau_{n_0}+\frac{1}{\lambda(\tau_{n_0})}<t \Longrightarrow |x(\tau_{n_0})-x(t)|+\left|\frac{1}{\lambda(\tau_{n_0})}-\frac{1}{\lambda(t)}\right|\leq \frac{1}{2} \left\{\sup_{\tau_{n_0}\leq r} \Big(|x(r)|+\frac{1}{\lambda(r)}\Big)\right\}.$$
Thus
$$ \sup_{\tau_{n_0}+\frac{1}{\lambda(\tau_{n_0})}\leq t} \Big(|x(t)|+\frac{1}{\lambda(t)}\Big)\leq \frac{1}{2} \left\{\sup_{\tau_{n_0}\leq t} \Big(|x(t)|+\frac{1}{\lambda(t)}\Big)\right\}+|x(\tau_{n_0})|+\frac{1}{\lambda(\tau_{n_0})},$$
which shows the boundedness of $x$ and $\lambda$.

By Lemma \ref{lem.virial} between $\sigma=t$ and $\tau=\tau_n$, and taking into account the fact that $x(t)$ and $\frac{1}{\lambda(t)}$ are bounded, we get $\int_t^{\tau_n} \dd(s)ds\leq C(\dd(t)+\dd(\tau_n))$. Letting $n$ goes to infinity we obtain $
\int_{t}^{+\infty} \dd(s)ds \leq C\dd(t).$ Thus for some constants $c,C>0$,
\begin{equation}
\label{expdd}
\int_{t}^{+\infty} \dd(s)ds \leq Ce^{-c t},
\end{equation} 
which is the first bound in \eqref{expparam}.

By \eqref{expdd}, Lemma \ref{lem.bound.x} and the fact that $x(t)$ and $\frac{1}{\lambda(t)}$ are bounded, we have, if $\sigma+\frac{1}{\lambda(\sigma)}<\tau$, $|x(\sigma)-x(\tau)|+\big|\frac{1}{\lambda(\sigma)}-\frac{1}{\lambda(\tau)}\big|\leq Ce^{-c\sigma}$.  Thus there exist $x_{\infty}\in \RR^N$, $\ell_{\infty}\in [0,+\infty)$ such that 
$$ |x(t)-x_{\infty}|+\left|\frac{1}{\lambda(t)}-\ell\right|\leq Ce^{-ct}.$$
Translating $u$, we will assume $x_{\infty}=0$. It remains to show that $\ell_{\infty}>0$. Assume that $\ell_{\infty}=0$.
Let $0\leq \sigma\leq s$. Let $\tau_n$ be the sequence such that $\dd(\tau_n)\rightarrow 0$. Then, by \eqref{ineg.crucial}, if $n$ is large enough (so that $s+\frac{1}{\lambda(s)}<\tau_n$),
\begin{equation*}
|x(s)-x(\tau_n)|+\left|\frac{1}{\lambda(s)}-\frac{1}{\lambda(\tau_n)}\right|\leq C_0 \left[\sup_{\sigma\leq r \leq \tau_n} \Big(|x(r)|+\frac{1}{\lambda(r)}\Big)\right](\dd(s)+\dd(\tau_n)).
\end{equation*}
Letting $n$ tends to infinity, we get, by the assumptions $\ell_{\infty}=0$, and $x_{\infty}=0$
\begin{equation*}
0\leq\sigma\leq s\Longrightarrow |x(s)|+\frac{1}{\lambda(s)}\leq C_0 \left[\sup_{\sigma\leq r} \left(|x(r)|+\frac{1}{\lambda(r)}\right)\right]\dd(\sigma).
\end{equation*}
Taking the supremum in $s$ in the preceding inequality, we get, if $\sigma=\tau_n$
\begin{equation*}
\sup_{\tau_n\leq s} |x(s)|+\frac{1}{\lambda(s)}\leq C_0 \left[\sup_{\tau_n\leq s} \left(|x(s)|+\frac{1}{\lambda(s)}\right)\right]\dd(\tau_n).
\end{equation*}
Recalling that $d(\tau_n)$ tends to $0$, we get a contradiction, showing that $\ell_{\infty}>0$. The proof of \eqref{expparam} is now complete.

\medskip

\noindent\emph{Step 2. Proof of \eqref{CVexpu}.} Let us first show by contradiction
\begin{equation}
\label{limite0}
\lim_{t\rightarrow +\infty} \dd(t)=0.
\end{equation}
Indeed, if it does not hold, there exists a subsequence of $(\tau_n)_n$ (that we still denote by $(\tau_n)_n$), and a sequence $(\ttau_n)_n$ such that 
$$\tau_n<\ttau_n, \quad \forall t\in [\tau_n,\ttau_n),\;\dd(t)<\delta_0/2, \quad \dd(\ttau_n)=\delta_0/2.$$
On $[\tau_n,\ttau_n]$ the parameters $\alpha(t)$, $\mu(t)$ and $\xx(t)$ of Lemma \ref{lem.est.modul} are well-defined. 
By \eqref{expparam} and Lemma \ref{lem.est.modul}.
\begin{equation}
\label{major.alpha'}
|\alpha(\tau_n)-\alpha(\ttau_n)|\leq \int_{\tau_n}^{\ttau_n} \left|\frac{\alpha'(t)}{\mu(t)}\right|dt \leq C\int_{\tau_n}^{\ttau_n}\dd(t)dt \leq Ce^{-\tau_n}.
\end{equation}
By Lemma \ref{lem.est.modul}, $\alpha(t)\approx \dd(t)$. As $\DD(\tau_n)\rightarrow 0$ and $\DD(\ttau_n)=\delta_0/2$, this contradicts \eqref{major.alpha'}, showing \eqref{limite0}.\par

In view of \eqref{limite0}, there exists $T>0$ such that for $t\geq T$, $\dd(t)<\delta_0$, so that $u$ is close to $\pm W$ for $t\geq T$. By continuity of $u$, the sign before $W$ does not change for large $t$. Changing $u$ into $-u$ if necessary, we can make it a $+$. Write as in Lemma \ref{lem.est.modul}
\begin{equation}
\label{umuxx}
u_{\mu,\xx}(t)=(1+\alpha(t))W+\tf(t).
\end{equation}
Integrating the estimate $|\alpha'(t)|\leq C\mu(t)\dd(t)$ of Lemma \ref{lem.est.modul}, we get, by \eqref{expparam}, $|\alpha(t)|\leq Ce^{-ct}$. Furthermore, again by Lemma \ref{lem.est.modul}, $\|\nabla \tf\|_2+\|\partial_t u\|_2\lesssim \dd(t)\approx |\alpha(t)|$. Thus
\begin{equation}
\label{final.expo}
\forall t\geq T, \quad |\alpha(t)|+|\mu(t)-\lambda_{\infty}|+|\xx(t)|+\|\nabla \tf\|_2+\|\partial_t u(t)\|_2\leq Ce^{-ct}.
\end{equation}
This implies \eqref{CVexpu} in view of \eqref{umuxx}.

\medskip

\noindent\emph{Step 3. Proof of \eqref{scattering-}.}
Assume, in addition to the assumption of Proposition \ref{prop.CVexpu}, that we have
\begin{equation}
\label{noscatter-}
\|u\|_{S(-\infty,0)}=+\infty.
\end{equation}
By Lemma \ref{lem.compact} there exist $\lambda(t)$ and $x(t)$, defined for $t\in\RR$ such that 
$$K:=\left\{ \big(u(t),\partial_t u(t)\big)_{\lambda(t),x(t)},\; t\in \RR\right\}$$
has compact closure in $\hdot\times L^2$. As a consequence of the preceding steps, applied to $u(t,x)$ and $u(-t,x)$, we get that $\dd(t)$  tends to $0$ as $t$ goes to $+\infty$ and $-\infty$,
and that $\frac{1}{\lambda(t)}$ and $x(t)$ are bounded independently of $t\in\RR$. By Remark \ref{remvirial},
\begin{equation}
\label{virialonR}
\sigma<\tau\Rightarrow \int_{\sigma}^{\tau} \dd(t)dt\leq C\left(\dd(\sigma)+\dd(\tau)\right).
\end{equation}
Letting $\sigma$ go to $-\infty$ and $\tau$ to $+\infty$, we get that $\dd(t)=0$ for all $t$. Thus $u=W$ up to the invariance of the equation, which contradicts the assumption $\|\nabla u_0\|_2<\|\nabla W\|_2$.
\end{proof}

\begin{proof}[Proof of Lemma \ref{lem.virial}]
Let $R>0$ to be chosen later and $g_{R}$ the function defined by \eqref{defgR}  

\noindent\emph{Step 1. Bound on $g_{R}$.} Let us show that there is a constant $C_0$ independent of $t\geq 0$ such that 
\begin{equation}
\label{boundgRbis}
|g_{R}(t)|\leq C_0 R\dd(t).
\end{equation}

Indeed by the explicit expression of $g_{R}$, the fact that $\psi_{R}\leq 2R$ and $\varphi_{R}\leq 2R/|x|$ and  Hardy's inequality we get 
$$ |g_{R}(t)|\leq C R \|\partial_t u(t)\|_{2}\|\nabla u(t)\|_2\leq CR\|\partial_tu(t)\|_2.$$
By Lemma \ref{lem.est.modul} $
\|\partial_t u\|_2\leq C \dd(t)$
for $t$ such that $\dd(t)\leq \delta_0$. As $\|\partial_t u\|_2$ is bounded by $\sqrt{2E(W)}$ (Claim \ref{usefulclaim}), this bounds is valid for any $t$, which concludes the proof of \eqref{boundgRbis}.

\noindent\emph{Step 2. Bound on $g_{R}'$.} In this step we show that there exist $\rho_0>0$, $c>0$, independent of $\sigma$ and $\tau$ such that if for some $t\in[\sigma, \tau]$,
\begin{equation}
\label{hyp.gR'}
R\geq \rho_0\left(\frac{1}{\lambda(t)}+|x(t)|\right),
\end{equation}
then
\begin{equation}
\label{boundgR'}
g_{R}'(t)\leq -c\dd(t).
\end{equation}
Indeed by Claim \ref{calculs} in the appendix,
$$
g_R'(t)=\frac{1}{N-2}\int |\partial_t u|^2dx-\frac{1}{N-2}\left(\int |\nabla W|^2dx-\int |\nabla u|^2dx\right)+A_{R}(u,\partial_t u),$$
where $A_R$ is defined  in \eqref{defAR}.
We first claim the following bounds on $A_{R}(u,\partial_t u)$:
\begin{gather}
\label{bound.larged}
\forall \eps>0, \, \exists\rho_{\eps}>0,\,\forall t\geq 0,\quad R\geq 2|x(t)|+\frac{\rho_{\eps}}{\lambda(t)}\Longrightarrow |A_{R}(u,\partial_t u)|\leq \eps.\\
\label{bound.smalld}
\exists C_1>0,\;\forall \rho\geq 1,\; \forall t\geq 0,\quad \left[R\geq 2|x(t)|+\frac{2\rho}{\lambda(t)}\text{ and }\dd(t)<\delta_0\right]\\
\notag
\qquad\qquad\Longrightarrow  |A_{R}(u(t),\partial_t u(t))|\leq C_1\left(\frac{1}{\rho^{\frac{N-2}{2}}}\dd(t)+\dd(t)^2\right).
\end{gather}

By \eqref{defAR}, there exists $C_2>0$ such that
\begin{equation}
\label{boundARn}
A_{R}(u,\partial_t u)\leq C_2\int_{\left|x\right|\geq R} r(u)dx,
\end{equation}
where $r(u)$ is defined in \eqref{defe(u)}.
Let $\rho_{\eps}:=2R_0(\eps/C_2)$, where $R_0$ is defined in \eqref{defR_0}. Assume that $\ds R\geq 2|x(t)|+\frac{\rho_{\eps}}{\lambda(t)}$. Then
$$ |x|\geq R\Longrightarrow |x-x(t)|\geq R-|x(t)|\geq \frac{R}{2}\geq\frac{R_0(\eps/C_2)}{\lambda(t)}.$$
By \eqref{boundARn} and the definition of $R_0$, we get \eqref{bound.larged}.

Let us show \eqref{bound.smalld}. Let $t$ such that $\dd(t)< \delta_0$, where $\delta_0$ is the parameter given by \S \ref{sub.modul}. Recall that by \eqref{hyp.equal.modul}, $\lambda(t)=\mu(t)$ and $\xx(t)=x(t)$.

For any $\lambda_0,x_0$, we know that $W_{\lambda_0,x_0}$ is a solution of \eqref{CP} independent of $t$, so that $g_R(t)=0$, and $g'_R(t)=0$, which shows by Claim \ref{calculs} that $A_R(W_{\lambda_0,x_0},0)=0$. Thus
$$ A_{R}(u,\partial_t u)=A_{R}(u,\partial_t u)-A_{R}\left(W_{\frac{1}{\mu},-\xx},0\right).$$
By the change of variable $x=\xx+\frac{y}{\mu}$ we get
\begin{multline}
\label{boundaR}
\int a_{R}^{jk}(x)\partial_j u(x) \partial_k u(x)dx-\int a_{R}^{jk} \frac{\partial}{\partial_{x_j}} \left(W_{\frac{1}{\mu},-\xx}\right)(x)\frac{\partial}{\partial_{x_k}}\left( W_{\frac{1}{\mu},-\xx}\right)(x)dx\\
=\int a_R^{jk}\big(\xx+\frac{y}{\mu}\big)\partial_{j} (W+f)\partial_{k} (W+f)dy-
\int a_{R}^{jk}\big(\xx+\frac{y}{\mu}\big)\partial_j W \partial_k Wdy\\
=\int a_{R}^{jk}\big(\xx+\frac{y}{\mu}\big)\left(\partial_j W\partial_k f+\partial_j f\partial_k W\right) dy+\int a_{R}^{jk}\big(\xx+\frac{y}{\mu}\big)\partial_j f\partial_k fdy.
\end{multline}
where $f=u_{\mu,\xx}-W$, is such that  $\|\nabla f(t)\|_2\leq C_0\dd(t)$ by Lemma \ref{lem.est.modul}. Now, a similar calculation on the other terms of $A_{R}(u,\partial_t u)-A_{R}(W_{\frac{1}{\mu},-\xx})$ yields the bound
\begin{multline}
\label{majorAR1}
\left|A_R(u,\partial_tu)\right|=\left|A_{R}(u,\partial_t u)-A_{R}\Big(W_{\frac{1}{\mu},-\xx},0\Big)\right|\\
\leq C\int_{\left|\xx+\frac{y}{\mu}\right|\geq R} \bigg[|\nabla f|^2+|\nabla W\cdot\nabla f|+ W^{2^*-1} |f|+|f|^{2*}+ \frac{1}{\mu^2\big|\xx+\frac{y}{\mu}\big|^2}\left(W|f|+|f|^2\right)\bigg] dy.
\end{multline}

Let us bound the terms of the right-hand side of \eqref{majorAR1} that are linear in $f$
Recall that $\mu(t)=\lambda(t)$ and $\xx(t)=x(t)$. Using that $R\geq 2|x(t)|+\frac{2\rho}{\mu(t)}$, we get, if $\left|\xx+\frac{y}{\mu}\right|\geq R$
$$\frac{|y|}{\mu}\geq 2\frac{\rho}{\mu}\Longrightarrow |y|\geq 2\rho
\text{ and }\left|\xx+\frac{y}{\mu}\right|\geq \left|\frac{y}{\mu}\right|-\xx\geq \left|\frac{y}{\mu}\right|-\frac{R}{2}\geq \left|\frac{y}{\mu}\right|-\frac{\rho}{\mu}\geq \frac{1}{2}\left|\frac{y}{\mu}\right|.$$
Thus, recalling that $W(y)\approx |y|^{2-N}$ for large $|y|$.
\begin{gather*}
\int_{\left|\xx+\frac{y}{\mu}\right|\geq R}|\nabla W\cdot\nabla f|dy \leq \left(\int_{|y|\geq 2\rho}|\nabla W|^2\right)^{1/2}\|\nabla f\|_{2}\leq \frac{C}{\rho^{\frac{N-2}{2}}}\dd(t)\\
\int_{\left|\xx+\frac{y}{\mu}\right|\geq R}W^{2^*-1} |f|dy\leq \left(\int_{|y|\geq 2\rho} W^{2^*}\right)^{\frac{N+2}{2N}}\|f\|_{2^*}\leq \frac{C}{\rho^{\frac{N+2}{2}}}\dd(t)\\ 
\int_{\left|\xx+\frac{y}{\mu}\right|\geq R} \,\frac{W|f|}{\mu^2\big|\xx+\frac{y}{\mu}\big|^2} dy \leq \int_{\left|\xx+\frac{y}{\mu}\right|\geq R} \,4\frac{W|f|}{\big|y\big|^2}\leq C \left(\int_{|y|\geq 2\rho}\frac{1}{|y|^2}|W|^2\right)^{1/2}\|\nabla f\|_2\leq\frac{C}{\rho^{\frac{N-2}{2}}}\dd(t).
\end{gather*}

By \eqref{majorAR1}, we get \eqref{bound.smalld}.\par

We are now ready to show \eqref{boundgR'}. Note that by Claim \ref{calculs}, we have, for a small constant $\tilde{c}>0$, 
\begin{equation}
\label{bound1gR'}
g_R'(t)\leq -\tilde{c}\dd(t)+|A_R(u,\partial_t u)|.
\end{equation}
Chose $\delta_1:=\min\left\{\delta_0,\frac{\tilde{c}}{4C_1}\right\}$ and $\rho_1>1$ such that $\frac{C_1}{\rho_1^{\frac{N-2}{2}}}\leq \frac{\tilde{c}}{4}$ where $C_1$ is the constant in \eqref{bound.smalld}. By \eqref{bound.smalld}
$$ \dd(t)<\delta_1\text{ and }R\geq 2\Big(|x(t)| +\frac{\rho_1}{\lambda(t)}\Big)\Longrightarrow \left|A_R(u,\partial_t u)\right|\leq \frac{\tilde{c}}{2}\dd(t).$$
According to \eqref{bound.larged} with $\eps:=\frac{\tilde{c}}{2}\delta_1$,
$$ \dd(t)\geq \delta_1\text{ and }R\geq 2|x(t)|+\frac{\rho_{\eps}}{\lambda(t)}\Longrightarrow \left|A_R(u,\partial_t u)\right|\leq \frac{\tilde{c}}{2}\delta_1\leq \frac{\tilde{c}}{2}\dd(t).$$
In view of \eqref{bound1gR'}, we get \eqref{boundgR'} under the assumption \eqref{hyp.gR'} for $\rho_0:=\max(2\rho_1,\rho_{\eps},2)$. Step 2 is complete.

\medskip
\noindent\emph{Step 3. End of the proof.}
Take $$R:=2\rho_0 \sup_{\sigma\leq t\leq \tau}\left(\frac{1}{\lambda(t)}+|x(t)|\right),$$
 where $\rho_0$ is given by Step $2$. Then by \eqref{boundgR'}
$$ \forall t\in  [\sigma,\tau],\quad c\dd(t)\leq -g'_R(t)$$
Integrating between $\sigma$ and $\tau$, we get, in view of \eqref{boundgRbis}
$$ c\int_{\sigma}^{\tau} \dd(t)dt\leq |g_R(\sigma)|+|g_R(\tau)|\leq C_0 R(\dd(\sigma)+\dd(\tau)),$$
which yields the conclusion of Lemma \ref{lem.virial}.
\end{proof}


\begin{proof}[Proof of Lemma \ref{lem.bound.x}]
\emph{Step 1. Bounds by compactness on a short time interval.} We show that there exists $C_1>0$ such that 
\begin{equation}
\label{boundx1}
\forall \tau,\sigma\geq 0,\quad |\tau-\sigma|\leq \frac{1}{\lambda(\tau)}\Longrightarrow \lambda(\tau)|x(\tau)-x(\sigma)|+\frac{\lambda(\tau)}{\lambda(\sigma)}+\frac{\lambda(\sigma)}{\lambda(\tau)}\leq C_1.
\end{equation}
If not, we may find sequences $\tau_n,\sigma_n\geq 0$ such that
\begin{equation}
\label{contrax1}
|\tau_n-\sigma_n|\leq \frac{1}{\lambda(\tau_n)},\quad \lambda(\tau_n)|x(\tau_n)-x(\sigma_n)|+\frac{\lambda(\tau_n)}{\lambda(\sigma_n)}+\frac{\lambda(\sigma_n)}{\lambda(\tau_n)}\underset{n\rightarrow +\infty}{\longrightarrow} +\infty.
\end{equation}
Extracting subsequences, we may assume
\begin{equation}
 \label{lim.s0}
\lim_{n\rightarrow +\infty} \lambda(\tau_n)(\sigma_n-\tau_n)=s_0\in [-1,1].
\end{equation}
Consider the solution of \eqref{CP}
$$ v_n(s,y):=\frac{1}{(\lambda(\tau_n))^{\frac{N-2}{2}}} u\left(\frac{s}{\lambda(\tau_n)}+\tau_n,\frac{y}{\lambda(\tau_n)}+x(\tau_n)\right).$$
By compactness of $\overline{K}$, extracting subsequences if necessary, $\left(v_n,\frac{\partial v_n}{\partial s}\right)_{\restriction s=0}$ has a limit $(v_0,v_1)$ in $\hdot\times L^2$. Let $v$ be the solution of \eqref{CP} with initial condition $(v_0,v_1)$, which is globally defined according to Corollary \ref{corolGlobal}. By Proposition \ref{prop.criterion} \eqref{continuity},
$$ w_{n}(y):=v_n\left(\lambda(\tau_n)(\sigma_n-\tau_n),y\right)=\frac{1}{(\lambda(\tau_n))^{\frac{N-2}{2}}} u\left(\sigma_n,\frac{y}{\lambda(\tau_n)}+x(\tau_n)\right)\underset{n\rightarrow +\infty}{\longrightarrow} v(s_0,y) \text{ in }\hdot.$$
Furthermore the compactness of $\overline{K}$ implies that the following sequence stays inside a compact set of $\hdot$.
$$ \frac{1}{(\lambda(\sigma_n))^{\frac{N-2}{2}}}u\left(\sigma_n,\frac{y}{\lambda(\sigma_n)}+x(\sigma_n)\right)=\left(\frac{\lambda(\tau_n)}{\lambda(\sigma_n)}\right)^{\frac{N-2}{2}}w_n\left(\frac{\lambda(\tau_n)}{\lambda(\sigma_n)}y+\lambda(\tau_n)(x(\sigma_n)-x(\tau_n))\right).$$
Thus
$\frac{\lambda(\tau_n)}{\lambda(\sigma_n)}$, $\frac{\lambda(\tau_n)}{\lambda(\sigma_n)}$ and $\lambda(\tau_n)(x(\tau_n)-x(\sigma_n))$ must be bounded, contradicting \eqref{contrax1}.

\medskip

\emph{Step 2. Control of the variations of $\dd$.} Let $\delta_0>0$ be given by Subsection \ref{sub.modul}. Let us show
\begin{equation}
\label{major.dd}
\exists \delta_1>0,\; \forall \tau\geq 0,\quad \sup_{\tau\leq t\leq \tau+\frac{1}{\lambda(\tau)}} \dd(t)>\delta_0 \Longrightarrow \inf_{\tau\leq t\leq \tau+\frac{1}{\lambda(\tau)}}\dd(t)> \delta_1.
\end{equation}
Indeed, assume that it does not hold, so that (extracting if necessary), we may find sequences $(\tau_n)_n$, $(t_n)_n$, $(t_n')_n$, such that 
\begin{equation}
\label{absurdtn}
t_n,t_n'\in \left[\tau_n,\tau_n+\frac{1}{\lambda(\tau_n)}\right],\quad \dd(t_n)\rightarrow 0\text{ and }\dd(t'_n)> \delta_0.
\end{equation}
Let
$$ v_n(s,y):=\frac{1}{(\lambda(t_n))^{\frac{N-2}{2}}} u\left(\frac{s}{\lambda(t_n)}+t_n,\frac{y}{\lambda(t_n)}+x(t_n)\right).$$
By the compactness of $K$, and the fact that $\dd(t_n)$ tends to $0$, we may assume that $(v_n(0),\partial_sv_n(0))$ tends to some $W_{\lambda_0,x_0}$.\par 
By Step 1, $\frac{1}{\lambda(t_n)}\leq \frac{C_1}{\lambda(\tau_n)}$, thus $\lambda(t_n)(t_n-t'_n)$ is bounded. Extracting if necessary, we may assume $\lim_n \lambda(t_n)(t_n-t'_n)=s_0\in [-1,1]$. By Proposition \ref{prop.criterion} \eqref{continuity},
\begin{equation}
\label{CVvn}
\lim_{n\rightarrow \infty} v_n\big(\lambda(t_n)(t_n-t'_n)\big)=\pm W_{\lambda_0,x_0} \text{ in } \hdot.
\end{equation}
Furthermore, $\ds v_n(\lambda(t_n)(t_n-t'_n))=\frac{1}{\lambda(t_n)^{\frac{N-2}{2}}}u\left(t'_n,\frac{y}{\lambda(t_n)}+x(t_n)\right)$. Thus by \eqref{absurdtn}, $\|\nabla W\|_2^2-\|\nabla v_n\|_2^2> \delta_0$, which contradicts \eqref{CVvn}. Step 2 is complete.

\medskip

\noindent\emph{Step 3. End of the proof}
We first show that , there exists $C>0$ such that
\begin{equation}
\label{weakcontrol}
0\leq \sigma\leq \tsigma\leq \ttau\leq \tau=\sigma+\frac{1}{C_1\lambda(\sigma)} \Longrightarrow |x(\ttau)-x(\tsigma)|+\left|\frac{1}{\lambda(\ttau)}-\frac{1}{\lambda(\tsigma)}\right|\leq C\int_{\sigma}^{\tau} \DD(r)dr,
\end{equation}
where $C_1\geq 1$ is the constant defined in Step 1. Indeed, if $\dd(t)\leq \delta_0$ for $t\in [\sigma,\tau]$, we have by \eqref{hyp.equal.modul} that $x(t)=\xx(t)$ and $\lambda(t)=\mu(t)$ on $[\sigma,\tau]$. Thus by \eqref{est.modul.3} in Lemma \ref{lem.est.modul},
\begin{equation*}
|x(\tsigma)-x(\ttau)|+\left|\frac{1}{\lambda(\tsigma)}-\frac{1}{\lambda(\ttau)}\right|=\left|\int_{\tsigma}^{\ttau} x'(t)dt\right|+\left|\int_{\tsigma}^{\ttau} \frac{\lambda'(t)}{\lambda^2(t)}dt\right|\\
\leq C \int_{\sigma}^{\tau} \dd(t)dt,
\end{equation*}
which yields \eqref{weakcontrol} in this case. The second case is when there exists a $t\in[\sigma,\tau]$ such that $\dd(t)> \delta_0$. By Step 2, we get that $\dd(t)>\delta_1$ for all $t\in [\sigma,\tau]$. Note that by Step 1, $|\tsigma-\ttau|\leq \frac{1}{C_1\lambda(\sigma)}\leq \frac{1}{\lambda(\tsigma)}$, and thus, again by Step 1, $|x(\tsigma)-x(\ttau)|\leq \frac{C_1}{\lambda(\tsigma)}$ and 
$\left|\frac{1}{\lambda(\tsigma)}-\frac{1}{\lambda(\ttau)}\right|=\frac{1}{\lambda(\tsigma)}\left|1-\frac{\lambda(\tsigma)}{\lambda(\ttau)}\right|\leq \frac{2C_1}{\lambda(\tsigma)}$. 
$$ |x(\tsigma)-x(\ttau)|+\left|\frac{1}{\lambda(\tsigma)}-\frac{1}{\lambda(\ttau)}\right|\leq \frac{3C_1}{\lambda(\tsigma)}\leq \frac{3C_1^2}{\lambda(\sigma)}= 3C_1^2|\sigma-\tau|\leq \frac{ 3C_1^2}{\delta_1}\int_{\sigma}^{\tau} \dd(t) dt.$$
The proof of \eqref{weakcontrol} is complete.\par
It is straightforward to deduce the conclusion of Lemma \ref{lem.bound.x} from \eqref{weakcontrol} , dividing the interval $[\sigma,\tau]$ into small subintervals, and we omit the details.
\end{proof}

\section{Supercritical case for $L^2$ solutions}
\label{sec.super}
In this section we study a solution $u$ of \eqref{CP} such that
\begin{gather}
\label{hyp.L2}
u_0 \in L^2\\
\label{hyp.super2}
E(u_0,u_1)=E(W,0),\quad \|\nabla u_0\|_2>\|\nabla W\|_2\\
\label{hyp.noblowup}
T_+(u)=+\infty.
\end{gather}
Our main result is the following.
\begin{prop}
\label{prop.sur.L2}
Let $u$ be a solution of \eqref{CP} with $N\in\{3,4,5\}$ satisfying \eqref{hyp.L2}, \eqref{hyp.super2} and \eqref{hyp.noblowup}. Then $N=5$ and changing $u$ into $-u$ if necessary, there exist $c,C>0$ and $\lambda_0,x_0$ such that
\begin{equation}
\label{CV.sur.L2}
\forall t\geq 0,\quad \|\nabla u(t)-\nabla W_{\lambda_0,x_0}\|_2+\|\partial_t u(t)\|_2\leq Ce^{-ct}.
\end{equation}
\end{prop}
\begin{remark}
In dimension $N=3$ or $N=4$, Proposition \ref{prop.sur.L2} asserts than any solution of \eqref{CP} satisfying \eqref{hyp.L2} and \eqref{hyp.super2} must blow-up in finite time for positive and negative time. We are not able to prove \eqref{CV.sur.L2}. Nevertheless, one can show the weaker property $$\lim_{T\rightarrow +\infty} \frac{1}{T}\int_0^T \dd(t)dt=0.$$
\end{remark}
Let 
\begin{equation*}
y(t):=\int_{\RR^N} (u(t))^2dx,
\end{equation*}
and define $\DD(t)$ by \eqref{defdd}. We first prove the following.
\begin{lemma}
\label{lem.y.y'}
Let $u$ satisfying the assumptions of Proposition \ref{prop.sur.L2}. Then
\begin{gather}
\label{y'<0}
\forall t\geq 0,\quad  y'(t)<0\\
\label{lim.L2}
\lim_{t\rightarrow +\infty} y(t)=y_{\infty}\in (0,+\infty)\\
\label{intdd2}
\int_t^{+\infty}\DD(s)ds \leq Ce^{-ct}.
\end{gather}
\end{lemma}
\begin{corol}
\label{corolblowup-}
Under the assumptions of Proposition \ref{prop.sur.L2}, $T_-(u)<\infty$.
\end{corol}
\begin{proof}[Proof of Corollary \ref{corolblowup-}]
Indeed by \eqref{y'<0}, $y'(t)<0$. But if $T_-(u)=+\infty$, \eqref{y'<0} applied to the solution $u(-t,x)$ of \eqref{CP} (which also satisfies the assumptions of Proposition \ref{prop.sur.L2}) shows that $y'(t)>0$, which is a contradiction.
\end{proof}
\begin{proof}[Proof of Lemma \ref{lem.y.y'}]
By direct calculation (and using equation \eqref{CP} and assumption \eqref{hyp.super2} to compute $y''$)
\begin{align}
\label{calcul.y'}
y'(t)&=2\int_{\RR^N} u(t)\partial_t u(t) dx\\
\label{calcul.y''}
y''(t)&= 2\int (\partial_t u (t))^2-|\nabla u(t)|^2+|u(t)|^{2^*}\\
\notag
&=4\frac{N-1}{N-2}\int (\partial_t u(t))^2+\frac{4}{N-2}\left[\int |\nabla u(t)|^2-\int |\nabla W|^2\right]\geq \DD(t).
\end{align}
Furthermore, by Cauchy-Schwarz inequality, 
\begin{equation}
\label{yy'y''}
y'(t)^2\leq 4\int (u(t))^2 \int (\partial_t u (t))^2\leq \frac{N-2}{N-1} y(t) y''(t).
\end{equation}

\medskip

\noindent\emph{Proof of \eqref{y'<0}.}We argue by contradiction. Note that by Remark \ref{RemPersist}, assumption \eqref{hyp.super2} implies that $\|\nabla u(t)\|_2>\|\nabla W\|_2$ for all $t$. By \eqref{calcul.y''}, $y''(t)>0$ for any $t\geq 0$. Assume that for some $t_0$, $y'(t_0)\geq 0$. 
\begin{equation}
\label{absurd.y'>0}
\forall t>t_0,\quad y'(t)>0.
\end{equation}
Hence by \eqref{yy'y''}, $\frac{N-1}{N-2}\frac{y'}{y}\leq  \frac{y''}{y'}$, which yields by integration
$$ \forall t\geq t_0+1,\quad \frac{y'(t)}{y'(t_0+1)} \geq \left(\frac{y(t)}{y(t_0+1)}\right)^{\frac{N-1}{N-2}},$$
which leads to blow-up in finite time from the fact that $\frac{N-1}{N-2}>1$, contradicting \eqref{hyp.noblowup}.

\medskip

\noindent\emph{Proof of \eqref{lim.L2}.} The function $y$ is positive and, by \eqref{y'<0}, decreasing. Thus
\begin{equation}
\label{lim.L2a}
\lim_{t\rightarrow +\infty} y(t)=y_{\infty}\in [0,+\infty).
\end{equation}

We must show $y_\infty>0$. Let us first show that for $t\geq 0$
\begin{equation}
\label{bnd.dtu}
|y'(t)|\leq C\|\partial_t u(t)\|_2\leq C \DD(t).
\end{equation}
By Cauchy-Schwarz, $|y'(t)|\leq \|u(t)\|_2\|\partial_t u(t)\|_2$. By \eqref{lim.L2a}, $\|u(t)\|_2=\sqrt{y(t)}$ is bounded, which shows the first bound in \eqref{bnd.dtu}. According to Lemma \ref{lem.est.modul}, if $\DD(s)\leq \delta_0$ then $\|\partial_t u(t)\|_2\leq C\DD(t).$
Furthermore, if $\DD(s)\geq \delta_0$, $\|\partial_t u(t)\|_2^2\leq \DD(t)\leq \frac{1}{\delta_0}\DD(t)^2$, hence the bound $\|\partial_t u(t)\|_2\leq C \DD(t)$, which concludes the proof of \eqref{bnd.dtu}.

To show that $y_{\infty}>0$, we argue by contradiction. Assume that $y_{\infty}=0$. By\eqref{bnd.dtu}, 
\begin{equation}
\label{inter1}
y(t)=-(y_{\infty}-y(t))=-\int_t^{+\infty} y'(s) ds\leq C \int_t^{+\infty} \DD(s)ds.
\end{equation}
Note that
\begin{equation}
\label{intdd}
\int_t^{+\infty}\DD(s)ds \leq |y'(t)|.
\end{equation}
Indeed $\int_{t}^T y''(s) ds= y'(T)-y'(t)\leq -y'(t)$, which yields \eqref{intdd} in view of \eqref{calcul.y''}.
Combining \eqref{inter1} and \eqref{intdd}, we get
$$\int_t^{+\infty} \DD(s) ds\leq |y'(t)|\leq 2\|\partial_tu(t)\|_2\big(y(t)\big)^{1/2}\leq C\|\partial_tu(t)\|_2\left(\int_t^{+\infty} \DD(s) ds\right)^{1/2}$$ and thus, by \eqref{bnd.dtu},
$\left(\int_t^{+\infty} \DD(s) ds\right)^{1/2}\leq C \DD(t)$ for $t\geq 0$.
This is a differential inequality of the form $\sqrt{Y}\leq -CY'$, which can not be valid on $[0,\infty)$ if $\forall t\geq 0,\;Y>0$. The proof of \eqref{lim.L2} is complete.

\medskip
\noindent\emph{Proof of \eqref{intdd2}.}
By \eqref{bnd.dtu} and \eqref{intdd} ,
$$ \forall t\geq 0,\quad \int_t^{\infty} \DD(s)ds \leq |y'(t)|\leq C\DD(t),$$
which implies \eqref{intdd2}.
\end{proof}

\begin{proof}[Proof of Proposition \ref{prop.sur.L2}.]
\noindent\emph{Step 1. Convergence in $L^2$.} We first show that there exits $u_{\infty}\in L^2$ such that
\begin{equation}
\label{limL2}
\lim_{t\rightarrow +\infty} \|u(t)-u_{\infty}\|_2=0.
\end{equation}
Indeed we have, if $0\leq t_1<t_2$,
\begin{equation}
\label{uCauchy1}
|u(t_1,x)-u(t_2,x)|^2=\left|\int_{t_1}^{t_2} \partial_t u(t,x) dt\right|^2\leq (t_2-t_1)\int |\partial_t u(t,x)|^2 dt.
\end{equation}
Integrating \eqref{uCauchy1} in space, we get by \eqref{intdd2}
$$ \|u(t_1)-u(t_2)\|^2_2\leq |t_1-t_2|\int_{t_1}^{t_2}\|\partial_t u(t)\|_2^2dt\leq C|t_1-t_2|\int_{t_1}^{t_2}d(t)dt\leq C|t_1-t_2|e^{-ct_1}.$$
By an elementary summation argument, we obtain, taking a larger constant $C$, the bound $\|u(t_1)-u(t_2)\|_2^2\leq Ce^{-ct1}$ for $t_1<t_2$. Thus $u$ satisfies the Cauchy criterion in $L^2$ as $t\rightarrow +\infty$, which yields \eqref{limL2}.

\medskip

\noindent\emph{Step 2. End of the proof}
By \eqref{intdd2}, there exists a sequence $t_n\rightarrow \infty$ such that $\DD(t_n)$ tends to $0$. Thus, extracting a subsequence and changing $u$ into $-u$ if necessary, there exists $\lambda_{0}$,\,$x_{0}$ such that $u(t_n)$ tends to $W_{\lambda_0,x_0}$ in $\hdot$, thus in $\DDD'(\RR^N)$. In view \eqref{limL2}, $u(t_n)$ tends also to $u_{\infty}$ in $\DDD'(\RR^n)$. Thus $W_{\lambda_0,x_0}=u_{\infty}\in L^2$. This shows that $N=5$ and 
\begin{equation}
\label{limL2bis}
\lim_{t\rightarrow +\infty} \|u(t)-W_{\lambda_0,x_0}\|_2=0.
\end{equation}
Let us show
\begin{equation}
\label{Dt0}
\lim_{t\rightarrow +\infty} \DD(t)=0.
\end{equation}
If \eqref{Dt0} does not hold, there exist increasing sequences $(t_n)_n$, $(t'_n)_n$ such that $t_n<t'_n$, $\DD(t_n)\rightarrow 0$, $\DD(t'_n)=\delta_0$ and $\DD(t)<\delta_0$ for $t\in [t_n,t'_n)$. On $[t_n,t'_n]$, the modulation parameters $\mu(t)$ and $X(t)$ are well-defined. Furthermore, by \eqref{limL2bis}, $\mu(t)$ must be bounded on $\bigcup_n [t_n,t'_n]$. Thus by \eqref{intdd2} and the same argument as in the proof of \eqref{limite0}, $\alpha(t'_n)$ tends to $0$  which contradicts the estimate $\DD(t'_n)\approx \alpha(t_n')$ of Lemma \ref{lem.est.modul}. Hence \eqref{Dt0}.

Thus there exists $T>0$ such that for $t\geq T$, $\DD(t)\leq \delta_0$. By \eqref{limL2bis}, $\mu$ converges to $\lambda_0^{-1}$. In view of estimate \eqref{est.modul.3} of Lemma \ref{lem.est.modul} and the boundedness of $\mu$,
\begin{equation}
\label{bnd.deriv}
|\alpha'(t)|+|\mu'(t)|+|X'(t)|\leq C \dd(t).
\end{equation}
In view of \eqref{intdd2}, this shows as in the end of the proof of Proposition \ref{prop.CVexpu} that $\dd(t)$, $\alpha(t)$, $\mu(t)$ and $X(t)$ converges exponentially when $t\rightarrow +\infty$, which implies \eqref{CV.sur.L2}. The proof of Proposition \ref{prop.sur.L2} is complete.
\end{proof}

\section{Preliminaries on the linearized equation near $W$}
\label{sec.lin}
This section is similar to the corresponding one in the NLS case \cite[Section 5]{DuMe07P}.

Let $u$ be a solution of \eqref{CP}, defined on $[0,+\infty)$, and close to $W$. Let $h:=u-W$. Then
$$ \partial_t^2 h-\Delta h-|W+h|^{\frac{4}{N-2}}(W+h) +W^{\frac{N+2}{N-2}}=0,$$
which we rewrite as
\begin{gather}
\label{eq.h}
\partial_t^2 h+L h=R(h),\\ 
\notag
L:=-\Delta- \frac{N+2}{N-2}W^{\frac{4}{N-2}},\quad 
R(h):=|W+h|^{\frac{4}{N-2}}(W+h)-W^{\frac{N+2}{N-2}}-\frac{N+2}{N-2}W^{\frac{4}{N-2}} h.
\end{gather}
Note that
$\frac{1}{2}(Lu,u)_{L^2}=Q(u)$
where $Q$ is the quadratic form defined in \eqref{defQ}. 
\subsection{Preliminary estimates}
Recall the definition of the spaces $\ell(I)$ and $N(I)$ defined in \eqref{defS}, \eqref{defl}.
\begin{lemma}
\label{lem.est}
There exists $C>0$ such that if $f\in L^{2*}$, $I$ is a time interval and $u,v\in \ell(I)$. 
\begin{align}
\label{Lest}
&\|D_x^{1/2}(W^{\frac{4}{N-2}}u)\|_{N(I)}\leq C \left(|I|^{\frac{2}{N+1}}+|I|^{\frac{5}{2(N+1)}}\right)\|u\|_{\ell(I)}\\
\label{NLest0}
&\|R(f)\|_{\frac{2N}{N+2}}\leq C\left(\|f\|_{2^*}^{2}+\|f\|_{2^*}^{\frac{N+2}{N-2}}\right)\\
\label{NLest}
\|D_x^{1/2}(R(u)&-R(v))\|_{N(I)} \\
\notag
&\leq C\left(1+|I|^{\frac{6-N}{2(N+1)}}\right) \|u-v\|_{\ell(I)}\left(\|u\|_{\ell(I)}+\|v\|_{\ell(I)}+\|u\|^{\frac{4}{N-2}}_{\ell(I)}+\|v\|^{\frac{4}{N-2}}_{\ell(I)}\right).
\end{align}
\end{lemma}
We postpone the proof of Lemma \ref{lem.est} to Appendix \ref{appendixest}.

We will also need the following version of Lemma \ref{lem.est} with exponentially decreasing norms.
\begin{corol}
\label{corol.est.expo}
Let $u,v\in \ell(t_0,+\infty)$, $t_0\in \RR$, such that for some $\gamma>0$, and some constant $M>0$,
\begin{equation*}
\forall t\geq t_0,\quad \|u\|_{\ell(t,+\infty)}+\|v\|_{\ell(t,+\infty)}\leq Me^{-\gamma t}.
\end{equation*}
Then there exists $C=C(\gamma,M)>0$ such that 
\begin{gather*}
\|D_x^{1/2}(W^{\frac{4}{N-2}}u)\|_{N(t,+\infty)}\leq C e^{-\gamma t},\quad 
\|R(u(t))\|_{\frac{2N}{N+2}}\leq Ce^{-2\gamma t}\\
\forall t\geq t_0,\quad \|D_x^{1/2}(R(u)-R(v))\|_{N(t,+\infty)} \leq C e^{-\gamma t}\|u-v\|_{\ell(t,+\infty)}.
\end{gather*}
\end{corol}
\begin{proof}
This is an immediate consequence of Lemma \ref{lem.est} and the following elementary Claim, which is Claim 5.8 in \cite{DuMe07P}:
\begin{claim}[Sums of exponential]
\label{summation}
Let $t_0>0$, $p\in [1,+\infty[$, $a_0 \neq 0$, $E$ a normed vector space, and $f\in L^p_{\rm loc}(t_0,+\infty;E)$ such that 
\begin{equation}
\label{small.tau}
\exists \tau_0>0,\; \exists C_0>0, \;\forall t\geq t_0, \quad \|f\|_{L^p(t,t+\tau_0,E)}\leq C_0e^{a_0 t}.
\end{equation}
Then for $t\geq t_0$,
\begin{equation}
\label{conclu.summation}
\|f\|_{L^p(t,+\infty,E)}\leq \frac{C_0e^{a_0 t}}{1-e^{a_0\tau_0}}\; \text{if } a_0<0;\quad \|f\|_{L^p(t_0,t,E)}\leq  \frac{C_0e^{a_0 t}}{1-e^{-a_0\tau_0}}\; \text{if }a_0>0.
\end{equation}
\end{claim}

\end{proof}
\begin{corol}[Strichartz estimates for the perturbative equation]
\label{Stri.expo}
Let $h$ be a solution of \eqref{eq.h} on $[0,\infty)$ such that
$$ \|\nabla h(t)\|_{2}+\|\partial_t h(t)\|_2\leq Ce^{-\gamma t}.$$
Then
\begin{gather*}
\|h\|_{\ell(t,+\infty)}+\|D_x^{1/2}W^{\frac{4}{N-2}} h\|_{N(t,+\infty)} \leq Ce^{-\gamma t},\\ \|R(h(t))\|_{\frac{2N}{N+2}}+\|D_x^{1/2} R(h)\|_{N(t,+\infty)}\leq Ce^{-2\gamma t}.
\end{gather*}
\end{corol}
\begin{proof}
The proof is the same than the one of \cite[Lemma 5.6]{DuMe07P}. We sketch it for the sake of completness. Note that all desired estimates are, by Corollary \ref{corol.est.expo}, a consequence of 
\begin{equation*}
\|h\|_{\ell(t,+\infty)}\leq Ce^{-\gamma t},
\end{equation*}
so that we only need to show this last estimate
We have
$$ \partial_t^2 h-\Delta h=\frac{N+2}{N-2} W^{\frac{4}{N-2}}h+R(h).$$
Let $t>0$ and $\tau\in (0,1)$. First note that $W+h$ is solution of \eqref{CP}, and thus, by the standard Cauchy problem theory for \eqref{CP}, $\|h\|_{\ell(t,t+\tau)}$ is finite. By Strichartz inequality (Proposition \ref{prop.Strichartz}) and Lemma \ref{lem.est}, 
\begin{align*}
\|h\|_{\ell(t,t+\tau)}&\leq C\left( \|D_{x}^{1/2}W^{\frac{4}{N-2}}h\|_{N(t,t+\tau)}+\|D_x^{1/2}R(h)\|_{N(t,t+\tau)}+e^{-{\gamma t}}\right)\\
&\leq C\left( \tau^{\frac{2}{N+1}} \|h\|_{\ell(t,t+\tau)}+\|h\|^{2}_{\ell(t,t+\tau)}+\|h\|^{\frac{N+2}{N-2}}_{\ell(t,t+\tau)}+e^{-{\gamma t}}\right).
\end{align*}
By a standard argument (see the proof of \cite[Lemma 5.7]{DuMe07P}), we deduce from the preceding inequality, $\tau$ is small, 
$$\|h\|_{\ell(t,t+\tau)}\leq Ce^{-\gamma t}.$$
The conclusion follows from Claim \ref{summation}.
\end{proof}
\subsection{Spectral theory for the linearized operator}
The following Proposition sums up spectral properties of $L$ (see \cite{KrSc05}, \cite{KrScTa07P} for the radial case in $\RR^3$).
\begin{prop}
\label{prop.spectral}
The operator $L$ on $L^2$ with domain $H^2$ is a self-adjoint operator with essential spectrum $[0,+\infty)$, no positive eigenvalue and only one negative eigenvalue $-e_0^2$, with a radial, exponentially decreasing, smooth eigenfunction $\YYY$. Furthermore, if 
$$ G_{\bot}:=\left\{ f\in \hdot, \quad \int \YYY f=\int \nabla f\cdot\nabla \tW=\int \nabla f\cdot \nabla W_1=\ldots=\int \nabla f\cdot\nabla W_N=0\right\}.$$
Then there exists $c_Q>0$ such that
\begin{equation}
\label{QsurG}
\forall f\in G_{\bot}, \quad Q(f)\geq c_Q\|\nabla f\|^2_2.
\end{equation}
\end{prop}
\begin{remark}
The proposition shows 
$$ \{u\in \hdot,\; Lu=0\}=\vect\{\tW,W_1,\ldots,W_N\}.$$
Indeed, the inclusion $\supset$ is already known. For the other inclusion, note that if the dimension of the space $\{u\in \hdot,\; Lu=0\}$ was strictly higher than $N+1$, we could find $\ZZZ\in \hdot$, $\ZZZ\neq 0$, such that $L\ZZZ=0$, and orthogonal to $\tW,W_1,\ldots,W_N$. By self-adjointness of $L$, $\int \ZZZ\YYY=-\frac{1}{e_0}\int \ZZZ L\YYY=0$. Thus on one hand $\ZZZ\in G_{\bot}\setminus \{0\}$, and on the other $Q(\ZZZ)=(L\ZZZ,\ZZZ)=0$ contradicting \eqref{QsurG}.
\end{remark}

\begin{proof}[Proof of the proposition]
\emph{Step 1. Existence of a negative eigenvalue.}
The fact that $L$ is a self-adjoint operator on $L^2$ with domain $H^2$ is well-known. Note that $L=-\Delta - \frac{N+2}{N-2}W^{\frac{4}{N-2}}$ with $W^{\frac{4}{N-2}}\approx 1/|x|^4$ for large $x$. In particular, $W^{\frac{4}{N-2}}$ is bounded and tends to $0$ at infinity, which shows that the essential spectrum of $L$ is $[0,\infty)$ (see e.g. \cite[Theorem XIII.14]{ReSi.T4}). Furthermore, $|x|W^{\frac{4}{N-2}}$ tends to $0$ at infinity, so that by Kato's Theorem \cite[Theorem XII.58]{ReSi.T4}, $L$ does not have any positive eigenvalue. 

By the equation $-\Delta W=W^{\frac{N+2}{N-2}}$, we have $LW=-\frac{4}{N-2}W^{\frac{N+2}{N-2}}$ and thus
$$\int LW W=-\int \frac{4}{N-2} W^{2^*}<0.$$
By approximating $W$ by compactly supported functions 
$$\inf_{\|u\|_2=1,u\in H^2} \int Lu u<0,$$ 
Thus $L$ has at least one negative eigenvalue $-e_0^2$, and the corresponding eigenfunction $\YYY$ is exponentially decreasing by Agmon estimate. We chose $-e_0^2$ to be the first eigenvalue of $L$, which implies that $\YYY$ is radial and $-e_0^2$ is a simple eigenvalue\par 

\noindent\emph{Step 2. Proof of \eqref{QsurG}.}
We first show 
\begin{equation}
\label{QsurG1}
\forall f\in G_{\bot}, \quad Q(f)>0.
\end{equation}
If not  $\exists f\in G_{\bot},\quad Q(f)\leq 0.$
Let $g\in \vect\left\{ f,\YYY,\tW,W_1,\ldots, W_N\right\}$. Taking into account that $L \tW=LW_1=\ldots=LW_N=0$, and that $\int L\YYY f=-e_0^2\int \YYY f=0$ we get
$$ Q(g)=\int Lg\,g=\int L(\alpha \YYY+\beta f)(\alpha \YYY+\beta f)=\beta^2 Q(f)-\alpha^2 e_0^2\int \YYY^2\leq 0.$$
Note that $\vect\left\{ f,\YYY,\tW,W_1,\ldots, W_N\right\}$ is a subspace of $\hdot$ of dimension $N+3$, whereas $H^{\bot}=\vect\{W,\tW,W_1,W_2,\ldots,W_N\}^{\bot}$ (the orthogonal is taken in $\hdot$) is of codimension $N+2$ in $\hdot$. Thus there exists a nonzero $g\in \vect\left\{ f,\YYY,\tW,W_1,\ldots, W_N\right\}\cap H^{\bot}$. By Claim \ref{coercivity}, $Q(g)>0$, whereas we have just shown that $Q(g)\leq 0$ yielding a contradiction. The proof of \eqref{QsurG1} is complete.

We now turn to the proof of \eqref{QsurG}. We argue again by contradiction. If \eqref{QsurG} does not hold, there exists a sequence $(f_n)$ such that
\begin{equation}
\label{absurdfn}
f_n\in G_{\bot},\quad \|\nabla f_n\|_{2}=1,\quad Q(f_n)\underset{n\rightarrow \infty}{\longrightarrow}0.
\end{equation}
Extracting a subsequence from $(f_n)$, we may assume 
$$ f_n \rightharpoonup f\text{ in  } \hdot.$$
The weak convergence of $f_n\in G_{\bot}$ to $f$ implies that $f\in G_{\bot}$. 
By Cauchy-Schwarz inequality for the positive quadratic form $Q$ on $G_{\bot}$, we get $\sqrt{Q(f_n)Q(f)}\geq \left|\int Lf\,f_n\right|$. Thus by \eqref{absurdfn},
$$ Q(f)=\lim_{n\rightarrow +\infty} \int Lf\,f_n=0.$$
As $f\in G_{\bot}$, \eqref{QsurG1} shows that $f=0$ and thus
$$ f_n \rightharpoonup 0 \text{ in }\hdot.$$
Now, by compactness $\int W^{\frac{4}{N-2}} |f_n|^2$ tends to $0$. Using that by \eqref{absurdfn}, $Q(f_n)$ tends to $0$, we get that $\|\nabla f_n\|_2$ tends to $0$, contradicting \eqref{absurdfn}. The proof of \eqref{QsurG} is complete.

\medskip

\noindent\emph{Step 3. Uniqueness of the negative eigenvalue.}
Assume that $L$ has a second eigenfunction $\YYY_1$, with eigenvalue $-e_1^2\leq 0$. As $-e_0^2$ is the first eigenvalue of $L$, we have that $-e_0^2<-e_1^2$ and $\int \YYY\YYY_1=0$. The same argument than above shows that 
$$ \forall f\in \vect\{ \YYY,\YYY_1,\tW,W_1,\ldots W_N\},\quad Q(f)\leq 0,$$
which yields a subspace of $\hdot$ of dimension $N+3$ where $Q$ is nonpositive, contradicting the fact that $Q$ is positive on the subspace $H^{\bot}$, which is of codimension $N+2$ in $\hdot$.
\end{proof}
In the sequel, we will chose $\YYY$ such that 
\begin{equation}
\label{normalisation2}
\int \YYY^2=1.
\end{equation}
\subsection{Properties of the nonhomogeneous linearized equation}

Let $t_0\geq 0$. We are now interested by the following problem
\begin{equation}
\label{eq.linear}
 \partial_t^2 h+L h=\eps, \quad t\geq t_0, 
\end{equation}
where $h\in C^0\big([t_0,+\infty),\hdot\big)$, $\partial_t h\in C^0\big([t_0,+\infty),L^2\big)$, $\eps\in C^0\left([t_0,+\infty),L^{\frac{2N}{N+2}}\right)$ and $D_x^{1/2}\eps\in N(t_0,+\infty)$ (see \eqref{defS} for the definition of $N(t_0,+\infty)$).
\begin{prop}
\label{prop.linear}
Let $h$ and $\eps$ be as above. Assume that for some constant $c_0,c_1$ such that $0<c_0<c_1$,
\begin{align}
\label{hyp.h}
\|\partial_t h(t)\|_2+\|\nabla h(t)\|_2&\leq Ce^{-c_0t}\\
\label{hyp.eps}
\|\eps(t)\|_{\frac{2N}{N+2}} +\|D_x^{1/2}\eps\|_{N(t,+\infty)}&\leq Ce^{-c_1t}.
\end{align}
Then, if $c_1^-$ is any arbitrary number $<c_1$.
\begin{itemize}
\item if $c_1>e_0$, there exists $A\in \RR$ such that
\begin{equation}
\label{dec.h1}
\left\|\partial_t \left(h(t)-Ae^{-e_0t}\YYY\right)\right\|_2+\left\|\left(\nabla (h(t)-Ae^{-e_0t}\YYY\right)\right\|_2\leq Ce^{-c_1^- t};
\end{equation}
\item if $c_1\leq e_0$, 
\begin{equation}
\label{dec.h2}
\|\partial_t h(t)\|_2+\|\nabla h(t)\|_2\leq Ce^{-c_1^- t}.
\end{equation}
\end{itemize}
\end{prop}
\begin{proof}
 Write 
$$ h(t)=\beta(t)\YYY+\tgam(t)\tW+\sum_{j=1}^N \gamma_j(t)W_j +g(t),\quad g(t)\in G_{\bot}.$$
By the definition of $G_{\bot}$, the condition $g(t)\in G_{\bot}$ is equivalent to
\begin{equation}
\label{betagamma}
\beta(t):=\int h(t)\YYY,\;
\tilde{\gamma}(t):=\int \nabla \left(h(t)-\beta(t)\YYY\right)\nabla \tW,\;
\gamma_j(t):=\int \nabla \left(h(t)-\beta(t)\YYY\right)\nabla W_j.
\end{equation}

\medskip

\noindent\emph{Step 1. Reduced case.}

In this case, we assume, in addition to the hypothesis of Proposition \ref{prop.linear}
\begin{equation}
\label{beta0}
\forall t\geq t_0,\quad \beta(t):=\int h(t)\YYY=0.
\end{equation}
And show that \eqref{dec.h1} (with $A=0$) or \eqref{dec.h2} hold. It is sufficient to show
\begin{equation}
\label{butstep2}
\|\partial_t h(t)\|_2+\|\nabla h(t)\|_2\leq Ce^{-\frac{(c_0+c_1)}{2}t}.
\end{equation}
An iteration argument will give the desired result. 

We first prove
\begin{equation}
\label{eq.Q}
\frac 12 \frac{d}{dt} \left(Q(h(t))+\|\partial_t h(t)\|_2^2\right)=\int_{\RR^N} \eps(t)\partial_t h(t).
\end{equation}
Indeed, recalling that $Q(h)=\int L h\,h$, we get
$$\frac 12 \frac{d}{dt} \left(Q(h(t))+\|\partial_t h(t)\|_2^2\right)=\int_{\RR^N} Lh(t)\partial_t h(t)+\int_{\RR^N} \partial_t^2 h(t)\partial_t h(t),$$
which gives directly \eqref{eq.Q} from equation \eqref{eq.linear}, .

We now turn to the proof of \eqref{butstep2}. Note that $h$ is exponentially decreasing in the Strichartz norms:
\begin{equation}
 \label{stri.h1}
\|h\|_{\ell(t,+\infty)}\leq Ce^{-c_0 t}
\end{equation}
Indeed $\partial_t^2h-\Delta h=\frac{N+2}{N-1} W^{\frac{4}{N+2}}h+\eps$ and by Corollary \ref{corol.est.expo}, assumptions \eqref{hyp.h} and \eqref{hyp.eps},
\begin{equation*}
\left\| D_x^{1/2}\left(\frac{N+2}{N-1} W^{\frac{4}{N+2}}h+\eps\right)\right\|_{N(t,+\infty)}\leq Ce^{-c_0 t}.
\end{equation*}
By Strichartz estimates (see Proposition \ref{prop.Strichartz}), we get \eqref{stri.h1}.

Now, by \eqref{eq.Q},
$$ \left|\frac{d}{dt}\left(Q(h(t))+\|\partial_t h(t)\|_2^2\right)\right|\leq C\|D_x^{1/2}\eps(t)\|_{\frac{2(N+1)}{N+3}}\|D_x^{-1/2}\partial_t h(t)\|_{\frac{2(N+1)}{N-1}}.$$
Integrating between $t$ and $+\infty$, we get, combining assumption \eqref{hyp.eps} on $\eps$, estimate \eqref{stri.h1}, and H\"older inequality in time,
$$ Q(h(t))+\|\partial_t h(t)\|_2^2\leq C\|D^{1/2}_x \eps\|_{N(t,+\infty)} \|h\|_{\ell(t,+\infty)}\leq C e^{-(c_0+c_1)t}.$$
By Claim \ref{coercivity}, and \eqref{beta0}, $\|\nabla g(t)\|_2^2\leq CQ(g(t))=CQ(h(t))\leq Ce^{-(c_0+c_1)t}$. Hence
\begin{equation}
\label{ineggh1}
\|\partial_t h(t)\|_2+\|\nabla g(t)\|_2\leq C e^{-\frac{(c_0+c_1)}{2} t}.
\end{equation}

Furthermore, note that by the definition of $\tgam$ in \eqref{betagamma} and \eqref{ineggh1}. 
$$\tilde{\gamma}(t)=\int_{\RR^N} h(t)\Delta \tW\underset{t\rightarrow+\infty}{\longrightarrow 0},\quad |\tilde{\gamma}'(t)|=\left|\int_{\RR^N} \partial_t h(t)\Delta \tW\right|\leq Ce^{-\frac{c_0+c_1}{2}}.$$
Hence
\begin{equation}
\label{inegtgam}
|\tgam(t)|\leq Ce^{-\frac{(c_0+c_1)t}{2}}.
\end{equation}
By an analogous argument
\begin{equation}
\label{ineggamj}
\sum_{j=1}^N |\gamma_j(t)|\leq Ce^{-\frac{(c_0+c_1)t}{2}}.
\end{equation}
This gives \eqref{butstep2} and concludes Step 1.

\medskip

\noindent\emph{Step 2. General case.}
We no longer assume $\beta(t)=0$. We have:
\begin{align}
\label{eq.beta}
 \beta''(t)&=e_0^2\beta(t)+ \eta(t),\text{ where }\eta(t):=\int_{\RR^N} \eps(t)\YYY.
\end{align}
Indeed,
$$\beta''(t)=\int_{\RR^N} \partial_t^2h(t)\YYY=-\int_{\RR^N} Lh(t) \YYY+\int_{\RR^N} \eps(t)\YYY=e_0^2\int_{\RR^N} h(t)\YYY+\eta(t)\YYY.$$
We will show that $\tilde{h}(t):=h(t)-\beta(t)\YYY$ and $\tilde{\eps}(t):=\eps(t)-\eta(t)\YYY$  satisfy the hypothesis of Step 1. 

By assumption \eqref{hyp.eps}, $\left|\eta(t)\right|\leq Ce^{-c_1t}.$ We distinguish two cases.

\noindent\emph{First case: $e_0<c_1$.}
Then $e^{e_0t} |\eta(t)|\leq Ce^{-(c_1-e_0)t}$ with $c_1-e_0>0$. Solving \eqref{eq.beta}, we see that there exist real parameters $\beta_+$, $\beta_-$ such that 
$$ \beta(t)=\beta_-e^{-e_0t}+\beta_+ e^{e_0t} -\frac{1}{2e_0} \int_t^{+\infty} e^{e_0(t-s)}\eta(s)ds+\frac{1}{2e_0} \int_t^{+\infty} e^{-e_0(t-s)}\eta(s)ds.$$
Note that 
$\left|\int_t^{+\infty} e^{e_0(t-s)}\eta(s)ds\right|+\left| \int_t^{+\infty} e^{-e_0(t-s)}\eta(s)ds\right|\leq Ce^{-c_1 t}.$
Furthermore $\beta(t)$ tends to $0$ by \eqref{hyp.h} which shows that 
$\beta_+=0.$ Hence
\begin{equation}
 \label{beta1}
\beta(t)=\beta_- e^{-e_0t}+O(e^{-c_1 t}).
\end{equation}

\noindent\emph{Second case: $c_1\leq e_0$.}
Solving again \eqref{eq.beta}, we get real parameters $\beta_+$, $\beta_-$ such that 
$$ \beta(t)=\beta_-e^{-e_0t}+\beta_+ e^{e_0t} -\frac{1}{2e_0} \int_t^{+\infty} e^{e_0(t-s)}\eta(s)ds-\frac{1}{2e_0} \int_0^{t} e^{-e_0(t-s)}\eta(s)ds.$$
Note that 
$$\left|\frac{1}{2e_0} \int_t^{+\infty} e^{e_0(t-s)}\eta(s)ds+\frac{1}{2e_0} \int_0^{t} e^{-e_0(t-s)}\eta(s)ds\right|\leq C \left\{ \begin{aligned} e^{-c_1 t}& \text{ if }c_1<e_0\\
te^{-c_1 t}& \text{ if }c_1=e_0 \end{aligned}\right.,$$
so that we must have again $\beta_+=0$. As a conclusion
\begin{equation}
 \label{beta2}
c_1<e_0\Longrightarrow \beta(t)=O(e^{-c_1 t}),\quad c_1=e_0\Longrightarrow \beta(t)=O(te^{-c_1 t}).
\end{equation}

In view of \eqref{eq.beta}, it is easy to check that in both cases, $\tilde{h}$ and $\tilde{\eps}$ satisfy the assumptions of Step 1, which implies, together with \eqref{beta1} or \eqref{beta2}, the conclusions \eqref{dec.h1} or \eqref{dec.h2} of Proposition \ref{prop.linear}.
\end{proof}

\section{Proof of main results}
\label{sec.proofs}
In this section we conclude the proofs of Theorem \ref{th.exist} an \ref{th.classif}. We start, in Subsection \ref{sub.approx}, by constructing approximate solutions $U^a_k$ of \eqref{CP} which converge to $W$ as $t\rightarrow +\infty$. Subsection \ref{sub.contract} is devoted to a fixed point argument near $U^a_k$ for large $k$. The proof of Theorems \ref{th.exist} and \ref{th.classif} is the object of Subsection \ref{sub.conclu}, except for the blow-up of $W^+$ for negative times, which is shown in Subsection \ref{subW+}.

\subsection{A family of approximate solutions converging to $W$}
\label{sub.approx}
\begin{lemma}
\label{lem.approximate}
Let $a\in\RR$. There exist functions $(\Phi^{\aexp}_{j})_{j\geq 1}$ in $\SSS(\RR^N)$, such that $\Phi^{\aexp}_1=\aexp \YYY$ and if 
\begin{equation}
\label{defWk}
U^{\aexp}_k(t,x):=W(x)+\sum_{j=1}^k e^{-je_0t}\Phi^{\aexp}_j(x),
\end{equation}
then as $t\rightarrow +\infty$,
\begin{equation}
\label{eqWk}
\eps_k:=\partial_t^2 U^{\aexp}_{k}-\Delta U_k^{\aexp}-\big|U_k^{\aexp}\big|^{\frac{4}{N-2}}U_k^{\aexp}=O(e^{-(k+1)e_0 t}) \text{ in }\SSS(\RR^N).
\end{equation}
\end{lemma}
\begin{proof}

\bigskip

The proof is identical to the proof of \cite[Lemma 6.1]{DuMe07P}. We sketch it for the sake of completness. Note that 
\begin{equation}
\label{eq.h2}
\partial_t^2 (W+h)-\Delta (W+h)+|W+h|^{\frac{4}{N-2}}(W+h)=\partial_t^2 h+Lh-R(h)
\end{equation}
where $L$ and $R$ are defined in \eqref{eq.h}. We have
$$ R(h)= W^{\frac{N+2}{N-2}} J(W^{-1} h),\quad J(t):=|1+t|^{\frac{4}{N-2}}(1+t)-1-\frac{N+2}{N-2}t.$$
The function $J$ is real analytic for $|t|<1$ and $J(0)=J'(0)=0$. Thus $J(t)$ is a power series of the form $\sum_{j\geq 2} c_j t^j$ which has radius of convergence $1$. In particular, if $h\in \SSS(\RR^N)$ and satisfies $|h(x)W^{-1}(x)|\leq 1/2$, for all $x\in \RR^N$, then
\begin{equation}
\label{DSE}
R(h)=\sum_{j=2}^{+\infty} c_j W^{\frac{N+2}{N-2}}\left(hW^{-1}\right)^j,
\end{equation}
where the series converges in $\SSS(\RR^N)$. 

Let us fix $a\in \RR$. We will omit most superscripts $a$ to simplify notations. Clearly, by \eqref{eq.h2}, if $U_1=W+a\YYY e^{-e_0t}$,
$$ \partial_t^2 U_{1}-\Delta U_{1}-\big|U_1\big|^{\frac{4}{N-2}}U_1=e_0^2 ae^{-e_0t}\YYY+ a e^{-e_0t} L(\YYY)-R\left(a\YYY e^{-e_0t}\right)=-R\left(a\YYY e^{-e_0t}\right),$$
which yields \eqref{eqWk} for $k=1$.

Now assume that for some $k\geq 1$, there exist $\Phi_1$,\ldots,$\Phi_k$ in $\SSS(\RR^N)$ such that \eqref{eqWk} holds with 
\begin{equation}
\label{defhk}
U_k=W+h_k,\text{ where } h_k:=\sum_{j=1}^ke^{-je_0t} \Phi_j.
\end{equation}
Let $\Phi_{k+1}\in \SSS(\RR^N)$. Then
\begin{multline}
\label{grosseequation}
\partial_t^2 \left(h_{k}+e^{-(k+1)e_0t}\Phi_{k+1}\right)+L\left(h_k+e^{-(k+1)e_0t}\Phi_{k+1}\right)-R\left(h_k+e^{-(k+1)e_0t}\Phi_{k+1}\right)\\
=\big((k+1)^2e_0^2 +L\big)e^{-(k+1)e_0 t}\Phi_{k+1}+\eps_k+R(h_k)-R\big(h_k+e^{-(k+1)e_0t}\Phi_{k+1}\big).
\end{multline}
By \eqref{DSE} we see that $\eps_k$ must be, for large $t>0$, an infinite sum of the form $\sum_{j\geq 0} e^{-je_0 t}\Psi_{j,k}(x)$, with convergence in $\SSS(\RR^N)$. Furthermore, the induction hypothesis \eqref{eqWk} shows that $\Psi_{j,k}=0$ for $j\leq k$. Thus
$$ \eps_k(t,x)=\sum_{j\geq k+1} e^{-je_0 t}\Psi_{j,k}(x).$$
Furthermore, $R(h_k)$ and $R\left(h_{k}+e^{-(k+1)e_0t}\right)$ have similar asymptotic developments, and if $t$ is large enough, using that $h_k=O(e^{-e_0t})$, we get $\left|R(h_k)-R(h_k+e^{-(k+1)e_0t})\right|\leq Ce^{-e_0(k+2)t}$. This shows
\begin{equation}
\label{petiteequation}
 \eps_k+R(h_k)-R\left(h_k+e^{-(k+1)e_0t}\Phi_{k+1}\right)=e^{-(k+1)e_0t}\Psi_{k+1,k}+O\big(e^{-(k+2)e_0t}\big) \text{ in }\SSS(\RR^N).
\end{equation}
By Proposition \ref{prop.spectral}, $-(k+1)^2e_0^2$ is not in the spectrum of $L$. It is classical that the resolvent $\big((k+1)^2e_0^2 +L\big)^{-1}$ maps $\SSS$ into $\SSS$ (see e.g. \cite[\S 7.2.2]{DuMe07P} for the proof of a similar fact). In view of \eqref{grosseequation} and \eqref{petiteequation}, it suffices to take
$$\Phi_{k+1}:=\big((k+1)^2e_0^2 +L\big)^{-1}\Psi_{k,k+1}$$ 
to get \eqref{eqWk} at rank $k+1$.
\end{proof}

\subsection{Contraction argument near an approximate solution of large order}
\label{sub.contract}
\begin{prop}
\label{prop.fxpt}
Let $a\in \RR$. There exists $k_0>0$ such that  for any $k\geq k_0$, there exists $t_k\geq 0$ and a solution $U^{\aexp}$ of \eqref{CP} such that for $t\geq t_k$,
\begin{equation}
\label{CondW_A}
\big\|U^{\aexp}-U_{k}^{\aexp}\big\|_{\ell(t,+\infty)}\leq e^{-(k+\frac{1}{2})e_0t}.
\end{equation}
Furthermore, $U^{\aexp}$ is the unique solution of \eqref{CP} satisfying \eqref{CondW_A} for large $t$. It is independent of $k$ and satisfies, for large $t$, 
\begin{equation}
\label{CondWa2}
\|\nabla(U^{\aexp}(t)-U_{k}^{\aexp})\|_{2}+\|\partial_t(U^{\aexp}(t)-U_{k}^{\aexp})\|_{2}\leq e^{-(k+\frac{1}{2})e_0t}.
\end{equation}
\end{prop}
\begin{proof}
We follow the lines of  the proof of \cite[Proposition 6.3]{DuMe07P}.

\noindent\emph{Step 1. Transformation into a fixed-point problem.}
As in the proof of Lemma \ref{lem.approximate}, we will fix $a\in\RR$ and omit most of the superscripts $a$. Recall the definition of $\eps_k$ and $h_k$ from \eqref{defWk} and \eqref{defhk}.
The proof relies on a fixed point argument to construct 
$$ w:=U^{\aexp}-W-h_k.$$
The function $U^{\aexp}$ is solution of \eqref{CP} if and only if $U^{\aexp}-W$ is solution of \eqref{eq.h}. Substracting equation \eqref{eq.h} on $U^{\aexp}-W$ and the equation $\partial_t^2 h_k+L h_k=R(h_k)+\eps_k$, we get that $U^{\aexp}$ satisfies \eqref{CP} if and only if $w$ satisfies $\partial_t^2 w+L w=R(h_k+w)-R(h_k)-\eps_k$. This may be written as
$$\partial_t^2 w-\Delta w=\frac{N+2}{N-2}W^{\frac{4}{N-2}}w+R(h_k+w)-R(h_k)-\eps_k.$$ 
Thus the existence of a solution $U^{\aexp}$ of \eqref{CP} satisfying \eqref{CondW_A} for $t\geq t_k$ may be written as the following fixed-point problem on $w$
\begin{multline}
\label{defM}
\forall t\geq t_k,\quad w(t)=\MMM_k(w)(t)\text{ and } \|w\|_{\ell(t,+\infty)}\leq e^{-\left(k+\frac{1}{2}\right)e_0 t},\text{ where } \MMM_k(w)(t):=\\
-\int_t^{+\infty} \frac{\sin\big((t-s)\sqrt{-\Delta}\big)}{\sqrt{-\Delta}}\left[\frac{N+2}{N-2} W^{\frac{4}{N-2}}w (s)+R(h_k(s)+w(s))-R(h_k(s))-\eps_k(s)\right]ds.
\end{multline}
Let us fix $k$ and $t_k$. Consider
\begin{align*}
E_{\ell}^k&:=\left\{w\in \ell(t_k,+\infty);\; \|w\|_{E_{\ell}^k}:=\sup_{t\geq t_k} e^{\left(k+\frac{1}{2}\right)e_0 t}\|w\|_{\ell(t,+\infty)}<\infty\right\}\\
B_{\ell}^k&:=\big\{w\in E_{\ell}^k,\; \|w\|_{E_{\ell}^k}\leq 1\big\}.
\end{align*}
The space $E_{\ell}^k$ is clearly a Banach space. In view of \eqref{defM}, it is sufficient to show that if $t_k$ and $k$ are large enough, the mapping $\MMM_k$ is a contraction on $B_{\ell}^k$. 
\medskip

\noindent\emph{Step 2. Contraction property.} Note that by Strichartz inequality on the free equation (Lemma \ref{Strichartz}), there is a constant $C^*>0$ such that if $w,\tw\in E_{\ell}^k$, $k\geq 1$,
\begin{align}
\label{Strichartz.M1}
\|\MMM_k(w)\|_{\ell(t,+\infty)}&\leq 
C^*\bigg[\left\|D_x^{1/2} \left(W^{\frac{4}{N-2}}w\right)\right\|_{N(t,+\infty)}\\ \notag &\qquad+\left\|D_x^{1/2} \big(R(h_k+w)-R(h_k)\big)\right\|_{N(t,+\infty)}+\|D_x^{1/2}\eps_k\|_{N(t,+\infty)}\bigg]\\
\label{Strichartz.M2}
\|\MMM_k(w)-\MMM_k(\tw)&\|_{\ell(t,+\infty)}\leq  C^*\bigg[\left\|D_x^{1/2}\left(W^{\frac{4}{N-2}}(w-\tw)\right)\right\|_{N(t,+\infty)}\\ \notag &\qquad\qquad\quad+\left\|D_x^{1/2} \big(R(h_k+w)-R(h_k+\tw)\big)\right\|_{N(t,+\infty)}\bigg].
\end{align}
\begin{claim}
\label{claim.contract}
There exists a constant $C_{k}>0$, depending only on $k$ such that for all $w,\tw\in B_{\ell}^k$ and $t\geq t_k$
\begin{gather}
\label{fxpt.epsk}
\|D_x^{1/2}\eps_k\|_{N(t,+\infty)}\leq C_{k} e^{-(k+1)e_0t},\\
\label{fxpt.R}
\big\|D_x^{1/2}\big(R(h_k+w)-R(h_k+\tw)\big)\big\|_{N(t,+\infty)} \leq C_{k}e^{-(k+\frac 32)e_0t} \|w-\tw\|_{E_{\ell}^k}.
\end{gather}
Furthermore, there exists $k_0>0$ such that for all $k\geq k_0$ and all $w\in E_{\ell}^k$
\begin{equation}
\label{fxpt.VVV}
\left\|D_x^{1/2} \left(W^{\frac{4}{N-2}}w\right)\right\|_{N(t,+\infty)} \leq \frac{1}{4C^*} e^{-(k+\frac 12)e_0t}\|w\|_{E_{\ell}^k}.
\end{equation}
\end{claim}

\begin{proof}[Proof of Claim \ref{claim.contract}]
The proof is very close  to \cite[Claim 6.4]{DuMe07P}. Estimate \eqref{fxpt.epsk} follows immediately from \eqref{eqWk}, and \eqref{fxpt.R} follows immediately from Corollary \ref{corol.est.expo}.\par
Let us show \eqref{fxpt.VVV}. Let $\tau_0\in (0,1)$. By Lemma \ref{lem.est}, there exists a constant $C_2>0$ such that 
$$\|D_x^{1/2}( W^{\frac{4}{N-2}}w)\|_{N(t,t+\tau_0)}\leq C_2\tau_0^{\frac{2}{N+1}}\|w\|_{\ell (t,t+\tau_0)}\leq C_2\tau_0^{\frac{2}{N+1}}e^{-\left(k+\frac{1}{2}\right)e_0 t}\|w\|_{E_{\ell}^k}.$$
By Claim \ref{summation},
$$\|D_x^{1/2}(W^{\frac{4}{N-2}}w)\|_{N(t,+\infty)}\leq \frac{C_2 e^{-(k+\frac 12)e_0t}} {1-e^{-(k+\frac 12)e_0\tau_0}}\tau_0^{\frac{2}{N+1}} \|w\|_{E_{\ell}^k}.$$
Chosing $\tau_0$ and $k_0$ such that $C_2 \tau_0^{\frac{2}{N+1}}=\frac{1}{8C^*}$, and $e^{-(k_0+\frac 12)e_0\tau_0}\leq \frac 12$, we get \eqref{fxpt.VVV} for $k\geq k_0$.
\end{proof}

Chose $k\geq k_0$. By \eqref{Strichartz.M1} and Claim \ref{claim.contract}, we get, if $g\in B_{\ell}^k$ 
\begin{align*}
\|\MMM_k(w)\|_{\ell(t,+\infty)}
&\leq \left(\frac 14 e^{-(k+\frac 12)e_0t}\|w\|_{E_{\ell}^k}+C^*C_{k}e^{-(k+\frac 32)e_0t}\|w\|_{E^k_{\ell}}+C^*C_{k} e^{-(k+1)e_0t_k}\right)\\
&\leq e^{-(k+\frac 12)e_0t}\left(\frac{1}{4}+C^*C_{k}e^{-e_0t_k}+C^*C_{k} e^{-\frac 12e_0t_k}\right).
\end{align*}
Chosing $t_k$ so large that $C^*C_{k}e^{-e_0t_k}+C^*C_{k} e^{-\frac 12 e_0t_k}\leq \frac 12$, we get that for large $k$, $\MMM_k(w)$ is in $B_{\ell}^k$.

Now, let $w,\tw\in B_{\ell}^k$. Similarly, by \eqref{Strichartz.M2} and Claim \ref{claim.contract},
\begin{gather*}
\|\MMM_k(w)-\MMM_k(\tw)\|_{E_{\ell}^k}\leq \|w-\tw\|_{E_{\ell}^k}\left(\frac{1}{4}+C^*C_{k} e^{-e_0t_k}\right),
\end{gather*}
which shows, chosing a larger $t_k$ if necessary, that $\MMM_k$ is a contraction of $B_{\ell}^k$.\par

Thus, for each $k\geq k_0$, \eqref{CP} has an unique solution $U^{\aexp}$ satisfying \eqref{CondW_A} for $t\geq t_k$. The preceding proof clearly remains valid taking a larger $t_k$, so that the uniqueness still holds in the class of solutions of \eqref{CP} satisfying \eqref{CondW_A} for $t\geq t'_k$, where $t'_k$ is any real number larger than $t_k$. Using the uniqueness in the Cauchy problem \eqref{CP}, it is now straightforward to show that $U^{\aexp}$ does not depend on $k\geq k_0$.


It remains to show \eqref{CondWa2}. Let $k>0$ be a large integer and $w\in B_{\ell}^k$. 
By Strichartz inequality \eqref{dualStrichartz}, and the definition of $\MMM_k$, we have, for $t\geq t_k$,
\begin{equation*}
\|\nabla \MMM_k(w)(t)\|_{2}+\|\partial_t \MMM_k(w)(t)\|_{2}\leq \Big\|D_x^{1/2}\Big(\frac{N+2}{N-2}W^{\frac{4}{N-2}}w+R(h_k+w)-R(h_k)-\eps_k\Big)\Big\|_{N(t,+\infty)}.
\end{equation*}
As a consequence of Claim \ref{claim.contract} and the fact that $\|w\|_{E_{\ell}}^k\leq 1$, we get
$$\|\MMM_k(w)(t)\|_{\hdot}\leq C\left(e^{-\left(k+\frac 12 \right)e_0t}\|w\|_{E_{\ell}^k}+e^{-(k+1)e_0 t}\right)\leq Ce^{-\left(k+\frac 12 \right)e_0 t}.$$
Applying the preceding inequality to the solution $w=U^{\aexp}-U_k^{\aexp}$ of the fixed point $w=\MMM_k(w)$, we get directly \eqref{CondWa2}. The proof of Proposition \ref{prop.fxpt} is complete.
\end{proof}
\subsection{Conclusion of the proofs}

\label{sub.conclu}
At this levels, the proof are similar to the one of \cite{DuMe07P}, except for the blowing-up of $W^+$ which is proven in the next subsection.
\begin{proof}[Proof of Theorem \ref{th.exist}.]
The function $\YYY$ is an eigenfunction for the first eigenvalue of $L$, thus it must have constant sign. Replacing $\YYY$ by $-\YYY$ if necessary, we may assume that $\YYY(x)>0$ for all $x$ and thus
\begin{equation}
\label{WY1>0}
\int \nabla \YYY \cdot \nabla W>0.
\end{equation}
Let 
$$W^{\pm}:=U^{\pm 1},$$ 
which yields two solutions of \eqref{CP} for large $t\geq t_0$. Then all the conditions of Theorem \ref{th.exist} are satisfied. Indeed, $W^{\pm}$ is globally defined gor $t\geq t_0$ and by \eqref{CondWa2}, $(W^{\pm},\partial_t W^{\pm})$ tends to $(W,0)$ in $\hdot\times L^2$, which yields \eqref{ex.lim}. The energy condition \eqref{ex.energy} then  follows from the conservation of the energy. Furthermore, again by \eqref{CondWa2},
$$ \|\nabla U^{\aexp}\|_{2}^2=\|\nabla W\|^{2}_{2}+2ae^{-e_0t} \int  (\nabla W \cdot \nabla \YYY+O(e^{-\frac{3}{2}e_0t}),$$
which shows, together with \eqref{WY1>0}, that for large $t>0$,
$$ \|\nabla W^+(t)\|_{2}>0, \quad \|\nabla W^-(t)\|_{2}<0.$$
From Remark \ref{RemPersist}, this inequalities remain valids for every $t$ in the intervals of existence of $W^+$ and $W^-$.

Finally $T_-(W^-)=-\infty$ by  \eqref{T+infty} in Proposition \ref{propsub} and $\|u\|_{S(-\infty,0)}<\infty$ by \eqref{scattering-} in Proposition \ref{prop.CVexpu} .\par

Except for the proof of the finite time blow-up of $W^+$ for negative time, which we postpone to Subsection \ref{subW+}, the proof of Theorem \ref{th.exist} is complete.
\end{proof}

\begin{proof}[Proof of Theorem \ref{th.classif}]
Let us first prove:
\begin{lemma}
\label{lem.unic}
If $u$ is a solution of \eqref{CP} satisfying 
\begin{equation}
\label{devtu1'}
\|\nabla(u(t)-W)\|_{2}+\|\partial_t u(t)\|_{2}\leq Ce^{-\gamma_0 t},\quad E(u_0,u_1)=E(W)
\end{equation}
then 
$$\exists! a\in \RR, \quad u=U^{\aexp}.$$
\end{lemma}
\begin{corol}
\label{corol.unic}
For any $a\neq 0$, there exists $T_a\in \RR$ such that
\begin{equation}
\label{W+-=Wa}\left\{
\begin{aligned}
 U^{\aexp}=W^+(t+T_a) &\text{ if } & a>0\\
U^{\aexp}=W^-(t+T_a) &\text{ if } & a<0.
\end{aligned}
\right.
\end{equation}
\end{corol}
\begin{proof}[Proof of Lemma \ref{lem.unic}]
Let $u=W+h$ be a solution of \eqref{CP} for $t\geq t_0$ satisfying \eqref{devtu1'}. Recall that $h$ satisfies equation \eqref{eq.h}.

\medskip

\noindent\emph{Step 1.} We show that there exists $a\in\RR$ such that
\begin{multline}
\label{condA1}
\forall \eta>0,\quad \|\nabla(h(t)-a e^{-e_0 t}\YYY)\|_{2}+\|\partial_th(t)+ae_0 e^{-e_0 t}\YYY\|_{2}\\
+\big\|h(s)-a e^{-e_0 s}\YYY\big\|_{\ell(t,+\infty)}\leq C_{\eta}e^{
-(2-\eta)e_0 t}.
\end{multline}
Indeed we will show
\begin{equation}
\label{decay.h}
\|\nabla h(t)\|_{2} +\|\partial_t h\|_2 \leq Ce^{-e_0t}, \quad \|R(h(t))\|_{\frac{2N}{N+2}}+\|D_x^{1/2}(R(h))\|_{N(t,+\infty)}\leq  C e^{-2e_0 t}.
\end{equation}
Assuming \eqref{decay.h}, we are in the setting of Proposition \ref{prop.linear}, with $\eps=R(h)$, $c_0=e_0$ and $c_1=2e_0$. The conclusion \eqref{dec.h2} of the lemma would then yield \eqref{condA1}. It remains to prove \eqref{decay.h}.

By Corollaries \ref{corol.est.expo} and \ref{Stri.expo} the bound on $R(h)$ in \eqref{decay.h} follows from the bound on $\|\nabla h(t)\|_{2}+\|\partial_t h(t)\|_2$, so that we only need to show this first bound. 
By Corollary \ref{Stri.expo}, assumption \eqref{devtu1'} implies $\|\nabla h(t)\|_{2}+\|\partial_t h(t)\|_2+\|h\|_{\ell(t,+\infty)}\leq C e^{-\gamma_0 t}.$ By Corollary \ref{corol.est.expo}
$$ \|R(h(t))\|_{L^{\frac{2N}{N+2}}}+\|\nabla(R(h))\|_{N(t,+\infty)}\leq  C e^{-2\gamma_0 t}.$$
Thus we can apply Proposition \ref{prop.linear}, showing that
$$ \|h(t)\|_{\hdot}\leq C\big(e^{-e_0t}+ e^{-\frac{3}{2}\gamma_0t}\big).$$
If $\frac{3}{2}\gamma_0\geq e_0$ the proof of \eqref{decay.h} is complete. If not, assumption \eqref{devtu1'} on $v$ holds with $\frac{3}{2}\gamma_0$ instead of $\gamma_0$, and an iteration argument yields \eqref{decay.h}.

\medskip
\noindent\emph{Step 2.} Let us show
\begin{equation}
\label{condA2}
\forall m>0,\; \exists t_0>0,\; \forall t\geq t_0,\quad \big\|\partial_t\big( u(t)- U^{\aexp}(t)\big)\big\|_{2}+\big\|\nabla\big(u-U^{\aexp}\big)\big\|_{\ell(t,+\infty)}\leq e^{-m t}.
\end{equation}
By Step $1$, \eqref{condA2} holds for $m=\frac{3}{2}e_0$. Let us assume \eqref{condA2} holds for some $m=m_1>e_0$. We will show that it holds for $m=m_1+\frac{e_0}{2}$, which will yield \eqref{condA2} by iteration and conclude the proof.\par
Write $h(t):=u(t)-W$, $h^{\aexp}(t):=U^{\aexp}(t)-W$ (so that in particular $u-U^{\aexp}=h-h^{\aexp}$). Then
$$ \partial_t^2 (h-h^{\aexp})+L(h-h^{\aexp})=-R(h)+R(h^{\aexp}).$$
By induction hypothesis $\|\partial_t(h(t)-h^{\aexp}(t))\|_{2}+\|\nabla(h(t)-h^{\aexp}(t))\|_{2}+\big\|\nabla\big(h-h^{\aexp}\big)\big\|_{\ell(t,+\infty)}\leq e^{-m_1 t}.$
According to Corollary \ref{corol.est.expo}
$$ \big\|D_x^{1/2}\big(R(h)-R(h^{\aexp})\big)\big\|_{N(t,+\infty)}+\big\|R(h(t))-R(h^{\aexp}(t))\big\|_{L^{\frac{2N}{N+2}}}\leq Ce^{-(m_1+e_0) t}.$$
Then by Proposition \ref{prop.linear}
$$ \|\partial_t(h(t)-h^{\aexp}(t))\|_{2}+\|\nabla(h(t)-h^{\aexp}(t))\|_{2}+\big\|\nabla\big(h-h^{\aexp}\big)\big\|_{\ell(t,+\infty)}\leq C e^{-\big(m_1+\frac{3}{4}e_0\big)t},$$
which yields \eqref{condA2} with $m=m_1+\frac{e_0}{2}$. By iteration, \eqref{condA2} holds for any $m>0$. Taking $m=(k_0+1)e_0$ (where $k_0$ is given by Proposition \ref{prop.fxpt}), we get that for large $t>0$
$$ \big\|\nabla\big(u-U_{k_0}^{\aexp}\big)\big\|_{Z(t,+\infty)}\leq e^{-(k_0+\frac{1}{2})e_0t}.$$
By uniqueness in Proposition \ref{prop.fxpt}, we get as announced that $u=W^a$ which concludes the proof of the lemma.
\end{proof}
\begin{proof}[Proof of Corollary \ref{corol.unic}]
Let $a\neq 0$ and chose $T_a$ such that $|a|e^{-e_0T_a}=1$.
By \eqref{CondWa2}, 
\begin{equation}
\label{important.unic}
\big\|\partial_t\big(U^{\aexp}(t+T_a)-W\mp e^{-e_0t} \YYY\big)\big\|_{2}+\big\|\nabla\big(U^{\aexp}(t+T_a)-W\mp e^{-e_0t} \YYY\big)\big\|_{2}\leq Ce^{-\frac 32 e_0t}.
\end{equation}
 Furthermore, $U^{\aexp}(\cdot+T_a)$ satisfies the assumptions of Lemma \ref{lem.unic}, which shows that there exists $a'\in \RR$ such that $U^{\aexp}(\cdot+T_a)=U^{\aexp'}$.
By \eqref{important.unic}, $a'=1$ if $a>0$ and $a'=-1$ if $a<0$, hence \eqref{W+-=Wa}. 
\end{proof}

Let us turn to the proof of Theorem \ref{th.classif}. Point \eqref{theo.crit} is an immediate consequence of the variational characterization of $W$ (\cite{Au76}, \cite{Ta76}). \par

Let us show \eqref{theo.sub}. Let $u$ be a solution of \eqref{CP} such that $E(u_0,u_1)=E(W,0)$ and $\|\nabla u_0\|_{2}<\|\nabla W\|_{2}$. Assume that $\|u\|_{S(\RR)}=\infty$. Replacing if necessary $u(t)$ by $u({-t})$, we may assume that $\|u\|_{S(0,+\infty)}=\infty$. Then (replacing $u$ by $-u$ if necessary), Proposition \ref{prop.CVexpu} shows that there exist $\mu_0>0,\;x_0\in\RR^N$, and $c,C>0$ such that $\|\partial_t u(t)\|_{2}+\|\nabla(u(t)-W_{\mu_0,x_0})\|_{2}\leq Ce^{-ct}.$
This shows that $u_{\mu_0^{-1},-\mu_0^{-1}x_0}$ fullfills the assumptions of Lemma \ref{lem.unic} with $\|\nabla u_0\|_{2}<\|\nabla W\|_{2}$. Thus there exists $a<0$ such that $u_{\mu_0^{-1},-\mu_0^{-1}x_0}=U^{\aexp}$. By Corollary \ref{corol.unic},
$$ u(t)=W^-_{\mu_0,x_0}(t+T_a),$$
which shows \eqref{theo.sub}.\par
The proof of \eqref{theo.super} is similar. Indeed if $u$ is a solution of \eqref{CP} defined on $[0,+\infty)$ and such that $E(u_0,u_1)=E(W,0)$, $\|\nabla u_0\|_{2}>\|\nabla W\|_{2}$ and $u_0\in L^2$, then by Proposition \ref{prop.sur.L2},  $\|u(t)-W_{\mu_0,x_0}\|_{\hdot}\leq Ce^{-ct}$, which shows using Lemma \ref{lem.unic} and the same argument as before that for some $T\in \RR$,
$$ u(t)=W^+_{\mu_0,x_0}(t+T).$$
The proof of Theorem \ref{th.classif} is complete.
\end{proof}

\subsection{Blow-up of $W^+$}
\label{subW+}
In this section we prove that the function $W^+$ blows-up in finite negative time.

We will argue by contradiction, assuming that $W^+$ is globally defined. 
As before, we will write $\DD(t):=\int \left|\nabla W^+\right|^2-\int \left|\nabla W\right|^2+\int \left|\partial_tW^+\right|^2$. Let $\varphi\in C^{\infty}_0(\RR^N)$, radial such that $0\leq \varphi(x)\leq 1$, $\varphi(x)=1$ if $|x|\leq 1$ and $\varphi(x)=0$ if $|x|\geq 2$. Let $\varphi_R(x)=\varphi(x/R)$. Consider
$$ y_R(t):=\int (W^+)^2\varphi_R.$$
We start with some estimates on $y'_R$ and $y''_R$.
\medskip

\noindent\emph{Step 1. Estimates for large positive $t$.} Let us show that 
there exists $R_0,\,t_0,\,c_0>0$ such that for all $R\geq R_0$,
\begin{gather}
\label{y''Rt>t0}
\forall t\geq t_0, \quad y''_R(t)\geq 
4\frac{N-1}{N-2}\int \left(\partial_tW^+\right)^2+\frac{2}{N-2}\left(\int\left|\nabla W^+\right|^2-\left|\nabla W)\right|^2\right)\geq \frac{2}{N-2}\DD(t).\\
\label{estimt0}
y'_R(t_0)\leq -2c_0,\qquad y_R(t_0)\leq \left\{\begin{array}{ll}\frac{C}{R}&\text{ if }N=3\\ C\log R&\text{ if }N=4\\ C&\text{ if }N=5\end{array}\right.
\end{gather}

By explicit computations and $E(W^+,\partial_t W^+)=E(W,0)$, we have 
\begin{gather}
\label{y'W+}
y'_R(t)=2 \int W^+\partial_t W^+\varphi_R,\\
\label{y''W+}
y''_R(t)=4\frac{N-1}{N-2}\int_{\RR^N} \left(\partial_tW^+\right)^2+\frac{4}{N-2}\left(\int_{\RR^N} \left|\nabla W^+\right|^2-\int_{\RR^N} |\nabla W|^2\right)\\
\notag \qquad\qquad+\int \Delta \varphi_R \,\left(W^+\right)^2+2\int(1-\varphi_R)\left(\left|\nabla W^+\right|^2-\left|W^+\right|^{2*}-\left|\partial_t W^+\right|^2\right).
\end{gather}
Replacing $W$ by $W^+$ in the preceding expressions, we see that the corresponding $y_R$ must be constant, so that in particular,
$$\int \Delta \varphi_R \left(W\right)^2+2\int(1-\varphi_R)\left(\left|\nabla W\right|^2-\left|W\right|^{2*}\right)=0.$$
By \eqref{CondW_A} in Proposition \ref{prop.fxpt}, $W^+=W+e^{-e_0t}\YYY+r_1$ with $\|\nabla r_1\|_2+\|\partial_t r_1\|_2\leq Ce^{-2e_0t}$. Developping $W^+$, we get, recalling that $\YYY$ is in $\SSS$ and that $\varphi_R(x)=1$ for $|x|\leq R$,
\begin{gather*}
\left|\int \Delta \varphi_R \left(W^+\right)^2+2\int(1-\varphi_R)\left(\left|\nabla W^+\right|^2-\left|W^+\right|^{2*}-\left|\partial_t W^+\right|^2\right)\right|\leq C\left(\frac{e^{-e_0t}}{R}+e^{-2e_0t}\right),\\
\int \left|\nabla W^+(t)\right|^2-\int |\nabla W|^2=2e^{-e_0t}\int W\YYY+O(e^{-2_0t}).
\end{gather*}
Thus by \eqref{y''W+}
\begin{multline*}
y''_R(t)\geq 4\frac{N-1}{N-2}\int \left(\partial_tW^+\right)^2+\frac{2}{N-2}\left[\int \left|\nabla W^+(t)\right|^2-\int |\nabla W|^2\right]\\
+\frac{4}{N-2}e^{-e_0t}\int W\YYY-C\left(\frac{e^{-e_0t}}{R}+e^{-2e_0t}\right).
\end{multline*}
As $\int W\YYY>0$, we get \eqref{y''Rt>t0} for $R\geq R_0$.

Now, fixing $R\geq R_0$, we get, using that $y'_R(t)$ tends to $0$ at infinity,
$$y'_R(t_0)=-\int_{t_0}^{+\infty} y''_R(t)dt\leq -4\frac{N-1}{N-2}\int_{t_0}^{+\infty}|\partial_t W^+(t)|^2dt,$$
which yields the first assertion in \eqref{estimt0}.

It remains to shows the second assertion in \eqref{estimt0}. Note that $W\approx \frac{C}{|x|^{N-2}}$ at infinity, so that 
\begin{equation}
\label{estimWL2}
\lim_{t\rightarrow +\infty} y_R(t)=\int W^2\varphi_R\approx\left\{
\begin{aligned}
R&\text{ if }N=3,\\
\log R&\text{ if }N=4,\\
\frac{1}{R}&\text{ if }N=5.
\end{aligned}\right.
\end{equation}
Furthermore, $|y'_R(t)|\leq C\|\partial_t W^+\|_2\sqrt{y_R(t)}\leq Ce^{-e_0t}\sqrt{y_R(t)}$, and thus
$$ \sqrt{y_R(t_0)}-\lim_{t\rightarrow \infty}\sqrt{y_R(t)}\leq C\int_{t_0}^{\infty} e^{-e_0t}dt\leq C,$$
which yields together with \eqref{estimWL2}, the second assertion in \eqref{estimt0}. Step 1 is complete

\medskip

\noindent\emph{Step 2. Estimates for $t_0-\eps_0R\leq t\leq t_0$.} As a consequence of the preceding estimates, we show that there exists $C_0>0$ such that for $R\geq R_0$ and $t_0-\eps_0 R\leq t\leq t_0$,
\begin{gather}
\label{estimyW+}
y''_R(t)\geq 4\frac{N-1}{N-2}\int \left(\partial_tW^+\right)^2
+\frac{4}{N-2}\left(\int\left|\nabla W^+\right|^2-\int |\nabla W|^2\right)-\frac{C_0}{R^{N-2}},\\
\label{estim.deriv}
y'_R(t)\leq -c_0,
\end{gather}
where $\eps_0:=\frac{c_0}{2C_0}$.

Estimate \eqref{estim.deriv} follows from \eqref{estimyW+} by integration in time and the fact that $y'_R(t_0)\leq -2c_0$ for $R$ large. Let us show \eqref{estimyW+}.

By \eqref{y''W+}, it is sufficient to show
\begin{equation}
\label{estimaterW+}
\int_{|x|\geq R} r(W^+)(t)dx\leq \frac{C}{R^{N-2}}.
\end{equation}
where $r(W^+)$ is defined in \eqref{defe(u)}. Writing $W^+=W+O(e^{-e_0t})$, and using that $W\approx \frac{1}{|x|^{N-2}}$, $|\nabla 
W|\approx \frac{1}{|x|^{N-1}}$, as $|x|\rightarrow +\infty$, we get for large $R$
$$ \int_{|x|\geq \frac{R}{6}} r(W^+)(t_0+R/4)\leq \frac{C}{R^{N-2}}+Ce^{-\frac{R}{4}e_0}\leq \frac{C}{R^{N-2}}.$$
Hence by finite speed of propagation (Proposition \ref{prop.criterion} \eqref{fsop}), and taking $R$ large,
\begin{equation}
\label{t<t_0}
\forall t\leq t_0,\quad \int_{|x|\geq \frac{R}{2}+t_0-t} r(W^+)(t)\leq \frac{C}{R^{N-2}},
\end{equation}
which yields \eqref{estimaterW+}, and thus \eqref{estimyW+}.

\medskip

\noindent\emph{Step 3. Differential inequalities.}
Let us show that there exists a constant $C>0$ such that
\begin{equation}
\label{inegyRy'R}
\forall t\in \left[t_0-\frac{\eps_0}{2}R,t_0\right],\quad 0\leq -y'_R(t)\leq \frac{C}{R}y_R(t_0).
\end{equation}
By \eqref{estimt0}, if $N=4$ or $N=5$, $\frac{y_R(t_0)}{R}\rightarrow 0$ as $R\rightarrow\infty$ and $2c_0\leq -y_R'(t_0)$. Thus \eqref{inegyRy'R} gives an immediate contradiction in this cases. The remaining case $N=3$, which is the limit case in \eqref{inegyRy'R} will be treated in Steps 4 and 5.

By Step 2 and the fact that $N\geq 3$ we have, for $t_0-\eps_0R\leq t\leq t_0$,
\begin{multline}
\label{inegdiffy'R}
\qquad y'_R(t)^2=4\left( \int \varphi_R W^+\partial_t W^+\right)^2\\
\leq 4\int \varphi_R \left(W^+\right)^2\int \varphi_R(\partial_t W^+)^2\leq \frac{N-2}{N-1}y_R(t)\left(y''_R(t)+\frac{C_0}{R}\right).\qquad
\end{multline}
\begin{claim}[Differential inequality argument]
Let $T>0$ and $y\in C^{2}([0,T])$. Assume 
\begin{equation*}
\forall t\in [0,T], \quad y'(t)\geq c_0>0,\; y(t)>0
\end{equation*}
and for some $C_1>0$,
\begin{equation}
\label{InegDiff}
\forall t\in [0,T],\quad y'(t)^2 \leq \frac{N-2}{N-1}y(t)\left[y''(t)+\frac{C_1}{T}\right].
\end{equation}
Then there is a constant $C>0$ (depending only on $N$, $c_0$ and $C_1$, but not on $T$, such that
\begin{equation*}
\forall t\in \left[0,\frac{T}{2}\right],\quad Ty'(t)\leq Cy(0).
\end{equation*}
\end{claim}
\begin{proof}
By \eqref{InegDiff}, 
\begin{equation*}
\frac{y'(t)}{y(t)}\leq \frac{N-2}{N-1}\left( \frac{y''(t)}{y'(t)}+\frac{C_1}{Ty'(t)}\right)\leq \frac{N-2}{N-1}\left( \frac{y''(t)}{y'(t)}+\frac{C_1}{Tc_0}\right).
\end{equation*}
Then, integrating between $s$ and $t$, $0\leq s\leq t\leq T$,
\begin{equation*}
\log\frac{y(t)}{y(s)}\leq \frac{N-2}{N-1}\left(\log\frac{y'(t)}{y'(s)}+\frac{C_1}{c_0}\right),\text{ i.e}\quad
e^{-\frac{C_1}{c_0}} \frac{y'(s)}{y(s)^{\frac{N-1}{N-2}}}\leq \frac{y'(t)}{y(t)^{\frac{N-1}{N-2}}}.
\end{equation*}
Integrating with respect to $t$ between $s$ and $T$, we get
\begin{equation*}
e^{-\frac{C_1}{c_0}}\frac{y'(s)}{y(s)^{\frac{N-1}{N-2}}}(T-s)\leq -(N-2)\left(\frac{1}{y(T)^{\frac{1}{N-2}}}-\frac{1}{y(s)^{\frac{1}{N-2}}}\right)\leq\frac{N-2}{y(s)^{\frac{1}{N-2}}},
\end{equation*}
which yields 
\begin{equation}
\label{inegdiff2}
\forall s\in[0,T],\quad \frac{y'(s)}{y(s)}(T-s)\leq (N-2)e^{\frac{C_1}{c_0}}\leq C.
\end{equation}
Integrating \eqref{inegdiff2} between $0$ and $t\in\left[0,\frac T2\right]$ we get
\begin{equation*}
\log(y(t))\leq \log(y(0))+C\log \left(\frac{T}{T-t}\right)\leq \log(y(0))+C\log 2, \text{ i.e.}\quad y(t)\leq C y(0),
\end{equation*}
which, gives together with \eqref{inegdiff2}, the announced result.
\end{proof}
By \eqref{inegdiffy'R} and the preceding claim on the function $y=y_R(t_0-t)$, with $R$ large, $T=\eps_0 R$ and $C_1=C_0\eps_0$, we get \eqref{inegyRy'R}.

\medskip

\noindent\emph{Step 4.} Let us show that there exists a constant $C>0$ such that
\begin{equation}
\label{y''<D}
\forall R\geq 1,\quad \forall t\in  \RR,\quad |y''_R(t)|\leq C\DD(t).
\end{equation}

Indeed 
\begin{equation}
\label{expy''}
y''_R(t)=8\int\left(\partial_tW^+\right)^2\varphi_R+
2\int \left(\left|\nabla W^+\right|^2-\left|W^+\right|^{6}-\left|\partial_t W^+\right|^2\right)\varphi_R+\int \left(W^+\right)^2\Delta\varphi_R.
\end{equation}
Let us fix $t\in \RR$. First assume that $\DD(t)\leq \delta_0$. Then by Lemma \ref{lem.est.modul}, there exists $\lambda_0$, $x_0$ such that $\left\|\nabla(W^+(t)-W_{\lambda_0,x_0})\right\|_2+\left\|\partial_t W^+(t)\right\|_2\leq C\DD(t).$
Noting that 
$$8\int\left(\partial_tW_{\lambda_0,x_0}\right)^2\varphi_R+
2\int \left(\left|\nabla W_{\lambda_0,x_0}\right|^2-\left|W_{\lambda_0,x_0}\right|^{6}\right)\varphi_R+\int \left(W_{\lambda_0,x_0}\right)^2\Delta\varphi_R=0,$$
we get \eqref{y''<D} for $\DD(t)\leq \delta_0$ by developping \eqref{expy''}.

Now assume that $\DD(t)\geq \delta_0$. Thus $\frac{2\|\nabla W\|^2_2}{\delta_0}\DD(t_0)\geq \|\nabla W\|_2^2$. As a consequence,
$$ \left(1+\frac{2\left\|\nabla W\right\|_2^2}{\delta_0}\right) \DD(t)\geq \left\|\nabla W^+\right\|^2_2+\left\|\partial_t W^+\right\|_2^2.$$
By the energy equality $E(W^+,\partial_t W_+)=E(W,0)$, $\|W^+\|_6^6=3\|\nabla W^+\|^2+3\|\partial_t W^+\|^2+E(W,0)$. Thus there exists a constant $C>0$ such that
$$ \DD(t_0)\geq \delta_0\Longrightarrow C\DD(t)\geq\left\|\nabla W^+\right\|^2_2+\left\|W^+\right\|_{6}^6+\left\|\partial_t W^+\right\|_2^2,$$
which shows, together with \eqref{expy''}, implies estimate \eqref{y''<D} in the case $\DD(t)\geq \delta_0$ 

\medskip

\noindent\emph{Step 5. End of the proof in the case $N=3$.}
Let us first show
\begin{equation}
\label{intDW+}
\int_{-\infty}^{+\infty} \DD(t)dt<\infty.
\end{equation}
Estimates \eqref{estimt0} and \eqref{inegyRy'R} give a constant $C$, independant or $R$, such that
$0\leq -y'_R(t_0-\eps_0 R)\leq C.$
Thus by \eqref{y''Rt>t0} and \eqref{estimyW+},
\begin{align*}
\int_{t_0-\eps_0R}^{+\infty}\DD(t)dt&=\int_{t_0-\eps_0R}^{t_0}\DD(t)+\int_{t_0}^{+\infty}\DD(t) 
\leq \int_{t_0-\eps_0R}^{t_0}\left(y''_R(t)+\frac{C}{R}\right)+\int_{t_0}^{+\infty}y''_R(t)\\
&\leq
C\eps_0+\int_{t_0-\eps_0R}^{+\infty} y''_R(t)dt =C\eps_0-y'_R\left(t_0-\eps_0R\right)\leq C.
\end{align*}
Letting $R$ tends to $\infty$ we get \eqref{intDW+}.

Let $M\gg 1$. Let $t\in [t_0-M,t_0]$. By \eqref{estimyW+}, for $R\geq R_0$, $y''_R(t)\geq 4\DD(t)-\frac{C_0}{R}$. Taking $R_M\gg 1$ so that $\min_{t_0-M\leq t\leq t_0} \DD(t)\geq \frac{C_0}{R_M}$, we get that $y''_{R_M}(t)\geq 2\DD(t)$ for $t_0-M\leq t\leq t_0$ and thus, in view of \eqref{y''Rt>t0},
\begin{equation}
\label{presquefini}
\forall t\geq t_0-M,\;\quad y''_{R_M}(t)\geq 2\DD(t).
\end{equation}
By \eqref{intDW+}, there exists a sequence $t_n\rightarrow -\infty$ such that $\DD(t_n)\rightarrow 0$. As a consequence, $y'_R(t_n)\rightarrow 0$. We have
$y'_R(t_n)=-\int_{t_n}^{t_0-M} y''_R(t)dt- \int_{t_0-M}^{+\infty} y''_R(t)dt$, which yields by \eqref{y''<D} and \eqref{presquefini}
$$ \int_{t_0-M}^{+\infty} \DD(t)dt\leq C \int_{t_n}^{t_0-M} \DD(t) dt+\frac{1}{2}|y'_R(t_n)|.$$
Letting $n$ tends to infinity, we obtain
$$ \forall M\gg 1,\quad \int_{t_0-M}^{+\infty} \DD(t)dt\leq C \int_{-\infty}^{t_0-M} \DD(t) dt.$$
Note that by \eqref{intDW+}, both integrals in the preceding inequality are finite. Letting $M$ tends to $+\infty$, we get
$\int_{-\infty}^{+\infty} \DD(t)dt=0$. This shows that $\DD(t)=0$ for all $t$, which is a contradiction. 
\qed
\appendix

\section{Estimates on the modulation parameters}
\label{Appendix.Decompo}
In this appendix we prove Lemma \ref{lem.est.modul}.
\noindent\emph{Proof of \eqref{est.modul.1}.} In the proof of estimate \eqref{est.modul.1}, $t$ is just a parameter that we will systematically omit.\par 
Developping the equality $E\big((1+\alpha) W+\tf,\partial_t u\big)=E(W,0)$, we get, with \eqref{W+f},
\begin{equation}
\label{es.mo1}
Q(\alpha W+\tf)+\frac 12 \|\partial_t u\|^2_2=O\big(\|\nabla(\alpha W+\tf)\|^3_2\big).
\end{equation}
By the orthogonality of $\tf$ with $W$ in $\hdot$, and the equation $\Delta W+W^{\frac{N+2}{N-2}}=0$ we have
\begin{equation*}
 \int \nabla W \cdot\nabla \tf=\int W^{\frac{N+2}{N-2}}\tf =0.
\end{equation*}
Thus $W$ and $\tf$ are $Q$-orthogonal and $Q(\alpha W+\tf)=-\alpha^2|Q(W)|+Q\big(\tf\big)$.
Hence
\begin{equation}
\label{es.mo2}
-\alpha^2|Q(W)|+Q\big(\tf\big)+\frac 12 \|\partial_t u\|^2_2=O\big(\|\nabla(\alpha W+\tf) \|^3_2\big).
\end{equation}
Now
\begin{equation}
\label{es.mo3}
\|\nabla(\alpha W+\tf)\|^2_2=\alpha^2\|\nabla W\|^2_2+\|\nabla \tf\|_2^2.
\end{equation}
If $\delta_0$ is small, so is $\|\nabla (\alpha W+\tf)\|_{L^2}$, so that by \eqref{es.mo2} and Claim \ref{coercivity}, there exists $c>0$ such that $\alpha^2\geq c \left( \|\nabla \tf\|_{2}^2+\|\partial_t u\|_{2}^2\right)$. Thus by \eqref{es.mo3},
$\big\|\nabla(\alpha W+\tf)\big\|_2^2\approx \alpha^2.$
Using again \eqref{es.mo2}, we get
$$ \alpha^2 \approx \|\nabla \tf\|_{2}^2+\|\partial_t u\|_2^2.$$
We have
\begin{equation*}
\|\nabla u\|_2^2-\|\nabla W\|_2^2=\|\nabla(W+\alpha W+\tf)\|_2^2-\|\nabla W\|_2^2=(2\alpha+\alpha^2) \|\nabla W\|_2^2+\big\|\nabla \tf\big\|_2^2\approx\alpha,
\end{equation*}
which concludes the proof of \eqref{est.modul.1}.

\medskip

\noindent\emph{Proof of \eqref{est.modul.3}.} Let $v(t):=u_{\mu,\xx}(t)$. Then
\begin{equation}
\label{u.module}
u=\mu^{\frac{N-2}{2}} v\big(t,\mu(x-\xx(t))\big).
\end{equation}
Differentiating \eqref{u.module}, and writing $y=\mu(x-\xx(t))$, we get
$$ \partial_t u(t,x)=\left(\frac N2-1\right)\frac{\mu'}{\mu} u(t,x)+
\mu^{\frac {N-2}{2}}\left[ \mu' (x-\xx(t))\cdot\nabla_y v\big(t,y\big)-\mu\xx'\cdot \nabla v(t,y)+\partial_t v(t,y)\right].$$
Recall that $v=W+\alpha W+\tf$. Multiplying the preceding equation by $\mu^{-\frac N2}$, we obtain
\begin{multline*}
w(t,y):=\frac{1}{\mu^{\frac N2}}\partial_t u\big(t,\frac{y}{\mu}+\xx\big)\\
=\left(\frac{N-2}{2}\right) \frac{\mu'}{\mu^2}(W+\alpha W+\tf)+\frac{\mu'}{\mu^2}y\cdot \nabla_y (W+\alpha W+\tf)-\xx'\cdot \nabla_y(W+\alpha W+\tf)+\frac{1}{\mu}\frac{\partial}{\partial t}(\alpha W+\tf).
\end{multline*}
Hence
\begin{gather*}
w=-\frac{\mu'}{\mu^2}\tW-\sum_{j=1}^{N} \xx_j'(t)W_j+\frac{\alpha'}{\mu}W+(R)\\
(R):=\frac{\mu'}{\mu^2}\left(\frac {N-2}{2}+y\cdot\nabla_y\right)(\alpha W+\tf)-\sum_{j=1}^{N} x_j'(t)\partial_{y_j}(\alpha W+\tf)+\frac{1}{\mu}\partial_t \tilde{f}.
\end{gather*}
Hence (using the orthogonality in $\hdot$ of $W$, $\tW$, $W_1$,\ldots,$W_N$)
\begin{equation}
\label{scalarprod}
\left\{
\begin{aligned}
 \int w\Delta W&=-\frac{\alpha'}{\mu} \|\nabla W\|^2_2+\int (R)\Delta W\\
\int w\Delta \tW&=-\frac{\mu'}{\mu^2}\|\nabla \tW\|^2_2+\int (R)\Delta \tW\\
\int w\Delta W_{j}&=-x_j'(t)\|\nabla W_{j}\|^2_2+\int (R)\Delta W_{j},\quad j=1\ldots N.
\end{aligned}
\right.
\end{equation}
Let
\begin{equation*}
\dd_1(t):=\dd(t)+\left|\frac{\alpha'}{\mu}\right|+\left|\frac{\mu'}{\mu^2}\right|+\left|\xx'(t)\right|.
\end{equation*}
Then, noting that for all $t$, $\partial_t \tf(t)\in H^{\bot}$ and using estimate \eqref{est.modul.1}, we have
$$\left|\int (R)\Delta W\right|+\left|\int (R)\Delta (-iW)\right|+\left|\int (R)\Delta \tW\right|+\sum_{j=1}^N\left|\int (R)\Delta W_{j}\right|\leq C \dd_1^2(t).$$
Summing up estimates \eqref{scalarprod} and using that $\|w(t)\|_2=\|\partial_t u(t)\|_2\lesssim \dd(t)$, we get
$$ \dd_1(t)\leq \dd_1^2(t)+\dd(t)$$
which yields \eqref{est.modul.3} for small $\dd(t)$.
\qed
\section{Some technical estimates}
\label{appendixest}
In this appendix we proof Lemma \ref{lem.est}.

We will need the following version of H\"older estimate for fractional Sobolev spaces 
\begin{claim}
\label{KPV}
\begin{enumerate}
Let $p,p_2,p_3,p_4\in(1,\infty)$, $p_1\in (1,\infty]$, such that $\frac{1}{p_1}+\frac{1}{p_2}=\frac{1}{p}$, and $\frac{1}{p_3}+\frac{1}{p_4}=\frac{1}{p}$. 
\item \label{produit} $\displaystyle  \big\|D_x^{1/2}(fg) \big\|_{p}\lesssim \big\|f\big\|_{p_1} \big\|D_x^{1/2}g \big\|_{p_2}+ \big\|D_x^{1/2}f \big\|_{p_3}\big\|g\big\|_{p_4}.$
\item \label{composition1} Let $F\in C^{1}(\RR)$ such that $F(0)=0$. Then
$$  \big\|D_x^{1/2}F(f) \big\|_p\lesssim \big\|F'(f)\big\|_{p_1} \big\|D_x^{1/2}f \big\|_{p_2}$$
\item \label{composition2} Let $m\in \NN^*$ and $r_1,r_2,r_3\in(1,\infty)$ such that $\frac{1}{r_1}+\frac{1}{r_2}+\frac{1}{r_3}=1$. Then
\begin{multline*}
 \big\|D_x^{1/2}(f^{2m}-g^{2m}) \big\|_{p}\lesssim \left(\big\|f^{2m-1}\big\|_{p_1}+\big\|g^{2m-1}\big\|_{p_1}\right) \big\|D_x^{1/2}(f-g) \big\|_{p_2}\\
+ \left(\big\|f^{2m-2}\big\|_{r_1}+\big\|g^{2m-2}\big\|_{r_1}\right)\left( \big\|D_x^{1/2}f \big\|_{r_2}+ \big\|D_x^{1/2}g \big\|_{r_2}\right)\big\|f-g\big\|_{r_3}.
\end{multline*}
\end{enumerate}
\end{claim}
Points \eqref{produit} and \eqref{composition1} follows from \cite[Theorems A.7, A.8, A.9 and A.12]{KePoVe93}. Point \eqref{composition2} is a consequence of \eqref{produit}.

\medskip

\noindent\emph{Proof of \eqref{Lest}.}

By Claim \ref{KPV} \eqref{produit}, we have
\begin{multline}
 \label{Holde1}
\left\|D_x^{1/2}\left(u(t)W^{\frac{4}{N-2}}\right)\right\|_{\frac{2(N+1)}{N+3}}\\
\leq C\left\{\left\|D_x^{1/2} u(t)\right\|_{\frac{2(N+1)}{N-1}}\left\|W^{\frac{4}{N-2}}\right\|_{\frac{N+1}{2}}+\left\|u(t)\right\|_{\frac{2(N+1)}{N-2}}\left\|D_x^{1/2}\left(W^{\frac{4}{N-2}}\right)\right\|_{\frac{2(N+1)}{5}}\right\}.
\end{multline}
Furthermore, $W$ is a $C^{\infty}$ function on $\RR^N$ such that $W^{\frac{4}{N-2}}\approx 1/|x|^4$, $D_x\left(W^{\frac{4}{N-2}}\right)\approx 1/|x|^5$. Hence $W^{\frac{4}{N-2}}$ belongs to $L^{\frac{N+1}{2}}\cap W^{1,\frac{2(N+1)}{5}}$. As a consequence we can rewrite \eqref{Holde1} as
$$\left\|D_x^{1/2}\left(u(t)W^{\frac{4}{N-2}}\right)\right\|_{\frac{2(N+1)}{N+3}}\leq C\left(\left\|D_x^{1/2} u(t)\right\|_{\frac{2(N+1)}{N-1}}+\Big\|u(t)\Big\|_{\frac{2(N+1)}{N-2}}\right),$$
which gives \eqref{Lest}, using H\"older inequality in time.

We will skip the proof of \eqref{NLest0} which is a direct consequence of H\"older inequality.

\medskip

\noindent\emph{Proof of \eqref{NLest}.}

Fixing $t$, we will show
\begin{multline}
\label{NLestx}
\big\|D_x^{1/2}(R(u)-R(v))\big\|_{\frac{2(N+1)}{N+3}}\lesssim \big\|D_x^{1/2}(u-v)\big\|_{\frac{2(N+1)}{N-1}}\Bigg[\Big\||u|+|v|\Big\|_{\frac{2(N+1)}{N-2}}+ \Big\||u|+|v| \Big\|_{\frac{2(N+1)}{N-2}}^{\frac{4}{N-2}}\Bigg]\\
+\big\|u-v\big\|_{\frac{2(N+1)}{N-2}}
\Bigg[\Big\||u|+|v|\Big\|_{\frac{2(N+1)}{N-2}}+ \Big\||u|+|v| \Big\|_{\frac{2(N+1)}{N-2}}^{\frac{4}{N-2}}\\
+\Big\|\big|D_{x}^{1/2}u\big|
+\big|D_{x}^{1/2}v\big| \Big\|_{\frac{2(N+1)}{N-1}}
+\Big\|\big|D_{x}^{1/2}u\big|+\big|D_{x}^{1/2}v\big| \Big\|_{\frac{2(N+1)}{N-1}}^{\frac{4}{N-2}}\Bigg].
\end{multline}
H\"older inequality in time will yield the desired result.
We have $R(u)=W^{\frac{N+2}{N-2}}J\left(W^{-1}u\right)$, where $J(s)=|1+s|^{\frac{4}{N-2}}(1+s)-1-\frac{N+2}{N-2} s$, $J'(s)=\frac{N+2}{N-2}|1+s|^{\frac{4}{N-2}}-\frac{N+2}{N-2}$. Hence
\begin{equation}
\label{diffR}
R(u)-R(v)=\frac{N+2}{N-2}(u-v)\underbrace{\int_0^1\left( \left|W+v+(u-v)\theta\right|^{\frac{4}{N-2}}-W^{\frac{4}{N-2}}\right)d\theta}_{I(u,v)}.
\end{equation}
Hence, by Claim \ref{KPV} \eqref{produit}, 
$$\Big\|D_x^{1/2}(R(u)-R(v))\Big\|_{\frac{2(N+1)}{N+3}}\lesssim \big\|D_x^{1/2}(u-v)\big\|_{\frac{2(N+1)}{N-1}}\big\|I(u,v)\big\|_{\frac{N+1}{2}}+\big\|u-v\big\|_{\frac{2(N+1)}{N-2}}\big\|D_x^{1/2}I(u,v)\big\|_{\frac{2(N+1)}{5}}.$$
In view of the integral expression of $I(u,v)$, \eqref{NLestx} will follow from the estimates
\begin{gather}
\label{EstReduite}
\Big\|\left|W+h\right|^{\frac{4}{N-2}}-W^{\frac{4}{N-2}}\Big\|_{\frac{N+1}{2}} \lesssim\big\|h \big\|_{\frac{2(N+1)}{N-2}}+  \big\|h \big\|_{\frac{2(N+1)}{N-2}}^{\frac{4}{N-2}}\\
\label{EstReduite2}
\Big\|D_x^{1/2}\left(\left|W+h\right|^{\frac{4}{N-2}}-W^{\frac{4}{N-2}}\right)\Big\|_{\frac{2(N+1)}{5}}
\qquad\qquad\qquad\qquad\\ \notag \qquad\qquad\qquad \qquad \lesssim\big\|h \big\|_{\frac{2(N+1)}{N-2}}+  \big\|h \big\|_{\frac{2(N+1)}{N-2}}^{\frac{4}{N-2}}+  \big\|D_{x}^{1/2}h \big\|_{\frac{2(N+1)}{N-1}}+\big\|D_{x}^{1/2}h \big\|_{\frac{2(N+1)}{N-1}}^{\frac{4}{N-2}}.
\end{gather}
Let us first show \eqref{EstReduite}. By the pointwise bound $\left| |W+h|^{\frac{4}{N-2}}-|W|^{\frac{4}{N-2}}\right|\lesssim W^{\frac{6-N}{N-2}}|h|+|h|^{\frac{4}{N-2}}$, and H\"older inequality
$$ \Big\|\left|W+h\right|^{\frac{4}{N-2}}-W^{\frac{4}{N-2}}\Big\|_{\frac{N+1}{2}}\lesssim 
\big\|W^{\frac{6-N}{N-2}}\big\|_{\frac{2(N+1)}{6-N}}\big\|h\big\|_{\frac{2(N+1)}{N-2}}+\big\|h\big\|_{\frac{2(N+1)}{N-2}}^{\frac{4}{N-2}}.$$
Noting that $\big\|W^{\frac{6-N}{N-2}}\big\|_{\frac{2(N+1)}{6-N}}=\big\|W\big\|_{\frac{2(N+1)}{N-2}}^{\frac{6-N}{N-2}}<\infty$, we get \eqref{EstReduite}. It remains to show \eqref{EstReduite2}. We will distinguish two cases.
\medskip

\noindent\emph{First case: $N=3$ or $N=4$.}
Then $\frac{4}{N-2}\in\{2,4\}$. By Claim \ref{KPV}, \eqref{composition2},
\begin{multline*} \Big\|D_x^{1/2}\left(\left(W+h\right)^{\frac{4}{N-2}}-W^{\frac{4}{N-2}}\right)\Big\|_{\frac{2(N+1)}{5}}
\\
\lesssim\left(\Big\|(W+h)^{\frac{6-N}{N-2}}\Big\|_{\frac{2(N+1)}{6-N}}+\Big\|h^{\frac{6-N}{N-2}}\Big\|_{\frac{2(N+1)}{6-N}}\right)\Big\|D_x^{1/2}h\Big\|_{\frac{2(N+1)}{N-1}}\\
+\left(\Big\|(W+h)^{\frac{8-2N}{N-2}}\Big\|_{\frac{2(N+1)}{8-2N}}+\Big\| h^{\frac{8-2N}{N-2}}\Big\|_{\frac{2(N+1)}{8-2N}}\right)\\
\times
\left(\Big\|D_x^{1/2}(W+h)\Big\|_{\frac{2(N+1)}{N-1}}+\Big\|D_x^{1/2}h\Big\|_{\frac{2(N+1)}{N-1}}\right)\Big\|h\Big\|_{\frac{2(N+1)}{N-2}}.
\end{multline*}
Hence
\begin{multline*} \Big\|D_x^{1/2}\left(\left(W+h\right)^{\frac{4}{N-2}}-W^{\frac{4}{N-2}}\right)\Big\|_{\frac{2(N+1)}{5}}
\lesssim\left(1+\Big\|h\Big\|_{\frac{2(N+1)}{N-2}}\right)^{\frac{6-N}{N-2}}\Big\|D_x^{1/2}h\Big\|_{\frac{2(N+1)}{N-1}}\\
\qquad\qquad\qquad\qquad\qquad\qquad\qquad+\left(1+\Big\| h\Big\|_{\frac{2(N+1)}{N-2}}^{\frac{8-2N}{N-2}}\right)
\left(1+\Big\|D_x^{1/2}h\Big\|_{\frac{2(N+1)}{N-1}}\right)\Big\|h\Big\|_{\frac{2(N+1)}{N-2}}\\
\lesssim\Big\|h\Big\|_{\frac{2(N+1)}{N-2}}+\Big\| h\Big\|_{\frac{2(N+1)}{N-2}}^{\frac{4}{N-2}}+\Big\|D_x^{1/2}h\Big\|_{\frac{2(N+1)}{N-1}}+\Big\|D_x^{1/2}h\Big\|_{\frac{2(N+1)}{N-1}}^{\frac{4}{N-2}},
\end{multline*}
by the convexity inequality $AB\leq \frac{6-N}{4}A^{\frac{4}{6-N}}+\frac{N-2}{4}B^{\frac{4}{N-2}}$.
This yields \eqref{EstReduite2}, and concludes the proof of \eqref{NLestx} (thus of \eqref{NLest}) when $N=3$ or $N=4$.

\medskip

\noindent\emph{Second case: $N=5$.}

In this case, we must bound $\left\| D_x^{1/2}\left( |W+h|^{4/3}-W^{4/3}\right)\right\|_{\frac{12}{5}}$ by sum of powers of $\|h\|_{4}$ and $\|D_x^{1/2}h\|_3$. Note that Claim \ref{KPV} \eqref{composition2} is no longer available. 
We have
\begin{equation}
\label{decomposition5}
|W+h|^{4/3}-W^{4/3}=W^{4/3}F\big(W^{-1}h\big)+|h|^{4/3},\quad F(s)=|1+s|^{4/3}-1-|s|^{4/3}.
\end{equation}
By Claim \ref{KPV}, \eqref{composition1}
\begin{equation}
\label{terme1}
 \Big\|D_x^{1/2}|h|^{4/3}\Big\|_{\frac{12}{5}}\lesssim \Big\||h|^{1/3}\Big\|_{12}\Big\|D_x^{1/2}h\Big\|_{3}\lesssim \Big\|h\Big\|_{4}^{1/3}\Big\|D_x^{1/2}h\Big\|_{3}\lesssim \Big\|h\Big\|_{4}^{4/3}+\Big\|D_x^{1/2}h\Big\|_{3}^{4/3},
\end{equation}
by the convexity inequality $AB\leq \frac 34 A^{4/3}+\frac 14 B^4$.

Note that $F$ is $C^1$ and that $F'$ is bounded. In order to apply Claim \ref{KPV} \eqref{composition1} to $W^{4/3}F\big(W^{-1}h\big)$ we will need a dyadic decomposition of $\RR^5$. Let $\varphi\in C^{\infty}(\RR^5)$ such that $\varphi(x)=1$ if $|x|\leq 1$ and $\varphi(x)=0$ if $|x|\geq 2$. 
Define $\psi(x):=\varphi(x/2)-\varphi(x)$, so that $\supp \psi\subset\{1\leq |x|\leq 4\}$. Let $\psi_k(x):=\varphi(x/2^{k-1})$ for $k\geq 1$ and $\psi_0(x):=\varphi(x)$. 
Then
\begin{equation*}
\supp \psi_k\subset\left\{\frac{1}{2^{k-1}}\leq |x|\leq \frac{1}{2^{k+1}}\right\},\;k\geq 1,\; \supp\psi_0\subset\{|x|\leq 2\};\quad \sum_{k\geq 0} \psi_k(x)=1.
\end{equation*}
Choose also $\tpsi\in C^{\infty}_0(\RR^5)$ such that $\supp\tpsi\subset\{1/2\leq |x|\leq 8\}$ and $\tpsi(x)=1$ on $\{2\leq |x|\leq 4\}$. Let $\tpsi_k(x):=\tpsi(x/2^{k-1})$ for $k\geq 1$ and $\psi_0(x):=\varphi(x/2)$. Then 
\begin{equation*}
\supp \tpsi_k\subset\left\{\frac{1}{2^{k-2}}\leq |x|\leq \frac{1}{2^{k+2}}\right\},\;k\geq 1,\; \supp\tpsi_0\subset\{|x|\leq 4\};\quad
x\in \supp \tpsi_k\Longrightarrow \psi_k(x)=1.
\end{equation*}
We have
\begin{equation}
\label{Fdyadic}
W^{4/3}F\big(W^{-1}h\big)=\sum_{k\geq 0} \psi_kW^{4/3}F\big(W^{-1}h\big)=\sum_{k\geq 0} \psi_kW^{4/3}F\left(W^{-1}\tpsi_k(x)h\right).
\end{equation}
We leave the proof of the following estimates which follow from the explicit expression of $W$ and scaling arguments to the reader.
\begin{claim}
\label{ClaimW5}
For all $p\in [1,\infty]$ and for all $k\geq 0$,
\begin{gather*}
\Big\|\psi_k W^{4/3}\Big\|_{p}\lesssim 2^{(-4+5/p)k},\quad\Big\|D_x^{1/2}\big(\psi_k W^{4/3}\big)\Big\|_{p}\lesssim 2^{(-9/2+5/p)k},\\
\Big\|\tpsi_k W^{-1}\Big\|_p\lesssim 2^{(3+5/p)k},\quad \Big\|D_x^{1/2}\big(\tpsi_k W^{-1}\big)\Big\|_p\lesssim 2^{(5/2+5/p)k}.
\end{gather*}
\end{claim}
By Claim \ref{KPV} \eqref{produit},
\begin{multline}
\label{dyadic1}
\Big\|D_x^{1/2}\big(\psi_kW^{4/3}F(W^{-1}\tpsi_kh)\big)\Big\|_{\frac{12}{5}}\\ 
\lesssim \Big\|D_x^{1/2}\big(\psi_kW^{4/3}\big)\Big\|_{6}\Big\|F(W^{-1}\tpsi_kh)\Big\|_{4}+\Big\|\psi_kW^{4/3}\Big\|_{12}\Big\|D_x^{1/2}\big(F(W^{-1}\tpsi_kh)\big)\Big\|_{3}.
\end{multline}
Note that $|F(s)|\lesssim |s|$, so that, in view of Claim \ref{ClaimW5},
$$ \Big\|F(W^{-1}\tpsi_kh)\Big\|_{4}\lesssim \Big\|W^{-1}\tpsi_kh\Big\|_{4}\lesssim 2^{3k}\big\|h\big\|_4 $$
Thus by Claim \ref{ClaimW5} again,
\begin{equation}
\label{dyadic2}
\Big\|D_x^{1/2}\big(\psi_kW^{4/3}\big)\Big\|_{6}\Big\|F(W^{-1}\tpsi_kh)\Big\|_{4}\lesssim 2^{(-9/2+5/6+3)k}=2^{-\frac 23k}\big\|h\big\|_4.
\end{equation}
Furthermore, $F'$ being bounded, by Claim \ref{KPV} \eqref{composition1}, then \eqref{produit}
\begin{multline*}
\Big\|D_x^{1/2}\big(F(W^{-1}\tpsi_kh)\big)\Big\|_{3}\lesssim  \Big\|D_x^{1/2}\big(W^{-1}\tpsi_kh\big)\Big\|_{3}\\
\lesssim \Big\|D_x^{1/2}\big(W^{-1}\tpsi_k\big)\Big\|_{12}\Big\|h\Big\|_4+\Big\|W^{-1}\tpsi_k\Big\|_{\infty}\Big\|D_x^{1/2}h\Big\|_{3}.
\end{multline*}
Thus by Claim \ref{ClaimW5},
\begin{multline}
\label{dyadic4}
\Big\|\psi_kW^{4/3}\Big\|_{12}\Big\|D_x^{1/2}\big(F(W^{-1}\tpsi_kh)\big)\Big\|_{3}\lesssim 2^{(-4+5/12)k}\left(2^{(3+5/12)k}\big\|h\big\|_4+2^{3k}\big\|D_x^{1/2}h\big\|_{3}\right)\\
\lesssim 2^{-\frac 16k}\big\|h\big\|_4+2^{-\frac{7}{12}k}\big\|D_x^{1/2}h\big\|_{3}.
\end{multline} 
By \eqref{dyadic1}, \eqref{dyadic2} and \eqref{dyadic4},
\begin{equation}
\label{dyadic5}
\Big\|D_x^{1/2}\big(W^{4/3}F(W^{-1}h)\big)\Big\|_{\frac{12}{5}}\lesssim \sum_{k\geq 0}
\Big\|D_x^{1/2}\big(\psi_kW^{4/3}F(W^{-1}\tpsi_kh)\big)\Big\|_{\frac{12}{5}}\lesssim \big\|h\big\|_{4}+\big\|D_x^{1/2}h\big\|_3.
\end{equation}
In view of \eqref{decomposition5}, we get, by \eqref{terme1} and \eqref{dyadic5}
$$ \Big\|D_x^{1/2}\big(|W+h|^{4/3}-W^{4/3}\big)\Big\|_{\frac{12}{5}}\lesssim \big\|h\big\|_{4}^{4/3}+\big\|D_x^{1/2}h\big\|_{3}^{4/3}+\big\|h\big\|_{4}+\big\|D_x^{1/2}h\big\|_3.$$
This yields \eqref{EstReduite2}, concluding the proof of \eqref{NLest} in the case $N=5$.
\qed 
\section{Derivative of $g_R$}

\begin{claim}
\label{calculs}
Let $u$ be a solution of \eqref{CP} such that $E(u_0,u_1)=E(W,0)$ and $g_R$ be defined by \eqref{defgR}.
There exist $C^{\infty}$ real-valued functions on $\RR^N$, $a_R^{jk}$, $b_R^1$, $b_R^2$, $b_R^3$, supported in $\{|x|\geq R\}$, bounded independently of $R$ and such that 
$$
g_R'(t)=\frac{1}{N-2}\int |\partial_t u|^2dx-\frac{1}{N-2}\left(\int |\nabla W|^2dx-\int |\nabla u|^2dx\right)+A_{R}(u,\partial_t u).$$
where 
\begin{equation}
\label{defAR}
A_{R}(u,\partial_t u):=\sum_{jk} \int a_{R}^{jk}\partial_j u\partial_k u\,dx+\int b_{R}^1 (\partial_t u)^2+b_{R}^2 u^{2^*}+\frac{1}{|x|^2}b_{R}^3u^2 dx.
\end{equation}
\end{claim}
\begin{proof}
\begin{equation}
\label{calcul1}
\frac{d}{dt} \int \psi_R\cdot \nabla u\, \partial_t u\,dx=\int \psi_R\cdot \nabla u\left(\Delta u+|u|^{\frac{N+2}{N-2}}u\right)dx+\int \psi_R\cdot \nabla \partial_t u\, \partial_t u\,dx.
\end{equation}
Furthermore, denoting by $\psi_{Rj}$, $j=1\ldots N$, the coordinates of $\psi_R$,
\begin{align*}
\int \psi_R\cdot \nabla u(\Delta u)dx&=\sum_{j,k} \int \psi_{Rj}\partial_j u\partial_{k}^2 u\,dx\\
&=-\frac 12 \sum_j \int \frac{\partial \psi_{Rj}}{\partial x_j}(\partial_j u)^2dx+\frac 12 \sum_{\substack{j,k\\ j\neq k}} \int \frac{\partial \psi_{Rj}}{\partial x_k}(\partial_j u)^2dx-\sum_{\substack{j,k\\ j\neq k}} \int \frac{\partial \psi_{Rj}}{\partial x_k}\partial_j u \partial_k u\,dx.
\end{align*}
Note that if $|x|\leq R$, $\frac{\partial \psi_{Rj}}{\partial x_j}(x)=1$ and $\frac{\partial \psi_{Rj}}{\partial x_k}(x)=0$ ($j\neq k$). Thus
\begin{equation*}
\int \psi_R\cdot \nabla u (\Delta u)dx=\frac{N-2}{2} \int |\nabla u|^2dx+\sum_{jk} \int_{|x|\geq R} \tilde{a}_{Rjk} \partial_j u\partial_k u\,dx,
\end{equation*}
where the $\tilde{a}_{Rjk}$ are bounded independently of $R$, $C^{\infty}$, and supported in $|x|\geq R$.
By similar integration by parts on the other terms of \eqref{calcul1}, we get
\begin{equation}
\label{calcul2}
\frac{d}{dt}  \int \psi_R(x)\cdot \nabla u \partial_t u=\frac{N-2}{2}\int |\nabla u|^2dx-\frac{N-2}{2}\int |u|^{2^*}dx-\frac{N}{2} \int  (\partial_t u)^2dx +\widetilde{A}_R(u,\partial_t u),
\end{equation}
where $\widetilde{A}_R(u,\partial_t u)$ is a sum of integrals of the desired form \eqref{defAR}.

Furthermore,
\begin{align*}
\frac{N-1}{2}\frac{d}{dt} \int \varphi_R u\partial_t u\,dx&=\frac{N-1}{2}\int \varphi_R (\partial_t u)^2+ \varphi_R u\left(\Delta u +|u|^{\frac{N+2}{N-2}}u\right)dx\\
&=\frac{N-1}{2}\int \varphi_R (\partial_t u)^2-\varphi_R |\nabla u|^2 +\frac{1}{2} (\Delta \varphi_R) u^2+\varphi_R |u|^{2^*}dx.
\end{align*}
Noting that for $|x|\leq R$, $\varphi_R=1$ and $\Delta \varphi_R=0$, and that if $|x|\geq 2R$, $\Delta \varphi_R=0$, we get 
\begin{equation}
\label{calcul3}
\frac{N-1}{2}\frac{d}{dt} \int \varphi_R u\partial_t u\,dx=\frac{N-1}{2}\int-|\nabla u|^2+ |u|^{2^*}+(\partial_t u)^2dx +\widehat{A}_R(u,\partial_t u),
\end{equation}
where $\widehat{A}_R(u,\partial_tu)$ is again of the form \eqref{defAR}. 

Summing up \eqref{calcul2} and \eqref{calcul3}, we obtain
\begin{equation*}
g'_R(t)=-\frac 12 \int |\nabla u|^2dx+\frac 12 \int |u|^{2^*}dx-\frac{1}{2}\int (\partial_t u)^2dx+A_R(u,\partial_t u),
\end{equation*}
where $A_R$ is defined by \eqref{defAR} for some functions $a_R^{jk}$, $b_R^1$, $b_R^2$, $b_R^3$. To conclude the proof, note that $E(u,\partial_tu)=E(W,0)$ implies
$$ \frac 12 \int |u|^{2^*}=\frac{N}{2(N-2)}\int |\nabla u|^2dx+\frac{N}{2(N-2)}\int (\partial_t u)^2dx-\frac{N}{N-2}E(W,0),$$
and recall that $E(W,0)=\frac{1}{N}\|\nabla W\|^2_2$, which gives \eqref{calcul1}.
\end{proof}

\bibliographystyle{alpha} 
\bibliography{wavecrit5}

\end{document}